\RequirePackage{fix-cm}
\documentclass[twocolumn,final,numbook,10pt]{svjour3}          
\smartqed  
\usepackage{graphicx}
\usepackage{xspace}
\usepackage{amssymb,amsmath}
\usepackage{subfigure}
\usepackage{booktabs}
\usepackage{graphicx}
\newcommand{\sign}{\operatorname{sign}}
\newcommand{\R}{\mathbb R}

\newcommand{\N}{\mathbb N}

\newcommand{\J}{{\mathcal J}}
\newcommand{\V}{{\mathcal V}}
\def\H{{\mathcal H}}
\def\1{\mathbf{1}}

\newcommand{\argmin}{\arg\min}

\newcommand{\Div}{\operatorname{div}}

\newcommand{\ST}{\operatorname{ST}}
\newcommand{\Dom}{\operatorname{Dom}}
\newcommand{\TV}{\mathcal{R}}
\newcommand{\tv}{R}
\newcommand{\mcN}{\mathcal{N}}
\newcommand{\mcD}{\mathcal{H}}
\newcommand{\mcR}{\mathcal{R}}
\newcommand{\mcS}{\mathcal{S}}
\newcommand{\mcE}{\mathcal{E}}

\newcommand{\Js}{\mathcal{J}}
\newcommand{\B}{\mathcal{B}}
\newcommand{\Bt}{\B_\tau}

\newcommand{\diff}{\textnormal{d}}
\newcommand{\G}{{\mathcal G}}

\begin{document}

\title{Automated Parameter Selection for Total Variation Minimization in Image Restoration
}


\author{Andreas Langer        
}


\institute{A. Langer \at
              Institute for Applied Analysis and Numerical Simulation, University of Stuttgart, Pfaffenwaldring 57, 70569 Stuttgart, Germany \\
              \email{andreas.langer@mathematik.uni-stuttgart.de}           
}

\date{}

\maketitle

\begin{abstract}
Algorithms for automatically selecting a scalar or locally varying regularization parameter for total variation models with an $L^{\tau}$-data fidelity term, $\tau\in \{1,2\}$, are presented. The automated selection of the regularization parameter is based on the discrepancy principle, whereby in each iteration a total variation model has to be minimized. In the case of a locally varying parameter this amounts to solve a multi-scale total variation minimization problem. For solving the constituted multi-scale total variation model convergent first and second order methods are introduced and analyzed. Numerical experiments for image denoising and image deblurring show the efficiency, the competitiveness, and the performance of the proposed fully automated scalar and locally varying parameter selection algorithms.



\end{abstract}

\section{Introduction}
Observed images are often contaminated by noise and may be additionally distorted by some measurement device. Then the obtained data $g$ can be described as 
$$
g=\mcN(T\hat{u}),
$$ 
where $\hat{u}$ is the unknown original image, $T$ is a linear bounded operator modeling the image-formation device, and $\mcN$ represents noise. In this paper, we consider images which are contaminated either by white Gaussian noise or impulse noise. While for white Gaussian noise the degraded image $g$ is obtained as
$$
g = T \hat{u} + \eta,
$$ 
where the noise $\eta$ is oscillatory with zero mean and standard deviation $\sigma$, there are two main models for impulse noise, that are widely used in a variety of applications, namely salt-and-pepper noise and random-valued impulse noise. We assume that $T\hat{u}$ is in the dynamic range $[0,1]$, i.e., $0 \leq T\hat{u} \leq 1$, then in the presence of salt-and-pepper noise the observation $g$ is given by
\begin{equation}\label{sap}
g(x)=
\begin{cases}
0 & \text{ with probability } r_1\in[0,1),\\
1 & \text{ with probability } r_2 \in[0,1), \\
T \hat{u} (x) & \text{ with probability } 1-r_1-r_2,\\
\end{cases}
\end{equation}
with $1-r_1-r_2>0$. If the image is contaminated by random-valued impulse noise, then $g$ is described as
\begin{equation}\label{rv}
g(x)=
\begin{cases}
\rho & \text{ with probability } r \in[0,1),\\
T \hat{u} (x) & \text{ with probability } 1-r,\\
\end{cases}
\end{equation}
where $\rho$ is a uniformly distributed random variable in the image intensity range $[0,1]$.

The recovery of $\hat{u}$ from the given degraded image $g$ is an ill-posed inverse problem and thus regularization techniques are required to restore the unknown image \cite{EngHanNeu}. A good approximation of $\hat{u}$ may be obtained by solving a minimization problem of the type
\begin{equation}\label{Eq:general:functional}
\min_u \mcD (u;g) + \alpha \mcR(u),
\end{equation}
where $\mcD (.;g) $ represents a data fidelity term, which enforces the consistency between the recovered and measured image, $\mcR$ is an appropriate filter or regularization term, which prevents over-fitting, and $\alpha>0$ is a regularization parameter weighting the importance of the two terms. We aim at reconstructions in which edges and discontinuities are preserved. For this purpose we use the total variation as a regularization term, first proposed in \cite{ROF} for image denoising. 
Hence, here and in the remaining of the paper we choose $\mcR(u)=\int_\Omega |Du|$, where $\int_\Omega |Du|$ denotes the total variation of $u$ in $\Omega$; see \cite{AmbFusPal,Giu} for more details. However, we note that other regularization terms, such as the total generalized variation \cite{BreKunPoc}, the non-local total variation \cite{KinOshJon}, the Mumford-Shah regularizer \cite{MumSha}, or higher order regularizers (see e.g. \cite{PapSch} and references therein) might be used as well. 

\subsection{Choice of the fidelity term}

The choice of $\mcD$ typically depends on the type of noise contamination. For images corrupted by Gaussian noise a quadratic $L^2$-data fidelity term is typically chosen and has been successfully used; see for example \cite{Cha,ChaDar,ChaLio,ChaGolMul,ComWaj,DarSig2005,DarSig2006,DauTesVes,DobVog,GolOsh,Nes,OshBurGolXuYin,WeiBlaAub,ZhuCha}. In this approach, which we refer to as the $L^2$-TV model, the image $\hat u$ is recovered from the observed data $g$ by solving 
\begin{equation}\label{L2TVmodel}
\min_{u\in BV(\Omega)} \frac{1}{2} \|Tu - g \|_{L^2(\Omega)}^2 + \alpha \int_\Omega |Du|,
\end{equation}
where $BV(\Omega)$ denotes the space of functions with bounded variation, i.e., $u\in BV(\Omega)$ if and only if $u\in L^1(\Omega)$ and $\int_\Omega |Du| <\infty$. In the presence of impulse noise, e.g., salt-and-pepper noise or random-valued impulse noise, the above model usually does not yield a satisfactory restoration. In this context, a more successful approach, suggested in \cite{All,Nik2002,Nik2004}, uses a non-smooth $L^1$-data fidelity term instead of the $L^2$-data fidelity term in \eqref{L2TVmodel}, i.e., one considers
\begin{equation}\label{L1TVmodel}
\min_{u\in BV(\Omega)}  \|Tu - g \|_{L^1(\Omega)} + \alpha \int_\Omega |Du|,
\end{equation}
which we call the $L^1$-TV model. In this paper, we are interested in both models, i.e., the $L^2$-TV and the $L^1$-TV model, and condense them into 
\begin{equation}\label{LtauTVmodel}
\min_{u\in BV(\Omega)} \left\{\Js_\tau(u;g) := \mcD_\tau(u;g) + \alpha \int_\Omega |Du| \right\}
\end{equation}
to obtain a combined model for removing Gaussian or impulsive noise, where $ \mcD_\tau(u;g):=\frac{1}{\tau}\|Tu-g\|_{L^\tau(\Omega)}^\tau$ for $\tau=1,2$. Note, that instead of \eqref{LtauTVmodel} one can consider the equivalent problem
\begin{equation}\label{LtauTVmodel2}
\min_{u\in BV(\Omega)} \lambda \mcD_\tau(u;g) +  \int_\Omega |Du|,
\end{equation}
where $\lambda=\frac{1}{\alpha}>0$. Other and different fidelity terms have been considered in connection with other type of noise models, as Poisson noise \cite{LeChaAsa}, multiplicative noise \cite{AubAuj}, Rician noise \cite{GetTonVes}. 
 For images which are simultaneously contaminated by Gaussian and impulse noise \cite{CaiChaNik} a combined $L^1$-$L^2$-data fidelity term has been recently suggested and demonstrated to work satisfactory \cite{HinLan2013}. 
 However, in this paper, we concentrate on images degraded by only one type of noise, i.e., either Gaussian noise or one type of impulse noise, and perhaps additionally corrupted by some measurement device.

\subsection{Choice of the scalar regularization parameter}

For the reconstruction of such images the proper choice of $\alpha$ in \eqref{LtauTVmodel} and $\lambda$ in \eqref{LtauTVmodel2} is delicate; cf. Fig. \ref{fig_motivation}. In particular, large $\alpha$ and small $\lambda$, which lead to an over-smoothed reconstruction, not only remove noise but also eliminate details in images. On the other hand, small $\alpha$ and large $\lambda$ lead to solutions which fit the given data properly but therefore retain noise in homogeneous regions. Hence a good reconstruction can be obtained by choosing $\alpha$ and respectively $\lambda$ such that a good compromise of the aforementioned effects are made. There are several ways of how to select $\alpha$ in \eqref{LtauTVmodel} and equivalently $\lambda$ in \eqref{LtauTVmodel2}, such as manually by the trial-and-error method, the unbiased predictive risk estimator method (UPRE) \cite{Mal,LinWohGuo}, the Stein unbiased risk estimator method (SURE) \cite{Ste1981,DonJoh,BluLui} and its generalizations \cite{DelVaiFadPey,Eld,GirElaEld}, the generalized cross-validation method (GCV) \cite{GolHeaWah,LiaLiNg,LinWohGuo,RamLiuRosNieFes}, the L-curve method \cite{Han,HanOLe}, the discrepancy principle \cite{Mor}, and the variational Bayes' approach \cite{BabMolKat}. Further parameter selection methods for general inverse problems can be found for example in \cite{EngHanNeu,ForNauPer,TikArs,Vog}.

\graphicspath{{./graphics/}}
\begin{figure*}[htbp!]
\begin{center}
\hspace{0cm}
    \subfigure[noisy image]{\includegraphics[height=4.1cm]{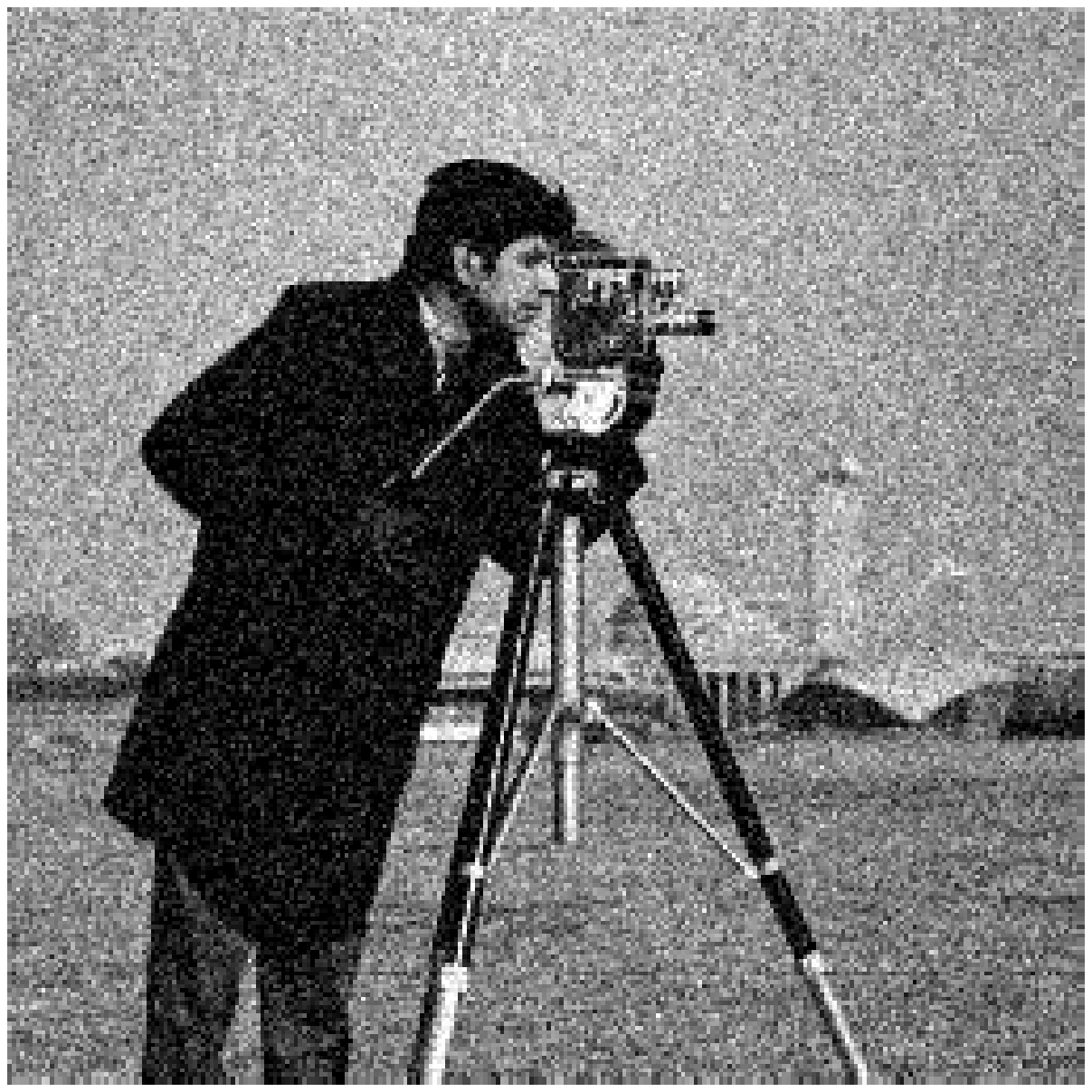}}
    \subfigure[over-smoothed reconstruction]{\includegraphics[height=4.1cm]{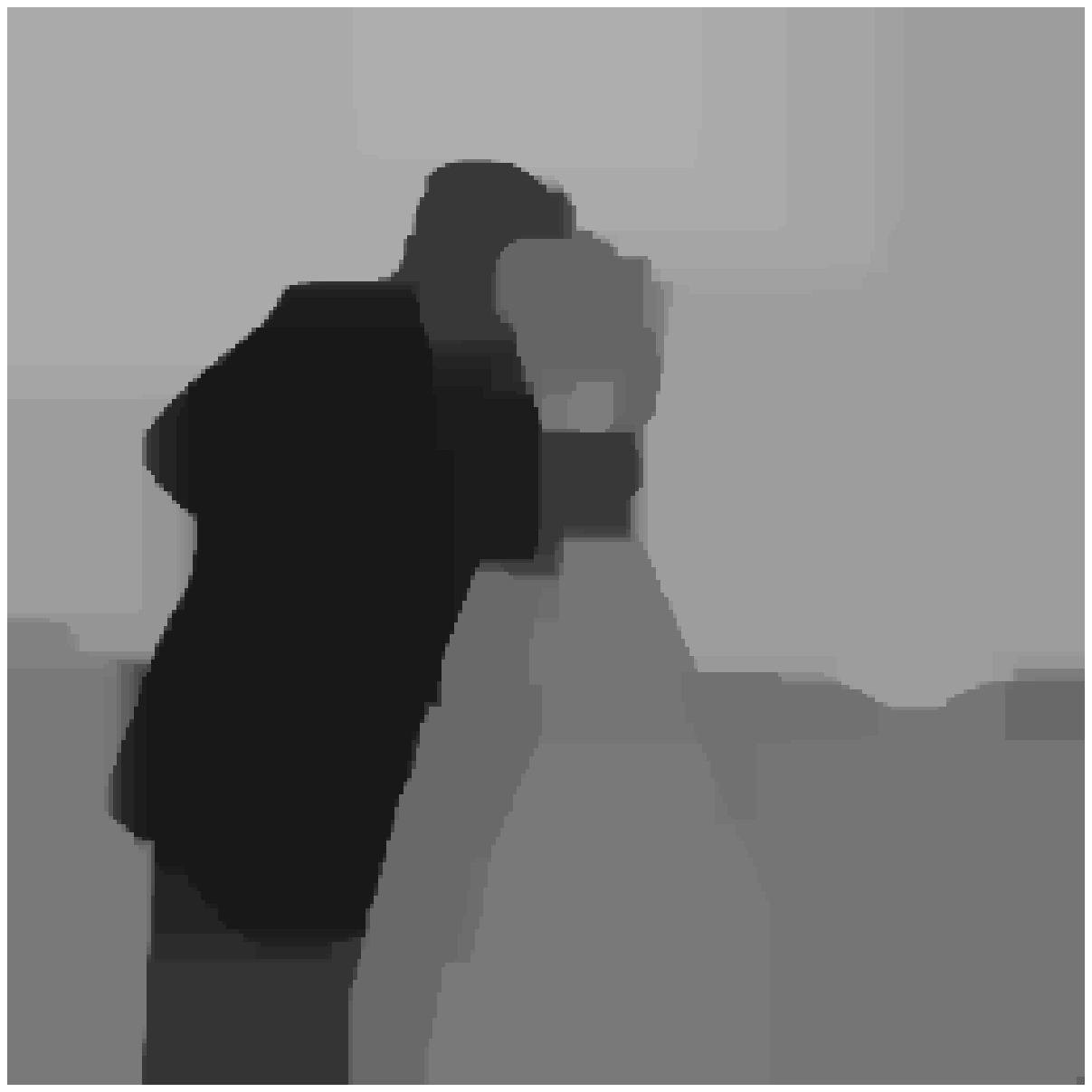}}
    \subfigure[over-fitted reconstruction]{\includegraphics[height=4.1cm]{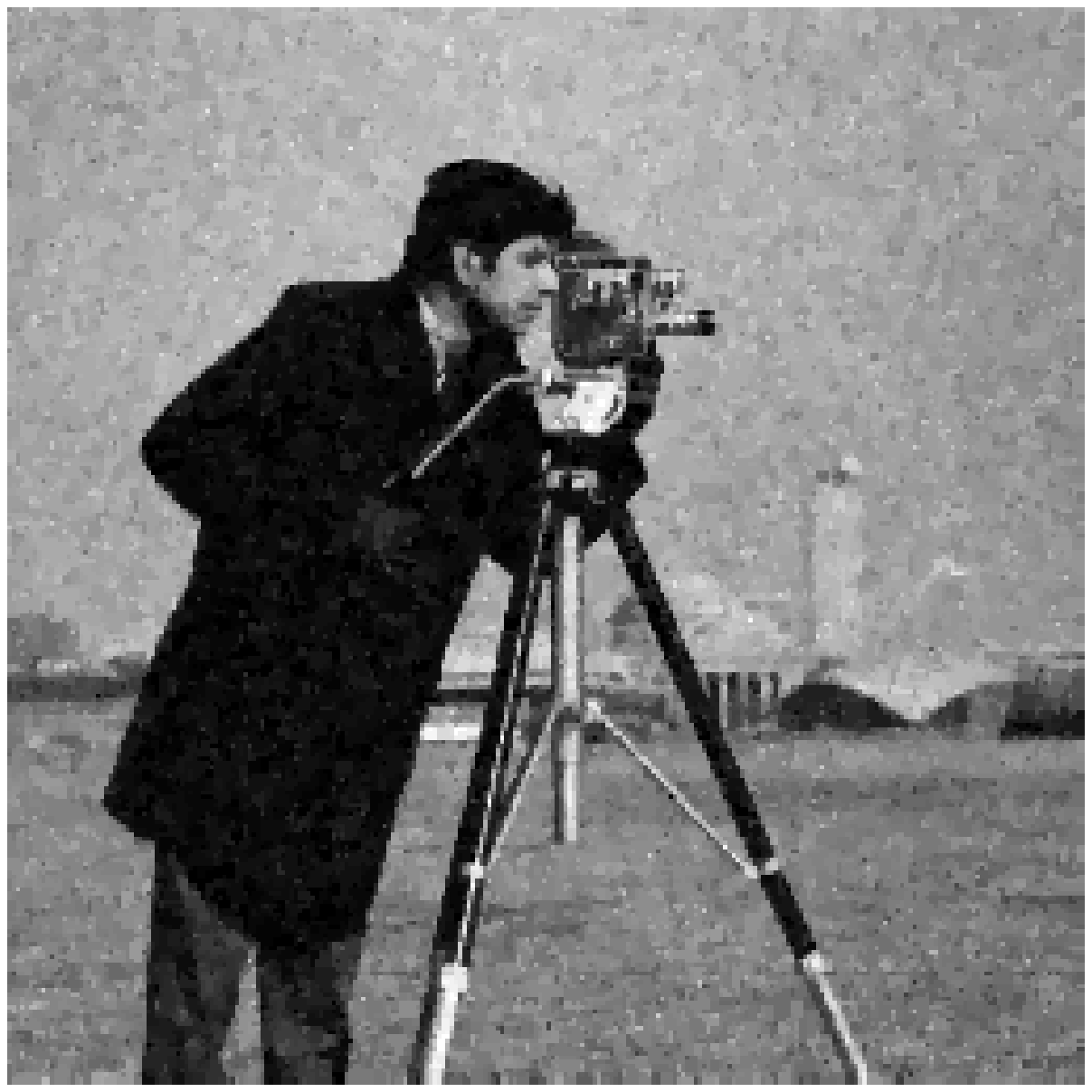}}
\end{center}    
\caption{\small \it Reconstruction of an image corrupted by Gaussian white noise with a ``relatively'' large parameter $\alpha$ in (b) and a ``relatively'' small parameter $\alpha$ in (c).} \label{fig_motivation}
\end{figure*}

Based on a training set of pairs $(g_k, \hat{u}_k)$, for $k=1, 2, \ldots , N \in \N$, where $g_k$ is the noisy observation and $\hat{u}_k$ represents the original image, for example in \cite{CalChuDeLSchVal,DeLSch,KunPoc} bilevel optimization approaches have been presented to compute suitable scalar regularization parameters of the corresponding image model. Since in our setting we do not have a training set given, these approaches are not applicable here.

Applying the discrepancy principle to estimate $\alpha$ in \eqref{LtauTVmodel} or $\lambda$ in \eqref{LtauTVmodel2}, the image restoration problem can be formulated as a constrained optimization problem of the form
\begin{equation}\label{conLtauTV}
\min_{u\in BV(\Omega)} \int_\Omega |Du| \quad \text{subject to (s.t.)} \quad \mcD_\tau(u;g)=\mathcal{B}_\tau
\end{equation}
where $\mathcal{B}_\tau:=\frac{\nu_\tau}{\tau} |\Omega|$ with $\nu_\tau>0$ being here a constant depending on the underlying noise, $\tau=1,2$, and $|\Omega|$ denoting the volume of $\Omega$; see Section \ref{Sec:APS} for more details. Note, that here we assume to know a-priori the noise level. In real applications this means that possibly in a first step a noise estimation has to be performed before the discrepancy principle may be used. However, in general it is easier to estimate the noise level than the regularization parameter \cite{Cha}. 

The constrained minimization problem \eqref{conLtauTV} is naturally linked to the unconstrained minimization problem \eqref{LtauTVmodel2} and accordingly to \eqref{LtauTVmodel}. In particular, there exists a constant $\lambda\geq 0$ such that the unconstrained problem \eqref{LtauTVmodel2} is equivalent to the constrained problem \eqref{conLtauTV} if $T$ does not annihilate constant functions, i.e., $T\in\mathcal{L}(L^2(\Omega))$ is such that $T\cdot 1 =1$; see Section \ref{AppendixA} for more details. Several methods based on the discrepancy principle and problem \eqref{conLtauTV} with $\tau=2$ have been proposed in the literature, see for example \cite{BloCha,Cha,HeChaZhaShi,WenCha} and references therein, while not so much attention has been given to the case $\tau=1$, see for example \cite{NgWeiYua,WeiBlaAub}. 

 
 \subsection{Spatially adaptive parameter}
Note, that a scalar regularization parameter might not be the best choice for every image restoration problem, since images usually have large homogeneous regions as well as parts with a lot of details. Actually it seems obvious that $\alpha$ should be small, or $\lambda$ should be large, in parts with small features in order to preserve the details. On the contrary $\alpha$ should be large, or $\lambda$ should be small, in homogeneous parts to remove noise considerable. With such a choice of a spatially varying weight we expect better reconstructions than with a globally constant parameter, as demonstrated for example in \cite{DonHinRin,HinRin}. 
 This motivated to consider multi-scale total variation models with spatially varying parameters initially suggested in \cite{RudOsh}. The multi-scale version of \eqref{LtauTVmodel} reads as
\begin{equation}\label{LtauTVmultimodel}
\min_{u\in BV(\Omega)} \mcD_\tau(u;g) + \int_\Omega \alpha(x) |Du|
\end{equation}
while for \eqref{LtauTVmodel2} one writes
\begin{equation}\label{LtauTVmultimodel2}
\min_{u\in BV(\Omega)} \frac{1}{\tau}\int_\Omega \lambda(x) |Tu-g|^\tau dx +  \int_\Omega |Du|,
\end{equation}
and in the sequel we refer to \eqref{LtauTVmultimodel} and \eqref{LtauTVmultimodel2} as the multi-scale $L^\tau$-TV model.

In \cite{StrCha1996} the influence of the scale of an image feature on the choice of $\alpha$ is studied and the obtained observations were later used in \cite{StrBloCha1997} to construct an updating scheme of $\alpha$. Based on \eqref{LtauTVmultimodel2} in \cite{BerCasRouSol} a piecewise constant function $\lambda$, where the pieces are defined by a partitioning of the image due to a pre-segmentation, is determined. In particular, for each segment a scalar $\lambda_i$, $i=1,\ldots,\#$pieces is computed by Uzawa's method \cite{Cia}.

Later it was noticed that stable choices of $\lambda$ respectively $\alpha$ should incorporate statistical properties of the noise. In this vein, in \cite{AlmBalCasHar,DonHinRin,GilSocZee} for the problem \eqref{LtauTVmultimodel2} automated update rules for $\lambda$ based on statistics of local constraints were proposed. In \cite{GilSocZee} a two level approach for variational denoising is considered, where in the first level noise and relevant texture are isolated in order to compute local constraints based on local variance estimation. In the second level a gradient descent method and an update formula for $\lambda(x)$ derived from the Euler-Lagrange equation is utilized. An adaptation of this approach to multiplicative noise can be found in \cite{LiNgShe}. For convolution type of problems in \cite{AlmBalCasHar} based on an estimate of the noise variance for each pixel an automatic updating scheme of $\lambda$ using Uzawa's method is created. This approach is improved in \cite{DonHinRin} by determining the fidelity weights due to the Gumbel statistic for the maximum of a finite number of random variables associated with localized image residuals and by incorporating hierarchical image decompositions, proposed in \cite{TadNezVes2004,TadNezVes2008}, to speed up the iterative parameter adjustment process. An adaptation of this approach to a total variation model with $L^1$ local constraints is studied in \cite{HinRin}. A different approach has been proposed in \cite{SutDelAuj} for image denoising only, where non-local means \cite{BuaColMor} are used to create a non-local data fidelity term. While in all these approaches the adjustment of $\lambda$ relies on the output of $T$ being a deteriorated image again, in \cite{HinLan2015} the method of \cite{DonHinRin} is adjusted to the situation where $T$ is an orthogonal wavelet transform or Fourier transform. Very recently also bilevel optimisation approaches are considered for computing spatially adaptive weights \cite{ChuDeLSch,HinRau,HinRauWuLan}.

\subsection{Contribution}

Our first contribution of this paper is to present a meth\-od which automatically computes the regularization parameter $\alpha$ in \eqref{LtauTVmodel} based on \eqref{conLtauTV} for $\tau=1$ as well as for $\tau=2$. Our approach is motivated by the parameter selection algorithm presented in \cite{Cha}, which was originally introduced for $L^2$-TV image denoising only, i.e., when $T=I$, where $I$ denotes the identity operator. In this setting the algorithm in \cite{Cha} is shown to converge to a parameter $\alpha^*$ such that the corresponding minimizer $u_{\alpha^*}$ of \eqref{LtauTVmodel} is also a solution of \eqref{conLtauTV}. The proof relies on the non-increase of the function $\alpha \mapsto \frac{\mcD_2(u_\alpha;g)}{\mathcal{B}_2}$. However, this important property does not hold for operators $T\not=I$ in general. Nevertheless, we generalize the algorithm from \cite{Cha} to problems of the type \eqref{LtauTVmodel} for $\tau=1,2$ and for general linear bounded operators $T$, e.g., $T$ might be a convolution type of operator. 
 Utilizing an appropriate update of $\alpha$, which is different than the one used in \cite{Cha}, we are able to show analytically and numerically that our approach indeed converges to the desired regularization parameter. Further, besides the general applicability of our proposed method it even possesses advantages for the case $\tau=2$ and $T=I$ over the algorithm from \cite{Cha} with respect to convergence. More precisely, in our numerics it turned out that our proposed method always needs less or at least the same number of iterations as the algorithm from \cite{Cha} till termination. 

Motivated by multi-scale total variation minimization, the second contribution of this paper is concerned with the automated selection of a suitable spatially varying $\alpha$ for the optimization problem in \eqref{LtauTVmultimodel} for $\tau=1,2$. Based on our considerations for an automatic scalar regularization parameter selection, we present algorithms where the adjustment of a locally varying $\alpha$ is fully automatic. Differently to the scalar case the adjustment of $\alpha$ is now based on local constraints, similarly as already considered for example in \cite{AlmBalCasHar,DonHinRin,HinRin}. However, our approach differs significantly from these previous works, where problem \eqref{LtauTVmultimodel2} is considered and Uzawa's method or an Uzawa-like method is utilized for the update of the spatially varying parameter. Note, that in Uzawa's method an additional parameter has to be introduced and chosen accordingly. We propose an update-scheme of $\alpha$ which does not need any additional parameter and hence is not similar to Uzawa's method. Moreover, differently to the approaches in \cite{DonHinRin,HinRin} where the initial regularization parameter $\lambda>0$ has to be set sufficiently small, in our approach any initial $\alpha>0$ is allowed. In this sense is our algorithm even more general than the ones presented in \cite{DonHinRin,HinRin}.

\subsection{Outline of the paper}

The remaining of the paper is organized as follows: In Section \ref{AppendixA} we revisit and discuss the connection between the constrained minimization problem \eqref{conLtauTV} and the unconstrained optimization problem \eqref{LtauTVmodel}.  
 Section \ref{Sec:APS} is devoted to the automated scalar parameter selection. In particular, we present our proposed method and analyze its convergence behavior. Based on local constraints we describe in Section \ref{Sec:locallyTV} our new locally adapted total variation algorithm in detail. Algorithms for performing total variation minimization for spatially varying $\alpha$ are presented in Section \ref{Sec:TVmin} where also their convergence properties are studied. To demonstrate the performance of the new algorithms we present in Section \ref{Sec:Numerics} numerical experiments for image denoising and image deblurring. Finally, in Section \ref{Sec:Conclusion} conclusions are drawn.

\section{Constrained versus unconstrained minimization problem}\label{AppendixA}
In this section we discuss the connection between the unconstrained minimization problem \eqref{LtauTVmodel} and the constrained optimization problem \eqref{conLtauTV}. For this purpose we introduce the following basic terminology. Let $\V$ be a locally convex space, $\V'$ its topological dual, and $\langle \cdot, \cdot \rangle$ the bilinear canonical pairing over $\V\times \V'$. The domain of a functional $\J : \V \to \R \cup\{+\infty\}$ is defined as the set
$$
\Dom(\J) := \{v\in \V \ : \ \J(v) < \infty\}.
$$
A functional $\J$ is called \emph{lower semicontinuous} (l.s.c) if for every weakly convergent subsequence $v^{(n)}\rightharpoonup \hat{v}$ we have 
$$
\liminf_{v^{(n)}\rightharpoonup \hat{v}} \J(v^{(n)}) \geq \J(\hat{v}).
$$
For a convex functional $\J : \V \to \R \cup\{+\infty\}$, we define the \emph{subdifferential} of $\J$ at $v\in \V$, as the set valued function $\partial \J(v) = \emptyset$ if $\J(v)=\infty$, and otherwise as
$$
\partial \J(v) = \{v^* \in \V' \ : \ \langle v^*, u-v\rangle + \J(v) \leq \J(u) \ \ \forall u\in \V \}.
$$
For any operator $T$ we denote by $T^*$ its adjoint and by $\mathcal{L}(L^2(\Omega))$ we denote the space of linear and continuous operators from $L^2(\Omega)$ to $L^2(\Omega)$. Moreover,  $g_\Omega$ describes the average value of the function $g\in L^1(\Omega)$ in $\Omega$ defined by $g_\Omega:=\frac{1}{|\Omega|} \int_\Omega g(x) \ \diff x$.

\begin{theorem}\label{Thm:ex}
Assume that $T\in \mathcal{L}(L^2(\Omega))$ does not annihilate constant functions, i.e., $T 1_\Omega\not=0$, where $1_\Omega(x) = 1$ for $x\in \Omega$. Then the problem 
\begin{equation}\label{cminPA}
\begin{split}
\min_{u\in BV(\Omega)}  \int_{\Omega} |D u|  \quad \text{ s.t. } \quad  \mcD_\tau(u;g) \leq \mathcal{B}_\tau 
\end{split}
\end{equation}
 has a solution for $\tau=1,2$.
\end{theorem}
\begin{proof}
For a proof we refer the reader to \cite{ChaLio} and \cite{HinRin}.
\qed
\end{proof}
Moreover, we have the following statement.
\begin{proposition}\label{Thm:equi1}
Assume that $T\in \mathcal{L}(L^2(\Omega))$ is such that $T\cdot1=1$ and $\nu_\tau |\Omega| \leq \|g-g_\Omega\|^\tau_{L^\tau(\Omega)}$. Then problem \eqref{cminPA} is equivalent to the constrained minimization problem \eqref{conLtauTV} for $\tau=1,2$.
\end{proposition}

\begin{proof}
For $\tau=2$ the statement is shown in \cite{ChaLio}. We state the proof for $\tau=1$ by noting it follows similar arguments as for $\tau=2$. Let $\tilde{u}$ be a solution of \eqref{cminPA}. Note, that there exists $u\in BV(\Omega)$ such that $\tilde{u}= u+g_\Omega$. We consider now the continuous function $f(s)=\|T(s u + g_\Omega)- g\|_{L^1(\Omega)}$ for $s\in[0,1]$. Note that $f(1)=\|T\tilde{u} - g\|_{L^1(\Omega)}\leq \nu_1|\Omega|$ and $f(0)=\|T g_\Omega - g\|_{L^1(\Omega)} = \|g - g_\Omega\|_{L^1(\Omega)}\geq \nu_1 |\Omega|$, since $T\cdot1=1$, and hence there exists some $s\in[0,1]$ such that $f(s)=\nu_1 |\Omega|$. Set $u'=su$ which satisfies $\|Tu' - g\|_{L^1(\Omega)}=\nu_1|\Omega|$ and 
$$
 \int_{\Omega} |D u'| = s  \int_{\Omega} |D u| \leq \liminf_{n\to \infty}  \int_{\Omega} |D u_{n}|,
$$
where $(u_n)_n$ is a minimizing sequence of \eqref{conLtauTV}. Hence $u'$ is a solution of \eqref{conLtauTV}. \qed
\end{proof}

Now we are able to argue the equivalence of the problems \eqref{LtauTVmodel} and \eqref{conLtauTV}.
\begin{theorem}\label{Thm:equi}
Let $T\in \mathcal{L}(L^2(\Omega))$ be such that $T\cdot 1 =1$ and $\nu_\tau |\Omega| \leq \|g-g_\Omega\|^\tau_{L^\tau(\Omega)}$. Then there exists $\alpha\geq 0$ such that the constrained minimization problem \eqref{conLtauTV} is equivalent to the unconstrained problem \eqref{LtauTVmodel}, i.e., $u$ is a solution of \eqref{conLtauTV} if and only if $u$ solves \eqref{LtauTVmodel}.
\end{theorem}

\begin{proof}
For $\tau=2$ the proof can be found in \cite[Prop. 2.1]{ChaLio}. By similar arguments one can show the statement for $\tau=1$, which we state here.

Set $ \mathcal{R}(u):= \int_{\Omega} |D u|$ and 
$$
\G(u)=\
\begin{cases}
+\infty  & \text{if } \|u-g\|_{L^1(\Omega)} >\nu_1 |\Omega|,\\
0         & \text{if } \|u-g\|_{L^1(\Omega)} \leq\nu_1 |\Omega|.
\end{cases}
$$
Notice, that $\mathcal{R}$ and $\G$ are convex l.s.c functions and problem \eqref{cminPA} is equivalent to $\min_u \mathcal{R}(u)+\G(Tu)$. We have $\Dom(\TV) = BV(\Omega) \cap L^2(\Omega)$ and $\Dom (\G)=\{u\in L^2(\Omega) : \G(u)<+\infty\}$. Since $g\in \overline{T\Dom(\TV)}$, there exists $\tilde u\in \Dom(\TV)$ with $\|T\tilde u - g\|_{L^1(\Omega)} \leq \nu_1|\Omega|/2$. As $T\in \mathcal{L}(L^2(\Omega))$ is continuous, $\G\circ T$ is continuous at $\tilde{u}$. Hence, by \cite[Prop. 5.6, p. 26]{EkeTem} we obtain
$$
\partial (\TV +\G\circ T)(u) = \partial \TV(u) + \partial(\G\circ T)(u) 
$$
for all $u$. Further, $\G$ is continuous at $T\tilde{u}$, and hence by \cite[Prop. 5.7, p. 27]{EkeTem} we have for all $u$,
$$
\partial(\G\circ T)(u) = T^*\partial \G(Tu)
$$
where
$\partial \G(u)=\{0\}$ if $\|u-g\|_{L^1(\Omega)}<\nu_1|\Omega|$ and $\partial \G(u)=\{\alpha \partial (\|u-g\|_{L^1(\Omega)}), \alpha\geq 0\}$ if $\|u-g\|_{L^1(\Omega)}=\nu_1|\Omega|$.

If $u$ is a solution of \eqref{cminPA} and hence of \eqref{conLtauTV}, then 
$$
0\in \partial (\TV +\G\circ T)(u) =  \partial \TV(u) + T^*\partial \G(Tu).
$$
Since any solution of \eqref{conLtauTV} satisfies $\|Tu-g\|_{L^1(\Omega)}=\nu_1 |\Omega|$, this shows that there exists an $\alpha\geq 0$ such that 
$$
0\in \partial \TV(u) + \alpha T^*\partial (\|Tu-g\|_{L^1(\Omega)}).
$$
Hence for this $\alpha\geq 0$, $u$ is a minimizer of the problem in \eqref{LtauTVmodel}. 

Conversely, a minimizer $u$ of \eqref{LtauTVmodel} with the above $\alpha$ is obviously a solution of \eqref{conLtauTV} with $\|Tu-g\|_{L^1(\Omega)}=\nu_1 |\Omega|$. This concludes the proof.\qed
\end{proof}


Note, that $\|T u_\alpha -g\|_{L^1(\Omega)}$ is (only) convex with respect to $Tu_\alpha$, and hence a minimizer of \eqref{L1TVmodel} is in general not unique even in the simple case when $T=I$, i.e., for two minimizers $u_\alpha^1$ and $u_\alpha^2$ in general we have $T u_\alpha^1 \not=T u_\alpha^2$. On the contrary, $\|T u_\alpha -g\|_{L^2(\Omega)}^2$ is strictly convex with respect to $Tu_\alpha$, i.e., for two minimizers $u_\alpha^1$ and $u_\alpha^2$ of \eqref{L2TVmodel} we have $T u_\alpha^1 =T u_\alpha^2$. Moreover, the function $\alpha \mapsto \mcD_1(u_\alpha;g)$  is in general not continuous \cite{ChaEse}, while $\alpha \mapsto \mcD_2(u_\alpha;g)$ indeed is continuous \cite{ChaLio}, where $u_\alpha$ is a respective minimizer of \eqref{LtauTVmodel}. Hence we have the following further properties:

\begin{lemma}\label{Lemma:Fnondecreasing}
Let $u_\alpha$ be a minimizer of \eqref{LtauTVmodel} 
 then $\alpha \mapsto \mcD_\tau(u_\alpha;g)$ is non-decreasing for $\tau =1,2$. Moreover, $\alpha \mapsto \mcD_2(u_\alpha;g)$ maps $\R^+$ onto $[0,\|g- g_\Omega \|_{L^2(\Omega)}^2]$. 
\end{lemma}
\begin{proof} For a proof see \cite{ChaLio,ChaEse}.\qed
\end{proof}

\begin{proposition}\label{Prop:alpha:estimate}
If $u_{\alpha_i}$ is a minimizer of 
$$
\mcE(u,\alpha_i) := \|u-g\|_{L^2(\Omega)}^2 +\alpha_i  \int_{\Omega} |D u|
$$
 for $i=1,2$, then we have
$$
\| u_{\alpha_1} - u_{\alpha_2} \|_{L^2(\Omega)} \leq C \left\|g- g_\Omega \right\|_{L^2(\Omega)}.
$$
with $C:=\min\left\{2\left|\frac{\alpha_2 - \alpha_1}{\alpha_2 + \alpha_1}\right|, \left|\frac{\alpha_2 - \alpha_1}{\alpha_2 + \alpha_1}\right|^{\frac{1}{2}}\right\}$
\end{proposition}

\begin{proof}
By \cite[Lemma 10.2]{Bar} we have
\begin{equation*}
\begin{split}
&\frac{1}{\alpha_1} \| u_{\alpha_1} - u_{\alpha_2} \|_{L^2(\Omega)}^2 \leq\frac{1}{\alpha_1} \left( \mcE(u_{\alpha_2},\alpha_1) - \mcE(u_{\alpha_1},\alpha_1)\right) \\
&\frac{1}{\alpha_2} \| u_{\alpha_1} - u_{\alpha_2} \|_{L^2(\Omega)}^2 \leq \frac{1}{\alpha_2} \left(\mcE(u_{\alpha_1},\alpha_2) - \mcE(u_{\alpha_2},\alpha_2)\right).
\end{split}
\end{equation*}
Summing up these two inequalities yields
\begin{equation*}
\begin{split}
&\left(\frac{1}{\alpha_1} +  \frac{1}{\alpha_2} \right) \| u_{\alpha_1} - u_{\alpha_2} \|_{L^2(\Omega)}^2 \\
&\leq \left(\frac{1}{\alpha_1}  - \frac{1}{\alpha_2}\right)\left(\|u_{\alpha_2}-g\|_{L^2(\Omega)}^2 -  \|u_{\alpha_1}-g\|_{L^2(\Omega)}^2\right)\
\end{split}
\end{equation*}
which implies
\begin{equation}\label{Ineq:1}
\begin{split}
 \| u_{\alpha_1} &- u_{\alpha_2} \|_{L^2(\Omega)}^2 \\
&\leq \frac{\alpha_2 - \alpha_1}{\alpha_1+\alpha_2} \left(\|u_{\alpha_2}-g\|_{L^2(\Omega)}^2 -  \|u_{\alpha_1}-g\|_{L^2(\Omega)}^2\right).
\end{split}
\end{equation}
By the non-decrease and boundedness of the function $\alpha \mapsto \mcD_2(u_\alpha;g)$, see Lemma \ref{Lemma:Fnondecreasing}, it follows
\begin{equation}\label{Ineq:2}
\begin{split}
 \| u_{\alpha_1} - u_{\alpha_2} \|_{L^2(\Omega)}^2 \leq \left|\frac{\alpha_2 - \alpha_1}{\alpha_1+\alpha_2} \right| \left\|g - g_\Omega \right\|_{L^2(\Omega)}^2 .
\end{split}
\end{equation}
On the other hand inequality \eqref{Ineq:1} implies
\begin{equation*}
\begin{split}
 \| u_{\alpha_1} &- u_{\alpha_2} \|_{L^2(\Omega)} \\
&\leq \left|\frac{\alpha_2 - \alpha_1}{\alpha_1+\alpha_2}\right| \left(\|u_{\alpha_2}-g\|_{L^2(\Omega)} + \|u_{\alpha_1}-g\|_{L^2(\Omega)}\right),
\end{split}
\end{equation*}
where we used the binomial formula $a^2 -b^2 = (a+b)(a-b)$ for $a,b\in\R$ and the triangle inequality. Using Lemma \ref{Lemma:Fnondecreasing} yields
\begin{equation}\label{Ineq:3}
\begin{split}
 \| u_{\alpha_1} - u_{\alpha_2} \|_{L^2(\Omega)} \leq 2 \left|\frac{\alpha_2 - \alpha_1}{\alpha_1+\alpha_2}\right| \left\| g - g_\Omega \right\|_{L^2(\Omega)}.
\end{split}
\end{equation}
The assertion follows then from \eqref{Ineq:2} and \eqref{Ineq:3}. \qed
\end{proof}

\begin{remark}
Without loss of generality let $\alpha_2 \geq \alpha_1$ in Proposition \ref{Prop:alpha:estimate}, then we easily check that
$$
C=\begin{cases}
\left(\frac{\alpha_2-\alpha_1}{\alpha_2+\alpha_1}\right)^{\frac{1}{2}}  & \text{if } \alpha_2  > \frac{5}{3}\alpha_1, \\
\frac{1}{2} & \text{if } \alpha_2 = \frac{5}{3}\alpha_1, \\
2 \frac{\alpha_2-\alpha_1}{\alpha_2+\alpha_1} & \text{otherwise}.
\end{cases}
$$
\end{remark}

\section{Automated scalar parameter selection} \label{Sec:APS}

In order to find a suitable regularization parameter $\alpha>0$ of the minimization problem \eqref{LtauTVmodel} 
 we consider the corresponding constrained optimization problem \eqref{conLtauTV}.
Throughout the paper we assume that $T$ does not annihilate constant function, which guarantees the existence of a minimizer of the considered optimization problems; see Section \ref{AppendixA}.
We recall, that in the constraint of \eqref{conLtauTV} the value $\Bt$ is defined as $\Bt=\frac{\nu_\tau}{\tau}|\Omega|$, where $\nu_\tau \in \R$ is a statistical value depending on the underlying noise and possibly on the original image.

\subsection{Statistical characterization of the noise}\label{Sec:CharNoise}

Let us characterize the noise corrupting the image in more details by making similar considerations as in \cite[Section 2]{HinRin}. Note, that at any point $x\in\Omega$ the contaminated image $g(x) = \mcN(T\hat{u})(x)$ is a stochastic observation, which depends on the underlying noise. Two important measures to characterize noise are the expected absolute value and the variance, which we denote by $\nu_1$ and $\nu_2$ respectively. For images contaminated by Gaussian white noise with standard deviation $\sigma$, we typically set $\tau=2$ and $\nu_2=\sigma^2$. If the image is instead corrupted by impulse noise, then we set $\tau=1$ and we have to choose $\nu_1$ properly. In particular, for salt-and-pepper noise $\nu_1\in [\min\{r_1,r_2\},\max\{r_1,r_2\}]$, while for random-valued impulse noise $\nu_1$ should be a value in the interval $[\frac{r}{4},\frac{r}{2}]$, where we used that for any point $x\in \Omega$ we have $T\hat{u}(x)\in [0,1]$; cf. \cite{HinRin}. Here $\nu_1$ seems to be fixed, while actually $\nu_1$ depends on the true (unknown) image $\hat{u}$. In particular, for salt-and-pepper noise the expected absolute value is given by
\begin{equation}\label{Eq:nu1:sp}
\nu_1(\hat{u}) := r_2-(r_2-r_1) \frac{1}{|\Omega|}\int_{ \Omega}(T\hat{u})(x) \ \diff x
\end{equation}
and for random-valued impulse noise we have
\begin{equation}\label{Eq:nu1:rv}
\nu_1(\hat{u}) := \frac{1}{|\Omega|}\int_{\Omega}r\left((T\hat{u})(x)^2 - (T \hat{u})(x) +\frac{1}{2}\right) \ \diff x.
\end{equation}
However, instead of considering the constraint $\H_\tau(u;g) = \Bt(u)$ in \eqref{conLtauTV}, which results in a quite nonlinear problem, in our numerics we choose a reference image and compute an approximate value $\Bt$. Since our proposed algorithms are of iterative nature (see APS- and pAPS-algorithm below), it makes sense to choose the current approximation as the reference image, i.e., the reference image changes during the iterations. 
Note, that for salt-and-pepper noise with $r_1=r_2$ the expected absolute value becomes independent of $\hat{u}$ and hence $\nu_1=r_1$. In case of Gaussian noise $\nu_\tau$ and $\Bt$ are independent of $\hat{u}$ too. Nevertheless, in order to keep the paper  concise, in the sequel instead of $\nu_\tau$ and $\Bt$ we often write $\nu_\tau(\tilde{u})$ and $\Bt(\tilde{u})$, where $\tilde{u}$ represents a reference image approximating $\hat{u}$, even if the values may actually be independent from the image. 

\subsection{Automated parameter selection strategy}
In order to determine a suitable regularization parameter $\alpha$ in \cite{Cha} an algorithm for solving the constrained minimization problem \eqref{conLtauTV} for $T=I$ and $\tau=2$ is proposed, i.e., in the presence of Gaussian noise with zero mean and standard deviation $\sigma$. This algorithm relies on the fact that $\alpha \mapsto \H_2(u_\alpha;g)$ is non-decreasing, which leads to the following iterative procedure.

\noindent
\fbox{
\begin{minipage}{7.9cm}
\textbf{Chambolle's parameter selection (CPS):} Choose $\alpha_{0}>0$ and set $n:=0$.
\begin{itemize}
\item[1)] Compute $$u_{\alpha_n}\in \argmin\limits_{u \in BV(\Omega)} \|u - g\|_{L^2(\Omega)}^2 + 2\alpha_n\int_\Omega |Du|$$
\item[2)] Update $\alpha_{n+1}:=\frac{\sigma \sqrt{|\Omega|}}{\|u_{\alpha_n} - g\|_{L^2(\Omega)}}\alpha_n$.
\item[3)] Stop or set $n:=n+1$ and return to step 1).
\end{itemize}

\end{minipage}
}\\

For the minimization of the optimization problem in step 1) in \cite{Cha} a method based on the dual formulation of the total variation is used. However, we note that any other algorithm for total variation minimization might be used for solving this minimization problem. 
The CPS-algorithm generates a sequence $(u_{\alpha_n})_n$ such that for $n\to \infty$, $\|u_{\alpha_n} - g\|_{L^2(\Omega)} \to \sigma \sqrt{|\Omega|}$ and $u_{\alpha_n}$ converges to the unique solution of \eqref{conLtauTV} with $T=I$ and $\tau=2$ \cite{Cha}. The proof relies on the fact that the function  $\alpha \to \frac{\|u_\alpha - g\|_{L^2(\Omega)}}{\alpha}$ is non-increasing. Note, that this property does not hold in general for operators $T\not=I$.


\subsubsection{The p-adaptive algorithm}

We generalize now the CPS-algorithm to optimization problems of the type \eqref{conLtauTV} for $\tau=1,2$ and for general operators $T$. In order to keep or obtain appropriate convergence properties, we need the following two conditions to be satisfied. Firstly, the function $\alpha \mapsto \H_\tau(u_\alpha;g)$ has to be monotonic, which is the case due to  Lemma \ref{Lemma:Fnondecreasing}. Secondly, in each iteration $n$ the parameter $\alpha_n$ has to be updated such that $(\H_\tau(u_{\alpha_n};g))_n$ is monotonic and bounded by $(\B_\tau(u_{\alpha_n}))_n$. More precisely, if $\H_\tau(u_{\alpha_0};g)\leq \Bt(u_{\alpha_0})$ then there has to exist an $\alpha_{n+1}\geq \alpha_n$ such that $ \H_\tau(u_{\alpha_{n+1}};g)\leq \Bt(u_{\alpha_{n+1}})$, while if $\H_\tau(u_{\alpha_0};g)> \Bt(u_{\alpha_0})$ then there has to exist an $\alpha_{n+1}\leq \alpha_n$ such that $\H_\tau(u_{\alpha_{n+1}};g)\geq \Bt(u_{\alpha_{n+1}})$. This holds true by setting in every iteration
\begin{equation}\label{alpha:up1}
\alpha_{n+1}:=\left( \frac{\Bt(u_{\alpha_n})}{\H_\tau(u_{\alpha_n};g)}\right)^p \alpha_n
\end{equation}
together with an appropriate choice of $p\geq 0$. In particular, there exists always a $p\geq 0$ such that this condition is satisfied.
\begin{proposition}
Assume $\|g-g_\Omega \|_{L^\tau(\Omega)}^\tau \geq \nu_\tau |\Omega|$ and $\alpha_{n+1}$ is defined as in \eqref{alpha:up1}.
\begin{itemize}
\item [(i)] If $\alpha_n>0$ such that $\H_\tau(u_{\alpha_n};g)= \Bt(u_{\alpha_{n}})$, then for all $p\in\R$ we have that $\H_\tau(u_{\alpha_{n+1}};g)\leq \Bt(u_{\alpha_{n+1}})$.
\item[(ii)] If $\alpha_n>0$ such that $0<\H_\tau(u_{\alpha_n};g)<\Bt(u_{\alpha_{n}})$, then there exist $p\geq0$ with $\H_\tau(u_{\alpha_{n+1}};g)\leq \Bt(u_{\alpha_{n+1}})$.
\item[(iii)] If $\alpha_n>0$ such that $\H_\tau(u_{\alpha_n};g)>\Bt(u_{\alpha_{n}})$, then there exist $p\geq0$ with $\H_\tau(u_{\alpha_{n+1}};g)\geq \Bt(u_{\alpha_{n+1}})$.
\end{itemize}
\end{proposition}

\begin{proof}
The assertion immediately follows by noting that for $p=0$ we have $\alpha_{n+1}=\alpha_n$.\qed
\end{proof}

 Taking these considerations into account, a generalization of the CPS-algorithm can be formulated as the following $p$-adaptive automated parameter selection algorithm:

\noindent
\fbox{
\begin{minipage}{7.9cm}
\textbf{pAPS-algorithm:} Choose $\alpha_{0}>0$, $p:=p_0>0$, and set $n:=0$.
\begin{itemize}
\item[1)]  Compute $u_{\alpha_{n}}\in \argmin_{u\in BV(\Omega)} \Js_\tau(u;g)$ 
\item[2)] Update $\alpha_{n+1}:=\left( \frac{\Bt(u_{\alpha_n})}{\H_\tau(u_{\alpha_n};g)}\right)^p \alpha_n$ if $\H_\tau(u_{\alpha_n};g)>0$ and continue with step 3). Otherwise increase $\alpha_n$, e.g., $\alpha_n :=10\alpha_n$, and go to step 1).
\item[3)] Compute $u_{\alpha_{n+1}}\in \argmin_{u\in BV(\Omega)} \Js_\tau(u;g) $
\item[4)] \begin{itemize}
				\item[a)] if $ \H_\tau(u_{\alpha_0};g)\leq \Bt(u_{\alpha_0})$
				\begin{itemize}
				\item[(i)] if $ \H_\tau(u_{\alpha_{n+1}};g)\leq \Bt(u_{\alpha_{n+1}})$ go to step 5)
				
				\item[(ii)] if $ \H_\tau(u_{\alpha_{n+1}};g) > \Bt(u_{\alpha_{n+1}})$, decrease $p$, e.g., set $p:=p/2$, and go to step 2)
				\end{itemize}

				\item[b)] if $\H_\tau(u_{\alpha_0};g)> \Bt(u_{\alpha_0})$
				\begin{itemize}
				\item[(i)] if $\H_\tau(u_{\alpha_{n+1}};g)\geq \Bt(u_{\alpha_{n+1}})$ go to step 5)
				
				\item[(ii)] if $\H_\tau(u_{\alpha_{n+1}};g)< \Bt(u_{\alpha_{n+1}})$, decrease $p$, e.g., set $p:=p/2$, and go to step 2)
				\end{itemize}						
				
			  \end{itemize} 
\item[5)]  Stop or set $n:=n+1$ and return to step 2).
\end{itemize}

\end{minipage}
}\\

Note, that due to the dependency of $\alpha_{n+1}$ on $p$ a proper $p$ cannot be explicitly computed, but only iteratively, as in the pAPS-algorithm.

The initial $p_0> 0$ can be chosen arbitrarily. However, we suggest to choose it sufficiently large in order to keep the number of iterations small. In particular in our numerical experiments in Section \ref{Sec:Numerics} we set $p_0=32$, which seems large enough to us.

\begin{proposition}\label{Prop:mono:conv:qAPS}
The pAPS-algorithm generates monotone sequences $(\alpha_n)_n$ and $\left(\H_\tau(u_{\alpha_n};g)\right)_n$ such that \newline $\left(\H_\tau(u_{\alpha_n};g)\right)_n$ is bounded. 
Moreover, if $ \H_\tau(u_{\alpha_0};g) > \Bt(u_{\alpha_0})$ or $\Bt(u_\alpha)\leq \frac{1}{\tau}\|g- g_\Omega \|_\tau^\tau$ for all $\alpha>0$, then $(\alpha_n)_n$ is also bounded.
\end{proposition}


\begin{proof}
If $\H_\tau(u_{\alpha_0};g) > \Bt(u_{\alpha_0})$, then by induction and Lemma \ref{Lemma:Fnondecreasing} one shows that $0< \alpha_{n+1} \leq \alpha_n$ and $0 \leq \H_\tau(u_{\alpha_{n+1}};g) \leq \H_\tau(u_{\alpha_n};g)$ for all $n\in \N$. Consequently $(\alpha_n)_n$ and $(\H_\tau(u_{\alpha_n};g))_n$ are monotonically decreasing and bounded. 

If $\H_\tau(u_{\alpha_0};g)\leq \Bt(u_{\alpha_0})$, due to Lemma \ref{Lemma:Fnondecreasing} we have that $0<\alpha_n\leq \alpha_{n+1}$ and $\H_\tau(u_{\alpha_n};g)\leq \H_\tau(u_{\alpha_{n+1}};g)$ for all $n\in \N$ and hence $(\alpha_n)_n$ and $(\H_\tau(u_{\alpha_n};g))_n$ are monotonically increasing. Since there exists $\Bt^*>0$ such that  $\H_\tau(u_{\alpha_n};g) \leq \Bt(u_{\alpha_n})\leq \Bt^*$ for all $n\in \N$, see Section \ref{Sec:CharNoise}, $(\H_\tau(u_{\alpha_n};g))_n$ is also bounded. If we additionally assume that $\Bt(u_\alpha)\leq \frac{1}{\tau}\|g-g_\Omega \|_\tau^\tau$ for all $\alpha>0$ and we set $\Bt^*:=\max_\alpha \Bt(u_\alpha)$, then Theorem \ref{Thm:equi} ensures the existence of an $\alpha^*\geq 0$ such that $\H_\tau(u_{\alpha^*};g) = \Bt^*$. By Lemma \ref{Lemma:Fnondecreasing} it follows that $\alpha_{n}\leq \alpha^*$ for all $n\in\N$, since $\H_\tau(u_{\alpha_{n}};g) \leq \Bt(u_{\alpha_n})\leq \Bt^*$. Hence, $(\alpha_n)_n$ is bounded, which finishes the proof.\qed
%
%
\end{proof}

Since any monotone and bounded sequence converges to a finite limit, also $(\alpha_n)_n$ converges to a finite value if one of the assumptions in Proposition \ref{Prop:mono:conv:qAPS} holds. For constant $\Bt$ we are even able to argue the convergence of the pAPS-algorithm to a solution of the constrained minimization problem \eqref{conLtauTV}.

\begin{theorem}\label{Theorem:conv2}
Assume that $\Bt(u)\equiv \Bt$ is a constant independent of $u$ and $\| g- g_\Omega\|_\tau^\tau \geq \nu_\tau |\Omega|$. Then the pAPS-algorithm generates a sequence $(\alpha_n)_n$ such that $\lim_{n\to \infty} \alpha_n=\bar{\alpha} >0$ with $\H_\tau(u_{\bar \alpha};g)=\lim_{n\to\infty}\H_\tau(u_{\alpha_n};g)$ $=$ $\Bt$ and $u_{\alpha_n} \to u_{\bar{\alpha}}\in \argmin_{u\in X} \Js_\tau(u,\bar\alpha)$ for $n\to\infty$.
\end{theorem}
\begin{proof}
Let us start with assuming that $\H_\tau(u_{\alpha_0};g)\leq B_\tau$. 
By induction, we show that $\alpha_n \leq \alpha_{n+1}$ and $\H_\tau(u_{\alpha_n};g)$ $\leq$ $ \H_\tau(u_{\alpha_{n+1}};g)\leq \Bt$. In particular, if $\H_\tau(u_{\alpha_n};g) \leq \Bt$ then $\alpha_{n+1}=\left(\frac{\Bt}{\H_\tau(u_{\alpha_n};g)} \right)^p\alpha_n > \alpha_n$, where $p> 0$ such that $\H_\tau(u_{\alpha_{n+1}};g)\leq \Bt$; cf. pAPS-algorithm. 
 Then by Lemma \ref{Lemma:Fnondecreasing} it follows that 
 $$
 \H_\tau(u_{\alpha_n};g)\leq \H_\tau(u_{\alpha_{n+1}};g)\leq \Bt.
 $$ 
Note, that there exists an $\alpha^*>0$ with $\H_\tau(u_{\alpha^*};g)=B_\tau$, see Theorem \ref{Thm:equi}, such that for any $\alpha\geq\alpha^*$, $\H_\tau(u_{\alpha};g) \geq \Bt$; cf. Lemma \ref{Lemma:Fnondecreasing}. If $\alpha_n\geq\alpha^*$, then $\H_\tau(u_{\alpha_n};g)\geq B_\tau$. Hence $\H_\tau(u_{\alpha_n};g) = B_\tau$ and $\alpha_{n+1}=\alpha_n$. Thus we deduce that the sequences $(\H_\tau(u_{\alpha_n};g))_n$ and $(\alpha_n)_n$ are non-decreasing and bounded. Consequently, there exists an $\bar\alpha$ such that $\lim_{n\to\infty}\alpha_n = \bar{\alpha}$ with $\H_\tau(u_{\bar{\alpha}};g) = B_\tau$. Let $\bar{\H}$ $=$ $\lim_{n\to\infty} \H_\tau(u_{\alpha_n};g)$, then $\bar{\H}=\H_\tau(u_{\bar \alpha};g)=B_\tau$. By the optimality of $u_{\alpha_n}$ we have that $0\in \partial \Js_\tau(u_{\alpha_n},\alpha_n) = \partial \H_\tau(u_{\alpha_n};g) + \alpha_n \partial \int_\Omega |D u_{\alpha_n}|$; see \cite[Prop. 5.6 + Eq. (5.21), p.26]{EkeTem}. Consequently there exist $v_{\alpha_n}\in\partial \int_\Omega |D u_{\alpha_n}|$ such that $-\alpha_n v_{\alpha_n}\in \partial \H_\tau(u_{\alpha_n};g) $ with $\lim_{n\to\infty} v_{\alpha_n}=v_{\bar{\alpha}}$. By \cite[Thm. 24.4, p. 233]{Roc} we obtain that $-\bar\alpha v_{\bar\alpha}\in \partial \H_\tau(u_{\bar\alpha};g) $ with $v_{\bar\alpha}\in\partial \int_\Omega |D u_{\bar \alpha}|$ and hence $0\in \partial \Js_\tau(u_{\bar\alpha},\bar\alpha)$ for $n\to \infty$.
 
 If $\H_\tau(u_{\alpha_0};g)>B_\tau$, then as above we can show by induction that $\alpha_n \geq \alpha_{n+1}$ and $\H_\tau(u_{\alpha_n};g)\geq \H_\tau(u_{\alpha_{n+1}};g)\geq B_\tau$. Thus we deduce that $(\H_\tau(u_{\alpha_n};g))_n$ and $(\alpha_n)_n$ are non-increasing and bounded. Note, that there exists an $\alpha^*>0$ with $\H_\tau(u_{\alpha^*};g)=B_\tau$ such that for any $\alpha\leq\alpha^*$, $\H_\tau(u_{\alpha};g)\leq B_\tau$. Hence if $\alpha_n\leq\alpha^*$, then $\H_\tau(u_{\alpha_n};g)\leq B_\tau$. This implies, that $\H_\tau(u_{\alpha_n};g) = B_\tau$ and $\alpha_{n+1}=\alpha_n$. The rest of the proof is identical to above.\qed

\end{proof}

\begin{remark}
The adaptive choice of the value $p$ in the pAPS-algorithm is fundamental for proving convergence in Theorem \ref{Theorem:conv2}. In particular, the value $p$ is chosen in dependency of $\alpha$, i.e., actually $p=p(\alpha)$, such that in the case of a constant $\Bt$ the function $\alpha \mapsto \frac{\H_\tau(u_\alpha;g)^{p(\alpha)}}{\alpha}$ is non-increasing; cf. Fig. \ref{fig_phantom1_decr}(a).
\end{remark}

\subsubsection{The non p-adaptive case}

A special case of the pAPS-algorithm accrues when the value $p$ is not adapted in each iteration but set fixed. For the case $p=1$ (fixed) we obtain the following automated parameter selection algorithm.

\noindent
\fbox{
\begin{minipage}{7.9cm}
\textbf{APS-algorithm:} Choose $\alpha_{0}>0$ and set $n:=0$.
\begin{itemize}
\item[1)] Compute $u_{\alpha_n}\in \argmin_{u\in BV(\Omega)} \Js_\tau(u,\alpha_n)$
\item[2)] Update $\alpha_{n+1}:=\frac{B_\tau(u_{\alpha_n})}{\H_\tau(u_{\alpha_n};g)}\alpha_n$ if $\H_\tau(u_{\alpha_n};g)>0$ and continue with step 3). Otherwise increase $\alpha_n$, e.g., $\alpha_n :=10\alpha_n$, and go to step 1).
\item[3)] Stop or set $n:=n+1$ and return to step 1).
\end{itemize}

\end{minipage}
}\\

Even in this case, although under certain assumptions, we can immediately argue the convergence of this algorithm.

\begin{theorem}\label{Theorem:conv}
For  $\alpha>0$ let $u_\alpha$ be a minimizer of $\Js_\tau(u,\alpha)$. Assume that $\B_\tau(u)\equiv \B_\tau$ is a constant independent of $u$, the function $\alpha \mapsto \frac{\H_\tau(u_\alpha;g)}{\alpha}$ is non-increasing, and $\|g- g_\omega\|_\tau^\tau \geq \nu_\tau |\Omega|$. Then the APS-algorithm generates a sequence $(\alpha_n)_n \subset \R^+$ such that $\lim_{n\to \infty} \alpha_n = \bar{\alpha}>0$, $\lim_{n\to\infty}\H_\tau(u_{\alpha_n};g)=B_\tau$ and $u_{\alpha_n}$ converges to $u_{\bar{\alpha}}\in \argmin_{u\in BV(\Omega)} \Js(u,\bar{\alpha})$ for $n\to\infty$.
\end{theorem}
\begin{proof}
We only consider the case when $\H_\tau(u_{\alpha_0};g)\leq \Bt$ by noting that the case $\H_\tau(u_{\alpha_0};g)> \Bt$ can be shown analogous.
By induction, we can show that $\alpha_n\leq \alpha_{n+1}$ and $\H_\tau(u_{\alpha_n};g) \leq \H_\tau(u_{\alpha_{n+1}};g)\leq \B_\tau$. More precisely, if $\H_\tau(u_{\alpha_n};g)\leq \B_\tau$ then $\alpha_{n+1}=\frac{\B_\tau}{\H_\tau(u_{\alpha_n};g)}\alpha_n \geq \alpha_n$ and by Lemma \ref{Lemma:Fnondecreasing} it follows that $\H_\tau(u_{\alpha_n};g)\leq \H_\tau(u_{\alpha_{n+1}};g)$. Moreover, by the assumption that $\alpha \mapsto \frac{\H_\tau(u_{\alpha};g)}{\alpha}$ is non-increasing we obtain $\H_\tau(u_{\alpha_{n+1}};g) \leq \frac{\alpha_{n+1}}{\alpha_n} \H_\tau(u_{\alpha_n};g)=\B_\tau.$ That is,
$$
\H_\tau(u_{\alpha_n};g)\leq \H_\tau(u_{\alpha_{n+1}};g) \leq \B_\tau. 
$$
The rest of the proof is analog to the one of Theorem \ref{Theorem:conv2}.\qed
\end{proof}


Nothing is known about the convergence of the APS-algorithm, if $\Bt(\cdot)$ indeed depends on $u$ and $\Bt(u_{\alpha_n})$ is used instead of a fixed constant. In particular, in our numerics for some examples, in particular for the application of removing random-valued impulse noise with $r=0.05$, we even observe that starting from a certain iteration the sequence $(\alpha_n)_n$ oscillates between two states, see Fig. \ref{fig_alpha_RVc}. This behavior can be attributed to the fact that, for example, if $H_\tau(u_{\alpha_n}) \leq \Bt(u_{\alpha_n})$, then it is not guaranteed that also $H_\tau(u_{\alpha_{n+1}}) \leq \Bt(u_{\alpha_{n+1}})$, which is essential for the convergence. 

The second assumption in the previous theorem, i.e., the non-increase of the function $\alpha \mapsto \frac{\H_\tau(u_\alpha;g)}{\alpha}$, can be slightly loosened, since for the convergence of the APS-algorithm it is enough to demand the non-increase starting from a certain iteration $\tilde n\geq 0$. That is, if there exists a region $U\subset \R^+$ where $\alpha \mapsto \frac{\H_\tau(u_\alpha;g)}{\alpha}$ is non-increasing and $(\alpha_n)_{n\geq \tilde{n}} \subset U$, then the algorithm converges; see Fig. \ref{fig_phantom1_decr}. Analytically, this can be easily shown via Theorem \ref{Theorem:conv} by just considering $\alpha_{\tilde{n}}$ as the initial value of the algorithm. If $\tau=2$, similar to the CPS-algorithm, we are able to show the following monotonicity property.
\begin{proposition}\label{Proposition4.1}
If there exists a constant $c > 0$ such that $\|T^*(T u - g)\|_{L^2(\Omega)} = c \|T u - g\|_{L^2(\Omega)}$ for all $u\in L^2(\Omega)$, then the function $\alpha \mapsto \frac{\sqrt{\H_2(u_\alpha;g)}}{\alpha}$ is non-increasing, where $u_\alpha$ is a minimizer of $ \Js_2(u,\alpha)$. 
\end{proposition}

\begin{proof}
We start by replacing the functional $\Js_2$ by a family of surrogate functionals denoted by $\bar{\mcS}$ and defined for $u,a\in X$ as
\begin{equation*}
\begin{split}
\bar{\mcS}(u,a)&:= \Js_{2}(u,\alpha) + \frac{\delta}{2}\|u-a\|_{L^2(\Omega)}^2 - \frac{1}{2}\|T(u-a)\|_{L^2(\Omega)}^2 \\
&=\delta\|u-z(a)\|_{L^2(\Omega)}^2 + 2\alpha \int_\Omega |Du| + \psi(a,g,T)
\end{split}
\end{equation*}
where $\delta>\|T\|^2$, $z(a):=a - \frac{1}{\delta}T^*(T a - g)$, and $\psi$ is a function independent of $u$. It can be shown that the iteration
\begin{equation}\label{surrogate}
u_{\alpha,0}\in X, \quad u_{\alpha,k+1}=\argmin_u \bar{\mcS}(u,u_{\alpha,k}), \ k\geq0
\end{equation}
generates a sequence $(u_{\alpha,k})_k$ which converges weakly for $k\to \infty$ to a minimizer $u_{\alpha}$ of $\Js_2(u,\alpha)$, see for example \cite{DauTesVes}. The unique minimizer $u_{\alpha,k+1}$ is given by $u_{\alpha,k+1} = (I-P_{\frac{\alpha}{\delta}K_{}})(z(u_{\alpha,k}))$, where $K$ is the closure of the set 
$$
\{\Div \xi \ : \ \xi\in C_c^1(\Omega,\R^2), |\xi(x)| \leq 1 \ \forall x\in \Omega\}
$$ 
and $P_K(u):=\argmin_{v\in K} \|u-v\|_{L^2(\Omega)}$; see \cite{Cha}. Then for $k\to \infty$, let us define 
$$
\tilde f(\alpha):= \left\|P_{\frac{\alpha}{\delta}K_{}}( z(u_{\alpha}))\right\|_{L^2(\Omega)} = \left\|\frac{1}{\delta}T^*(T u_{\alpha} - g)\right\|_{L^2(\Omega)}.
$$ 
Since $\|T^*(T u - g)\|_{L^2(\Omega)} = c \|T u - g\|_{L^2(\Omega)}$, it follows that $\tilde f(\alpha)=\frac{c}{\delta} \sqrt{2 \H_2(u_\alpha;g)}$. The assertion follows by applying \cite[Lemma 4.1]{Cha}, which is extendable to infinite dimensions, to $\tilde f$ and by noting that the non-increase of $\alpha \mapsto \frac{\tilde{f}(\alpha)}{\alpha}$ implies the non-increase of $\alpha \mapsto \frac{\sqrt{\H_2(u_\alpha;g)}}{\alpha}$.\qed
\end{proof}

We remark that for convolution type of operators the assumption of Proposition \ref{Proposition4.1} does not hold in general. However, there exist several operators $T$, relevant in image processing, with the property $\|T^*(Tu-g)\|_{L^2(\Omega)} = \|Tu-g\|_{L^2(\Omega)}$. Such operators include $T=I$ for image denoising, $T=1_D$ for image inpainting, where $1_D$ denotes the characteristic function of the domain $D\subset \Omega$, and $T=S\circ A$, where $S$ is a subsampling operator and $A$ is an analysis operator of a Fourier or orthogonal wavelet transform. The latter type of operator is used for reconstructing signals from partial Fourier data \cite{CanRomTao} or in wavelet inpainting \cite{ChaSheZho}, respectively. For all such operators the function $\alpha \mapsto \frac{\sqrt{\H_2(u_\alpha;g)}}{\alpha}$ is non-increasing and hence by setting $p=\frac{1}{2}$ fixed in the pAPS-algorithm or changing the update of $\alpha$ in the APS-algorithm to
$$
\alpha_{n+1} := \sqrt{\frac{\B_2}{\H_2(u_{\alpha_n};g)}} \alpha_n,
$$
where $\B_2$ is a fixed constant, chosen according to \eqref{LtauTVmodel}, we obtain in these situations a convergent algorithm. 
 
We emphasize once more, that in general the non-increase of the function $\alpha \mapsto \frac{\H_\tau(u_\alpha;g)}{\alpha}$ is not guaranteed. 
Nevertheless, there exists always a constant $p\geq0$ such that $\alpha \mapsto \frac{(\H_\tau(u_\alpha;g))^p}{\alpha}$ is indeed non-increasing. For example, $p=\frac{1}{2}$ for operators $T$ with the property $\|T^*(Tu-g)\|_{L^2(\Omega)} = \|Tu-g\|_{L^2(\Omega)}$; cf. Propsition \ref{Proposition4.1}. In particular, one easily checks the following result.
\begin{proposition}\label{Proposition:q}
Let $0<\alpha\leq \beta$, and $u_\alpha$ and $u_\beta$ minimizers of $\Js_\tau(\cdot,\alpha)$ and $\Js_\tau(\cdot,\beta)$, respectively, for $\tau=1,2$. Then $\frac{(\H_\tau(u_{\beta};g))^p}{\beta} \leq \frac{(\H_\tau(u_{\alpha};g))^p}{\alpha}$ if and only if $p \leq \frac{\ln{\beta} - \ln{\alpha}}{\ln{\H_\tau(u_{\beta};g)} - \ln{\H_\tau(u_{\alpha};g)}}$.
\end{proposition}

\section{Locally constrained TV problem}\label{Sec:locallyTV}

In order to enhance image details, while preserving homogeneous regions, we formulate, as in \cite{DonHinRin,HinRin}, a locally constrained optimization problem. That is, instead of considering \eqref{LtauTVmodel} we formulate 
\begin{equation}\label{cont:lcminP}
\begin{split}
\min_{u\in BV(\Omega)} \int_\Omega |Du| \ \text{ s.t. } \ \int_\Omega w(x,y) |Tu - g|^\tau(y) dy \leq \nu_\tau 
\end{split}
\end{equation}
for almost every $x\in \Omega$, where $w$ is a normalized filter, i.e., $w\in L^\infty(\Omega\times\Omega)$, and $w\geq0$ on $\Omega\times\Omega$ with
\begin{equation}\label{cont:filter}
\begin{split}
&\int_{\Omega}\int_\Omega w(x,y)dy dx =1 \quad \\
&\text{ and } \quad \\
&\int_{\Omega}\int_\Omega w(x,y)|\phi(y)|^\tau dy dx \geq \epsilon \|\phi\|^\tau_{L^\tau(\Omega)} \ 
\end{split}
\end{equation}
for all $\phi\in L^{\tau}(\Omega)$ and for some $\epsilon>0$ independent of $\phi$; cf. \cite{DonHinRin,HinRin}.

\subsection{Local filtering}

In practice for $w$ we may use the mean filter together with a windowing technique, see for example \cite{DonHinRin,HinRin}. In order to explain the main idea we continue in a discrete setting.  Let $\Omega^h$ be a discrete image domain containing $N_1\times N_2$ pixels, $N_1,N_2\in\N$, and  by $|\Omega^h|=N_1N_2$ we denote the size of the discrete image (number of pixels). We approximate functions $u$ by discrete functions, denoted by $u^h$. The considered functions spaces are $X=\R^{N_1\times N_2}$ and $Y=X\times X$. In what follows for all $u^h\in X$ we use the following norms
$$
 \|u^h\|_{\tau} := \|u^h\|_{\ell^\tau(\Omega^h)} = \left( \sum_{x\in\Omega^h} |u^h(x)|^\tau\right)^{1/\tau} 
$$
for $1\leq \tau <+\infty$. Moreover we denote by $u^h_\Omega$ the average value of $u^h\in X$, i.e, $ u^h_\Omega := \frac{1}{|\Omega^h|} \sum_{x\in\Omega^h} u^h(x)$. The discrete gradient $\nabla^h: X \to Y$ and the discrete divergence $\Div^h : Y \to X$ are defined in a standard-way by forward and backward differences such that $\Div^h:= - (\nabla^h)^* $; see for example \cite{Cha,ChaPoc,HinLan2015_1,LanOshSch}. With the above notations and definitions the discretization of the general function in \eqref{LtauTVmultimodel} is given by
\begin{equation}\label{minPh}
J_\tau(u^h,\alpha):=H_\tau(u^h) + \tv_ \alpha(u^h)
\end{equation}
where  $H_\tau(u^h)=\frac{1}{\tau} \|T^hu^h-g^h\|_{\tau}^{\tau}$, $\tau \in\{1,2\}$, $T^h : X \to X$ is a bounded linear operator, $\alpha\in (\R^+)^{N_1 \times N_2}$, and 
\begin{equation}\label{DefTVa}
R_\alpha(u^h) :=  \sum_{x\in\Omega^h} \alpha(x) |\nabla^h u^h(x) |_{l^2}
\end{equation}
with $|y|_{l^2}=\sqrt{y_1^2 + y_2^2}$ for every $y=(y_1,y_2)\in \R^2$. In the sequel if $\alpha$ is a scalar or $\alpha \equiv 1$ in \eqref{DefTVa}, we write instead of $R_\alpha$ or $R_1$ just $\alpha R$ or $R$, respectively, i.e.,
$$
R(u^h) = \sum_{x\in\Omega^h}|\nabla^h u^h(x)|_{l^2}
$$
is the discrete total variation of $u$ in $\Omega^h$, and we write $\bar{E}_\tau$ instead of $E_\tau$ to indicate that $\alpha$ is constant. Introducing some step-size $h$, then for $h\to 0$ (i.e. the number of pixels $N_1N_2$ goes to infinity) one can show, similar as for the case $\alpha\equiv 1$, that $R_{\alpha}$ $\Gamma$-converges to $\int_\Omega \alpha |Du|$; see \cite{Bra,Lan}. 

We turn now to the locally constrained minimization problem, which is given in the discrete setting as
\begin{equation}\label{lcTVP}
\min_{u^h\in X} \tv(u^h)  \quad \text{ s.t. } \quad S_{i,j}^{\tau}(u^h)\leq\frac{\nu_\tau}{\tau} \ \text{ for all } x_{i,j} \in \Omega^h.
\end{equation}
Here $\nu_\tau$ is a fixed constant and 
$$
S_{i,j}^{\tau}(u^h) := \frac{1}{M_{i,j}}\sum_{x_{s,t}\in \mathcal{I}_{i,j}}\frac{1}{\tau}|(T^h u^h)(x_{s,t}) - g^h(x_{s,t}) |^\tau
$$ 
denotes the local residual at $x_{i,j}\in \Omega^h$ with $\mathcal{I}_{i,j}$ being some suitable set of pixels around $x_{i,j}$ of size $M_{i,j}$, i.e., $M_{i,j}= |\mathcal{I}_{i,j}|$. For example, in \cite{DonHinRin,HinLan2015,HinRin} for $\mathcal{I}_{i,j}$ the set
$$
\Omega_{i,j}^{\omega} = \left\{x_{s+i,t+j}\in \Omega^h : - \frac{\omega-1}{2}\leq s,t \leq  \frac{\omega-1}{2}\right\}
$$
with a symmetric extension at the boundary and with $\omega$ being odd is used. That is, $\Omega_{i,j}^\omega$ is a set of pixels in a $\omega$-by-$\omega$ window centered at $x_{i,j}$, i.e., $M_{i,j}=\omega^2$ for all $i,j$, such that $\Omega_{i,j}^\omega \not\subset \Omega^h$ for $x_{i,j}$ sufficiently close to $\partial \Omega$. Additionally we denote by $\tilde \Omega_{i,j}^{\omega}$ a set of pixels in a window centered at $x_{i,j}$ without any extension at the boundary, i.e., 
\begin{equation*}
\begin{split}
\tilde \Omega_{i,j}^{\omega} = \Bigg\{x_{s+i,t+j} : \max\left\{1-(i,j),- \frac{\omega-1}{2}\right\}\leq (s,t) \\
 \leq  \min\left\{\frac{\omega-1}{2},(N_1-i,N_2-j)\right\}\Bigg\}.
\end{split}
\end{equation*}
Hence $\tilde \Omega_{i,j}^\omega \subset \Omega^h$ for all $x_{i,j}\in \Omega^h$. Before we analyze the difference between $\Omega_{i,j}^{\omega}$ and $\tilde \Omega_{i,j}^{\omega}$ with respect to the constrained minimization problem \eqref{lcTVP}, we note that, since $T^h$ does not annihilate constant functions, the existence of a solution of \eqref{lcTVP} is guaranteed; see \cite[Theorem 2]{DonHinRin}\cite[Theorem 2]{HinRin}.

In the following we set $B_\tau:=\frac{\nu_\tau}{\tau}|\Omega^h|$.

\begin{proposition}\label{Lemma:pixelset}
\begin{itemize}
\item[(i)]
If $u^h$ is a solution of \eqref{lcTVP} with $\mathcal{I}_{i,j}= \Omega_{i,j}^{\omega}$, then $H_\tau(u^h) < B_\tau.$
\item[(ii)] 
If $u$ is a solution of \eqref{lcTVP} with $\mathcal{I}_{i,j}= \tilde\Omega_{i,j}^{\omega}$, then $H_\tau(u^h) \leq B_\tau.$
\end{itemize}
\end{proposition}

\begin{proof}
\begin{itemize}
\item[(i)]
Since $u^h$ is a solution of \eqref{lcTVP} and $\Omega_{i,j}^\omega$ is a set of pixels in a $\omega$-by-$\omega$ window, we have 
\begin{equation*}
\begin{split}
B_\tau &\geq \sum_{i,j}S_{i,j}^{\tau}(u^h) \\
&=  \sum_{i,j} \frac{1}{\tau \omega^2}\sum_{x_{s,t}\in \Omega_{i,j}^\omega}|g^h(x_{s,t})- T^h u^h(x_{s,t}) |^\tau \\
 &> \frac{1}{\tau}\sum_{i,j} |g^h(x_{i,j})- T^h u^h(x_{i,j})|^\tau = H_\tau(u^h).
\end{split}
\end{equation*}
Here we used that due to the sum over $i,j$ each element (pixel) in $\Omega_{i,j}^\omega$ appears at most $\omega^2$ times. More precisely, any pixel-coordinate in the set $\Lambda^\omega:= \{(i,j) : \min\{i-1,j-1,N_1-i,N_2-j\}\geq \frac{\omega-1}{2}\}$ occurs exactly $\omega^2$-times, while any other pixel-coordinate appears strictly less than $\omega^2$-times. This shows the first statement.
 
\item[(ii)] For a minimizer $u^h$ of \eqref{lcTVP} we obtain
\begin{equation*}
\begin{split}
B_\tau &\geq \sum_{i,j}S_{i,j}^{\tau}(u^h) \\
&=  \sum_{i,j} \frac{1}{\tau M_{i,j}}\sum_{x_{s,t}\in \Omega_{i,j}^\omega}|g^h(x_{s,t}) - T^h u^h(x_{s,t})|^\tau \\
&=  \frac{1}{\tau}\sum_{i,j} |g^h(x_{i,j})- T^h u^h(x_{i,j})|^\tau = H_\tau(u^h),
\end{split}
\end{equation*}
which concludes the proof.\qed
\end{itemize}
\end{proof}
Note, that if $\mathcal{I}_{i,j} = \tilde\Omega_{i,j}^{\omega}$ then by Proposition \ref{Lemma:pixelset} a minimizer of \eqref{lcTVP} also satisfies the constraint of the problem
\begin{equation} \label{cminPh}
\min_{u^h\in X} R(u^h) \quad \text{s.t.} \quad H_\tau(u^h)\leq B_\tau
\end{equation}
(discrete version of \eqref{cminPA}) but is in general of course not a solution of \eqref{cminPh}. 

\begin{proposition}\label{Prop:TVlvsTVs}
Let $\mathcal{I}_{i,j} = \tilde\Omega_{i,j}^{\omega}$, $u_s^h$ be a minimizer of \eqref{cminPh} and $u_l^h$ be a minimizer of \eqref{lcTVP}, then $\tv(u_s^h)\leq \tv(u_l^h)$. 
\end{proposition}
\begin{proof}
Assume that $\tv(u_s^h)> \tv(u_l^h)$. Since $u_s^h$ is a solution of \eqref{cminPh} it satisfies the constraint $H_\tau(u_s^h)\leq B_\tau$. By Proposition \ref{Lemma:pixelset} we also have $H_\tau(u_l^h)\leq B_\tau$. Since $\tv(u_s^h) > \tv(u_l^h)$, $u_s^h$ is not the solution of \eqref{cminPh} which is a contradiction. Hence, $\tv(u_s^h)\leq \tv(u_l^h)$.\qed
\end{proof}

\begin{remark}\label{Remark:TVlvsTVs}
Proposition \ref{Lemma:pixelset} and its consequence are not special properties of the discrete setting. 
Let the filter $w$ in \eqref{cont:lcminP} be such that the inequality in \eqref{cont:filter} becomes an equality with $\epsilon=1/|\Omega|$, as it is the case in Proposition \ref{Lemma:pixelset}(ii), then a solution ${u}_l$ of the locally constrained minimization problem \eqref{cont:lcminP} satisfies
$$
\mcD_\tau({u}_l;g)\leq\frac{\nu_\tau}{\tau}|\Omega| \ \text{ and }\ \int_\Omega |D{u}_l| \geq\int_\Omega |D{u}_s| 
$$
where ${u}_s$ is a solution of \eqref{cminPA}.
\end{remark}

From Proposition \ref{Prop:TVlvsTVs} and Remark \ref{Remark:TVlvsTVs} we conclude, since $\tv(u_s^h) \leq \tv(u_l^h)$ and $\int_\Omega |D{u}_s| \leq\int_\Omega |D{u}_l|$, that $u_s^h$ and ${u}_s$ are smoother than $u_l^h$ and ${u}_l$, respectively. Hence the solution of the locally constrained minimization problem is expected to preserve details better than the minimizer of the globally constrained optimization problem. Since noise can be interpreted as fine details, which we actually want to eliminate, this could also mean, that noise is possibly left in the image.

\subsection{Locally adaptive total variation algorithm}
Whenever $\nu_\tau$ depends on $\hat{u}$ problem \eqref{lcTVP} results in a quite nonlinear problem. Instead of considering nonlinear constraints we choose as in Section \ref{Sec:APS} a reference image $\tilde u$ and compute an approximate $\nu_\tau=\nu_\tau(\tilde u)$. Note, that in our discrete setting for salt-and-pepper noise we have now
\begin{equation*}
\nu_1({u}^h) := r_2-(r_2-r_1) \frac{1}{|\Omega|}\sum_{x\in \Omega^h}(T^h{u}^h)(x)
\end{equation*}
and for random-valued impulse noise we have
\begin{equation*}
\nu_1({u}^h) := \frac{1}{|\Omega^h|}\sum_{x\in \Omega^h}r\left((T^h{u}^h)(x)^2 - (T^h {u}^h)(x) +\frac{1}{2}\right).
\end{equation*}
 In our below proposed locally adaptive algorithms we choose as a reference image the current approximation (see LATV- and pLATV-algorithm below), as also done in the pAPS- and APS-algorithm above. Then we are seeking for a solution $u^h$ such that $S_{i,j}^{\tau}(u^h)$ is close to $\frac{\nu_\tau}{\tau}$. 

We note, that for large $\alpha>0$ the minimization of \eqref{minPh} yields an over-smoothed restoration $u_\alpha^h$ and the residual contains details, i.e., we expect $H_\tau(u_\alpha^h)> B_\tau$. 
 Hence, if $ S_{i,j}^{\tau}(u_\alpha^h) > \frac{\nu_\tau}{\tau}$
we suppose that this is due to image details contained in the local residual image. In this situation we intend to decrease $\alpha$ in the local regions $\mathcal{I}_{i,j}$. In particular, we define, similar as in \cite{DonHinRin,HinRin}, the local quantity $f^\omega_{i,j}$ by
$$
f^\omega_{i,j}:= \begin{cases}
S_{i,j}^{\tau}(u_\alpha^h) & \text{ if } S_{i,j}^{\tau}(u_\alpha^h)> \frac{\nu_\tau}{\tau},\\
\frac{\nu_\tau}{\tau} & \text{otherwise}.
\end{cases}
$$ 
Note, that $\frac{\nu_\tau}{\tau f^\omega_{i,j}}\leq 1$ for all $i,j$ and hence we set
\begin{equation}\label{alpha}
\alpha(x_{i,j}):= \frac{1}{ M_{i,j}}\sum_{x_{s,t}\in \mathcal{I}_{i,j}} \left( \frac{\nu_\tau }{\tau f^\omega_{s,t}} \right)^p \alpha(x_{s,t}).
\end{equation}
On the other hand, for small $\alpha>0$ we get an under-smoothed image $u_\alpha^h$, which still contains noise, i.e., we expect $H_\tau(u_\alpha^h)< B_\tau$. Analogously, if $S_{i,j}^{\tau}(u_\alpha^h) \leq \frac{\nu_\tau}{\tau}$, we suppose that there is still noise left outside the residual image in $\mathcal{I}_{i,j}$. Hence we intend to increase $\alpha$ in the local regions $\mathcal{I}_{i,j}$ by defining 
$$
f^\omega_{i,j}:= \begin{cases}
S_{i,j}^{\tau}(u_\alpha^h) & \text{ if } S_{i,j}^{\tau}(u_\alpha^h)< \frac{\nu_\tau}{\tau},\\
\frac{\nu_\tau}{\tau} & \text{otherwise},
\end{cases}
$$ 
and setting $\alpha$ as in \eqref{alpha}. Notice, that now $\frac{\nu_\tau }{\tau f^\omega_{i,j}}\geq 1$. These considerations lead to the following locally adapted total variation algorithm.

\noindent
\fbox{
\begin{minipage}{7.9cm}
\textbf{LATV-algorithm:} Choose $\alpha_{0}>0$, $p:=p_0>0$, and set $n:=0$.
\begin{itemize}
\item[1)] Compute $u_{\alpha_n}^h \in \argmin_{u^h\in X} J_\tau(u^h,\alpha_n)$
\item[2)] 
	\begin{itemize}
	\item[(a)]  If $H_\tau(u_{\alpha_0}^h) >B_\tau(u_{\alpha_0}^h)$, then set $$f_{i,j}^{\omega}:= \max\left\{S_{i,j}^{\tau}(u_{\alpha_n}^h),\tfrac{\nu_\tau(u_{\alpha_n}^h)}{\tau}\right\}$$
	\item[(b)]  If $H_\tau(u_{\alpha_0}^h)\leq B_\tau(u_{\alpha_0}^h)$, then set $$f_{i,j}^{\omega}:= \max\left\{\min\left\{S_{i,j}^{\tau}(u_{\alpha_n}^h),\tfrac{\nu_\tau(u_{\alpha_n}^h)}{\tau}\right\},\varepsilon\right\}$$
	\end{itemize}

\item[3)] Update 
$$
\alpha_{n+1}(x_{i,j}):=\tfrac{1}{ M_{i,j}}\sum_{x_{s,t}\in \mathcal{I}_{i,j}} \left( \tfrac{\nu_\tau(u_{\alpha_n}^h) }{\tau f^\omega_{s,t}} \right)^p \alpha_n(x_{s,t}).
$$
\item[4)] Stop or set $n:=n+1$ and return to step 1).
\end{itemize}

\end{minipage}
}\\

Here and below $\varepsilon>0$ is a small constant (e.g., in our experiments we choose $\varepsilon=10^{-14}$) to ensure that $f_{i,j}^\omega>0$, since it may happen that $S_{i,j}^{\tau}(u_{\alpha_n}^h)=0$.

If $H_\tau(u_{\alpha_0}^h) >B_\tau(u_{\alpha_0}^h)$, we stop the algorithm as soon as the residual $H_\tau(u_{\alpha_n}^h)<B_\tau(u_{\alpha_n}^h)$ for the first time and set the desired locally varying $\alpha^*=\alpha_{n}$. If $H_\tau(u_{\alpha_0}^h) \leq B_\tau(u_{\alpha_0}^h)$, we stop the algorithm as soon as the residual $H_\tau(u_{\alpha_n}^h)>B_\tau(u_{\alpha_n}^h)$ for the first time and set the desired locally varying $\alpha^*=\alpha_{n-1}$, since $H_\tau(u_{\alpha_{n-1}}^h)\leq \frac{\nu_\tau(u_{\alpha_{n-1}}^h)}{\tau}$.

The LATV-algorithm has the following monotonicity properties with respect to $(\alpha_n)_n$.

\begin{proposition}
Assume $\mathcal{I}_{i,j} = \Omega_{i,j}^\omega$ and let $\varepsilon>0$ be sufficiently small. If $\alpha_0 >0$ such that $H_\tau (u_{\alpha_0}^h) \leq B_\tau(u_{\alpha_0}^h)$, then the LATV-algorithm generates a sequence $(\alpha_n)_n$ such that $$
\sum_{i,j} \alpha_{n+1}(x_{i,j}) > \sum_{i,j} \alpha_{n}(x_{i,j}).
$$
\end{proposition}

\begin{proof}
By the same argument as in the proof of Proposition \ref{Lemma:pixelset} we obtain
\begin{equation*}
\begin{split}
\sum_{i,j} \alpha_{n+1}(x_{i,j}) &= \sum_{i,j} \left(\frac{(\nu_\tau(u_{\alpha_n}^h))^p}{\tau^p\omega^2}\sum_{(s,t)\in \Omega_{i,j}^\omega} \frac{\alpha_n(x_{s,t})}{(f_{s,t}^\omega)^p}\right)\\
& > \sum_{i,j} \left(\frac{(\nu_\tau(u_{\alpha_n}^h))^p \omega^2}{\tau^p \omega^2} \frac{\alpha_n(x_{i,j})}{(f_{i,j}^\omega)^p}\right) .
\end{split}
\end{equation*}
Note that $\nu_\tau(\cdot)$ is bounded from below, see Section \ref{Sec:CharNoise}. Consequently there exists an $\varepsilon>0$ such that $\frac{\nu_\tau(u^h)}{\tau} \geq \varepsilon$ for any $u^h$. Then, since $H_\tau(u_{\alpha_0}^h)\leq B_\tau(u_{\alpha_0}^h)$ we have by the LATV-al\-gorithm that  $$f_{i,j}^\omega:= \max\left\{\min\left\{S_{i,j}^{\tau}(u_{\alpha_n}^h),\frac{\nu_\tau(u_{\alpha_n}^h)}{\tau}\right\},\varepsilon\right\}\leq \frac{\nu_\tau(u_{\alpha_n}^h)}{\tau}$$ and hence 
$\sum_{i,j} (\alpha_{n+1})(x_{i,j}) > \sum_{i,j} (\alpha_{n})(x_{i,j}).$ \qed
\end{proof}

\begin{proposition}
Let $\mathcal{I}_{i,j} = \tilde\Omega_{i,j}^\omega$ and $\varepsilon>0$ be sufficiently small. 
\begin{itemize}
\item[(i)] If $\alpha_0 >0$ such that $H_\tau (u_{\alpha_0}^h) > B_\tau(u_{\alpha_0}^h)$, then the LATV-algorithm generates a sequence $(\alpha_n)_n$ such that 
$$
\sum_{i,j} (\alpha_{n+1})(x_{i,j}) \leq \sum_{i,j} (\alpha_{n})(x_{i,j}).
$$
\item[(ii)] If $\alpha_0 >0$ such that $H_\tau (u_{\alpha_0}^h) \leq B_\tau(u_{\alpha_0}^h)$, then the LATV-algorithm generates a sequence $(\alpha_n)_n$ such that 
$$
\sum_{i,j} (\alpha_{n+1})(x_{i,j}) \geq \sum_{i,j} (\alpha_{n})(x_{i,j}).
$$
\end{itemize}
\end{proposition}
\begin{proof}
\begin{itemize}
\item[(i)] By the same argument as in the proof of Proposition \ref{Lemma:pixelset} and since $f_{i,j}^{\omega}:= \max\left\{S_{i,j}^{\tau}(u_{\alpha_n}^h),\frac{\nu_\tau(u_{\alpha_n}^h)}{\tau}\right\} \geq \frac{\nu_\tau(u_{\alpha_n}^h)}{\tau}$ we obtain
\begin{equation*}
\begin{split}
\sum_{i,j} \alpha_{n+1}(x_{i,j}) &= \sum_{i,j} \left(\frac{\nu_\tau(u_{\alpha_n}^h)^p}{\tau^p M_{i,j}}\sum_{x_{s,t}\in \tilde\Omega_{i,j}^\omega} \frac{\alpha_n(x_{s,t})}{(f_{s,t}^\omega)^p}\right)\\
&=\sum_{i,j} \left(\left(\frac{\nu_\tau(u_{\alpha_n}^h)}{\tau}\right)^p \frac{\alpha_n(x_{i,j})}{(f_{i,j}^\omega)^p}\right)\\
& \leq \sum_{i,j} \alpha_{n}(x_{i,j}).
\end{split}
\end{equation*}

\item[(ii)] Since $\nu_\tau(\cdot)$ is bounded from below, see Section \ref{Sec:CharNoise}, there exists an $\varepsilon>0$ such that $\frac{\nu_\tau(u^h)}{\tau} \geq \varepsilon$ for any $u^h$. Hence 
\begin{equation*}
\begin{split}
f_{i,j}^\omega:= &\max\left\{\min\left\{S_{i,j}^{\tau}(u_{\alpha_n}^h),\frac{\nu_\tau(u_{\alpha_n}^h)}{\tau}\right\},\varepsilon\right\} \\
&\leq \frac{\nu_\tau(u_{\alpha_n}^h)}{\tau}
\end{split}
\end{equation*}
and by the same arguments as above we get 
\begin{equation*}
\begin{split}
\sum_{i,j} \alpha_{n+1}(x_{i,j})&=\sum_{i,j} \left(\left(\frac{\nu_\tau(u_{\alpha_n}^h)}{\tau}\right)^p \frac{\alpha_n(x_{i,j})}{(f_{i,j}^\omega)^p}\right) \\
&\geq \sum_{i,j} \alpha_{n}(x_{i,j}).
\end{split}
\end{equation*} \qed
\end{itemize}
\end{proof}

In contrast to the pAPS-algorithm in the LATV-algorithm the power $p>0$ is not changed during the iterations and should be chosen sufficiently small, e.g., we set $p=\frac{1}{2}$ in our experiments. Note, that small $p$ only allow small changes of $\alpha$ in each iteration. In this way the algorithm is able the generate a function $\alpha^*$ such that $H_\tau(u_{\alpha^*}^h)$ is very close to $\frac{\nu_\tau(u_{\alpha^*}^h)}{\tau}$. On the contrary, small $p$ have the drawback that the number of iterations till termination are kept large. Since the parameter $p$ has to be chosen manually, the LATV-algorithm, at least in the spirit, seems to be similar to Uzawa's method, where also a parameter has to be chosen. The proper choice of such a parameter might be complicated and hence we are desiring for an algorithm where we do not have to tune parameters manually. Because of this and motivated by the pAPS-algorithm we propose the following $p$ adaptive algorithm:



\noindent
\fbox{
\begin{minipage}{7.9cm}
\textbf{pLATV-algorithm:} Choose $\alpha_{0}>0$, $p:=p_0>0$, and set $n:=0$.
\begin{itemize}
\item[0)] Compute $u_{\alpha_n} \in \argmin_{u^h\in X}  J_\tau(u^h,\alpha_n)$
\item[1)] 
	\begin{itemize}
	\item[(a)]  If $H_\tau(u_{\alpha_0}^h) >B_\tau(u_{\alpha_0}^h)$, then set $$f_{i,j}^\omega:= \max\left\{S_{i,j}^{\tau}(u_{\alpha_n}^h),\tfrac{\nu_\tau(u_{\alpha_n}^h)}{\tau}\right\}$$
	\item[(b)]  If $H_\tau(u_{\alpha_0}^h)\leq B_\tau(u_{\alpha_0}^h)$, then set $$f_{i,j}^\omega:= \max\left\{\min\left\{S_{i,j}^{\tau}(u_{\alpha_n}^h),\tfrac{\nu_\tau(u_{\alpha_n}^h)}{\tau}\right\},\varepsilon \right\}$$
	\end{itemize}

\item[2)] Update 
$$
\alpha_{n+1}(x_{i,j}):=\tfrac{\alpha_n(x_{i,j})}{ M_{i,j}}\sum_{x_{s,t}\in \mathcal{I}_{i,j}} \left( \tfrac{\nu_\tau(u_{\alpha_n}^h) }{\tau f^\omega_{s,t}} \right)^p.
$$

\item[3)] Compute $u_{\alpha_{n+1}}^h \in \argmin_{u^h\in X} J_\tau(u^h,\alpha_{n+1})$

\item[4)] 
\begin{itemize}
\item[(a)] if $H_\tau(u_{\alpha_0}^h) \leq B_\tau(u_{\alpha_0}^h)$
\begin{itemize}
\item[(i)] if $H_\tau(u_{\alpha_{n+1}}^h) \leq B_\tau(u_{\alpha_{n+1}}^h)$, go to step 5)
\item[(ii)] if $H_\tau(u_{\alpha_{n+1}}^h) > B_\tau(u_{\alpha_{n+1}}^h)$, decrease $p$, e.g., set $p=p/10$, and go to step 2)
\end{itemize}

\item[(b)] if $H_\tau(u_{\alpha_0}^h) > B_\tau(u_{\alpha_0}^h)$
\begin{itemize}
\item[(i)] if $H_\tau(u_{\alpha_{n+1}}^h) \geq B_\tau(u_{\alpha_{n+1}}^h)$, go to step 5)
\item[(ii)] if $H_\tau(u_{\alpha_{n+1}}^h) < B_\tau(u_{\alpha_{n+1}}^h)$, decrease $p$, e.g., set $p=p/10$, and go to step 2)
\end{itemize}
\end{itemize}

\item[5)] Stop or set $n:=n+1$ and return to step 1).
\end{itemize}

\end{minipage}
}\\

In our numerical experiments this algorithm is terminated as soon as $|H_\tau(u_{\alpha_n}^h) - B_\tau(u_{\alpha_n}^h)| \leq 10^{-6}$ and $H_\tau(u_{\alpha_{n}}^h) \leq B_\tau(u_{\alpha_{n}}^h)$. Additionally we stop iterating when $p$ is less than machine precision, since then anyway no progress is to expect. Due to the adaptive choice of $p$ we obtain a monotonic behavior of the sequence $(\alpha_n)_n$.

\graphicspath{{./graphics/}}
\begin{figure*}[htbp!]
\begin{center}
\hspace{0cm}
\includegraphics[height=4.2cm]{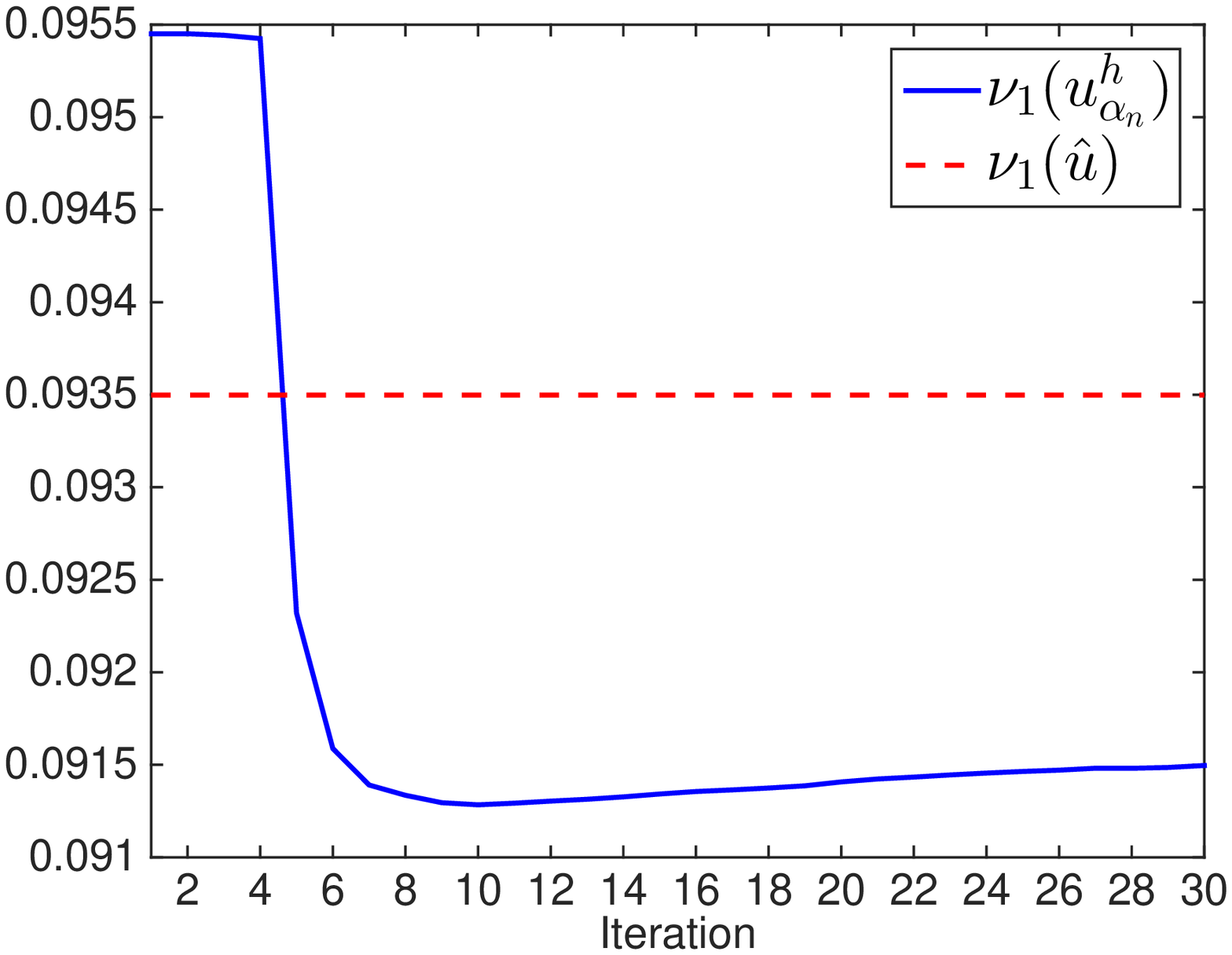}
\includegraphics[height=4.2cm]{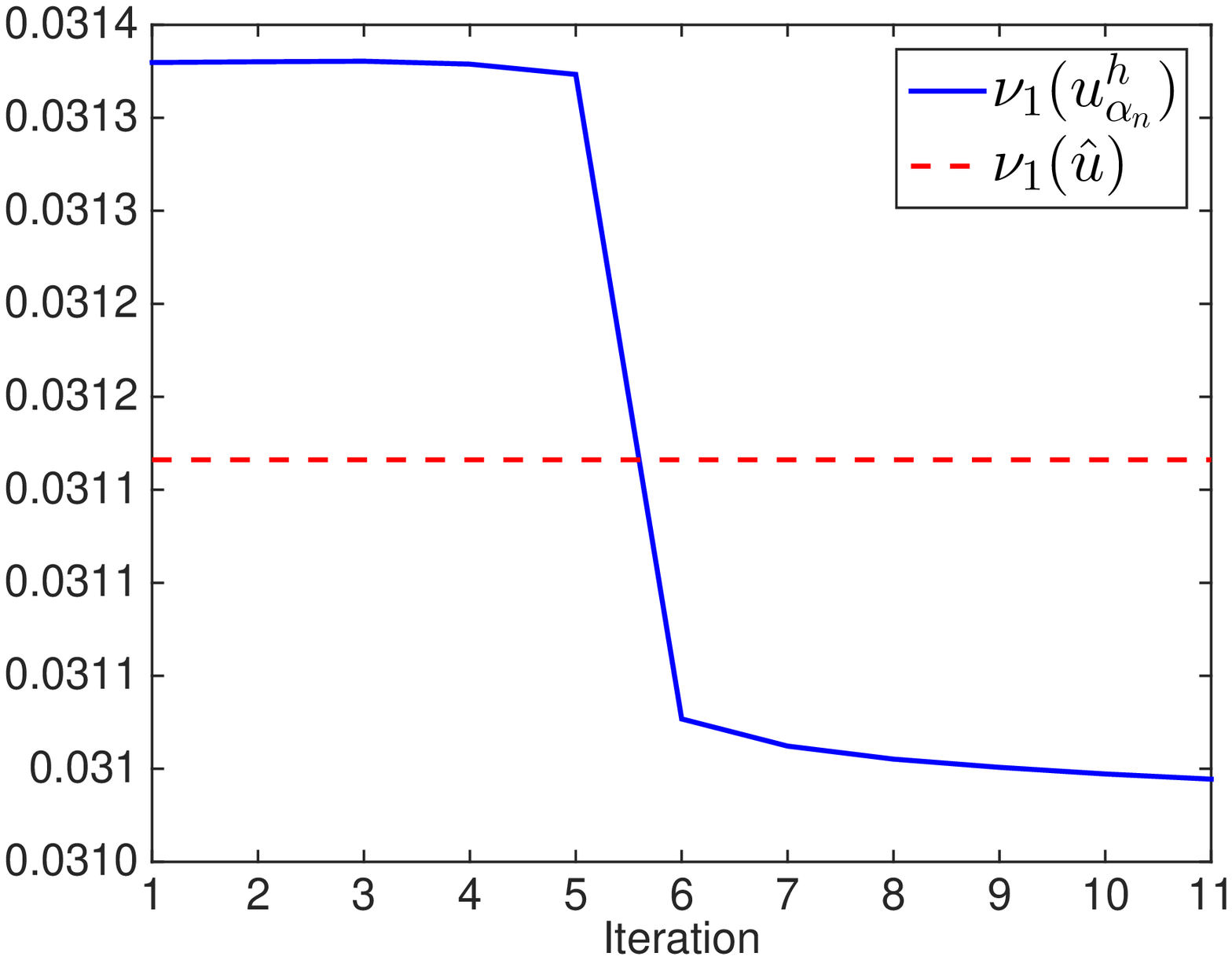}
\includegraphics[height=4.2cm]{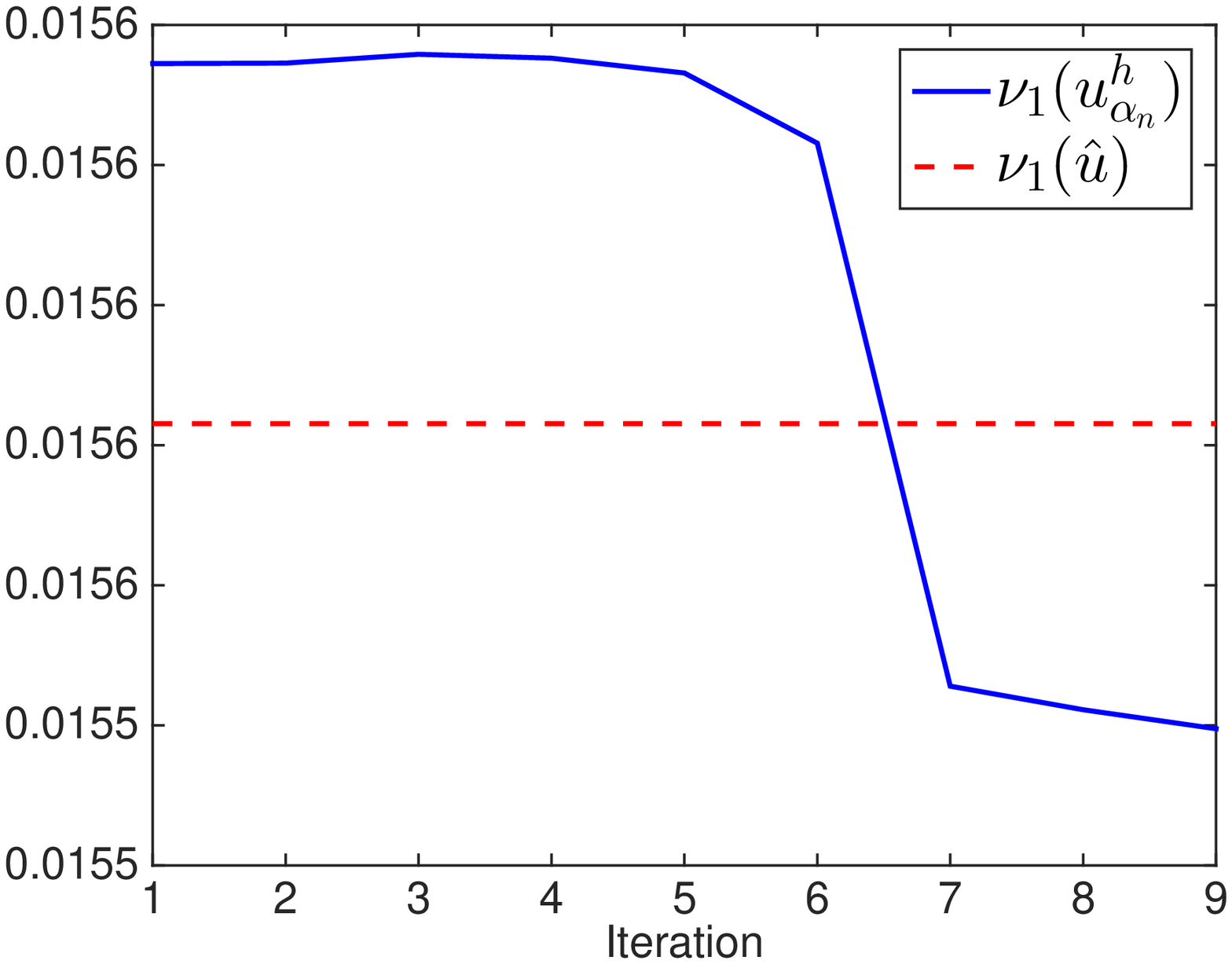}
\end{center}    
\caption{\small \it Progress of $\nu_\tau(u_{\alpha_n}^h)$ of the LATV-algorithm with $p=\frac{1}{8}$ and $\alpha_0=10^{-2}$ for removing random-valued impulse noise with $r=0.3$ (left), $r=0.1$ (middle), $r=0.05$ (right) from the cameraman-image (cf. Fig. \ref{fig_org1}(b)).  }
\label{fig_nu_plot_03}
\end{figure*}

\graphicspath{{./graphics/}}
\begin{figure*}[htbp!]
\begin{center}
\hspace{0cm}
\includegraphics[height=4.2cm]{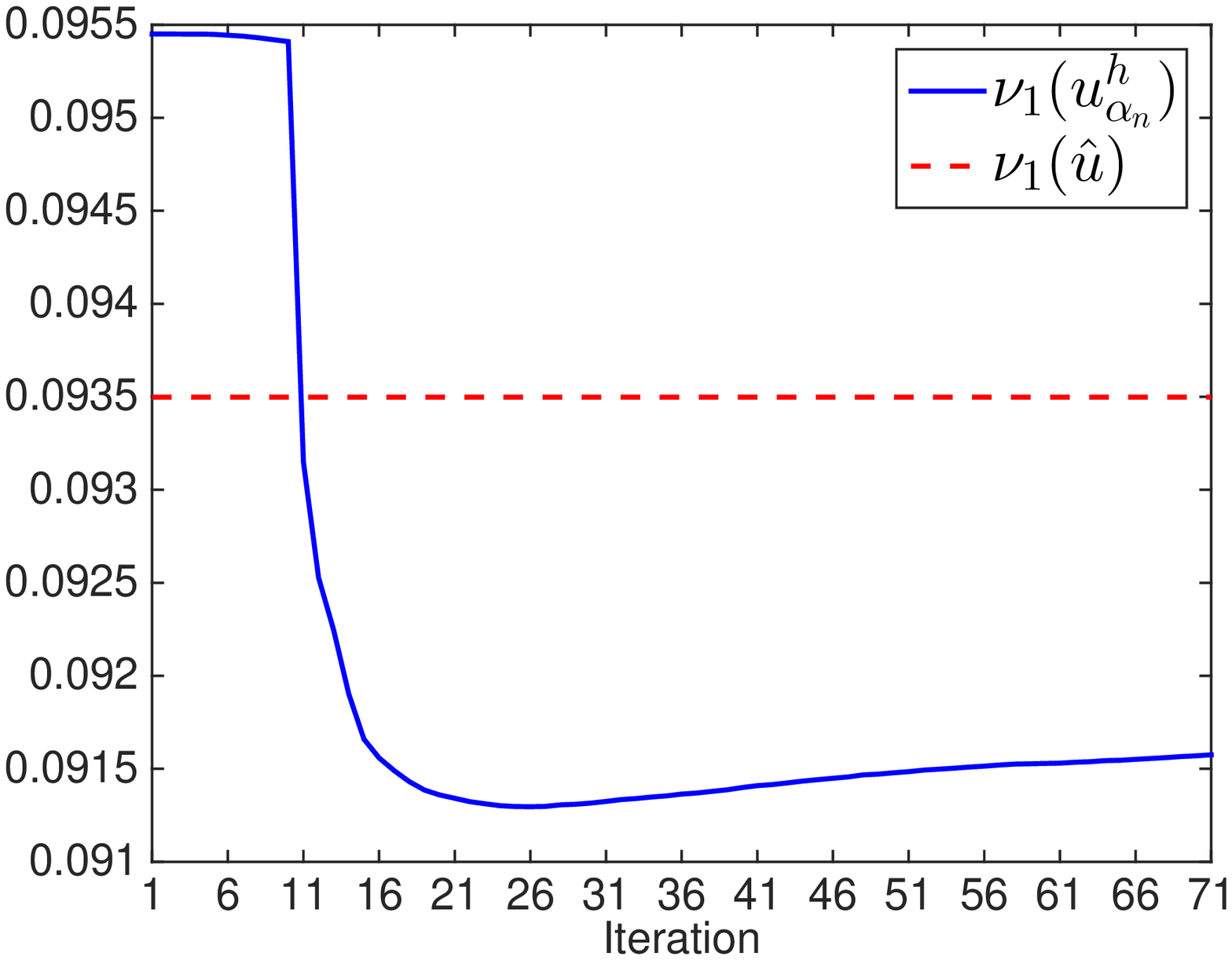}
\includegraphics[height=4.2cm]{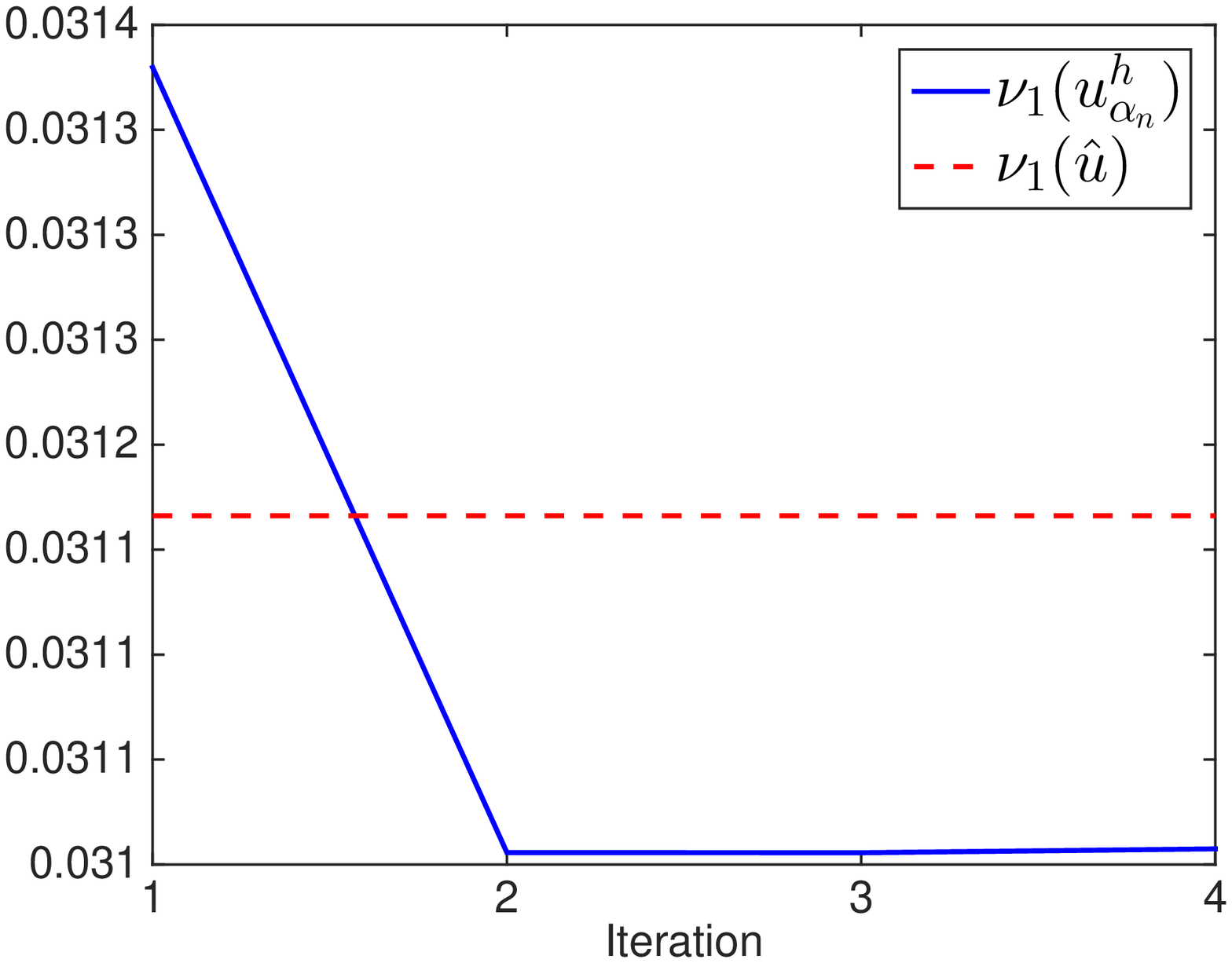}
\includegraphics[height=4.2cm]{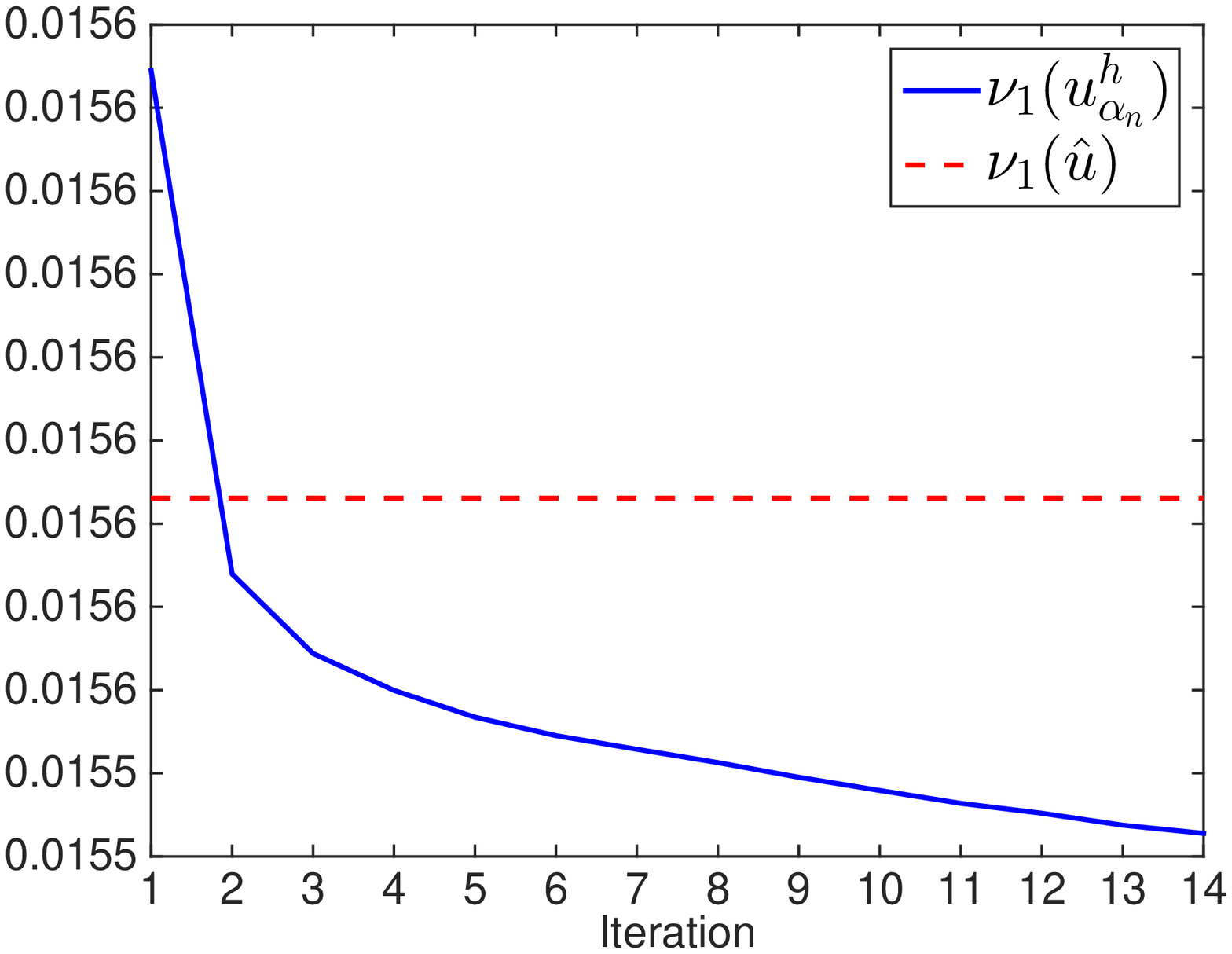}
\end{center}    
\caption{\small \it Progress of $\nu_\tau(u_{\alpha_n}^h)$ of the pLATV-algorithm with $p_0=\frac{1}{2}$ and $\alpha_0=10^{-2}$ for removing random-valued impulse noise with $r=0.3$ (left), $r=0.1$ (middle), $r=0.05$ (right) from the cameraman-image (cf. Fig. \ref{fig_org1}(b)).  }
\label{fig_nup_plot_03}
\end{figure*}

\begin{proposition}
The sequence $(\alpha_n)_n$ generated by the pLATV-algorithm is for any point $x\in \Omega$ monotone. In particular, it is monotonically decreasing for $\alpha_0$ such that $H_\tau(u_{\alpha_0}^h)>B_\tau(u_{\alpha_0}^h)$, and monotonically increasing for $\alpha_0$ such that $H_\tau(u_{\alpha_0}^h)\leq B_\tau(u_{\alpha_0}^h)$.
\end{proposition}

\begin{proof}
For $H_\tau(u_{\alpha_0}^h) > \Bt(u_{\alpha_0}^h)$ we can show by induction that by the pLATV-algorithm $f_{i,j}^\omega \geq \frac{\nu_\tau(u_{\alpha_n}^h)}{\tau}$ and hence $1 \geq \frac{\nu_\tau(u_{\alpha_n}^h)}{\tau f_{i,j}^\omega}$ for all $n$. Then by the definition of $\alpha_{n+1}$ it follows
\begin{equation*}
\begin{split}
\alpha_{n+1}(x_{i,j})&:=\frac{\alpha_n(x_{i,j})}{ M_{i,j}}\sum_{x_{s,t}\in \mathcal{I}_{i,j}} \left( \frac{\nu_\tau(u_{\alpha_n}^h) }{\tau f^\omega_{s,t}} \right)^p \\
&\leq \alpha_n(x_{i,j}).
\end{split}
\end{equation*}
By similar arguments we obtain for $\alpha_0$ with $H_\tau(u_{\alpha_0}^h) \leq B_\tau(u_{\alpha_0}^h)$ that $
\alpha_{n+1}(x_{i,j}) \geq \alpha_n(x_{i,j})
$ for all $x_{i,j}\in \Omega$. \qed
\end{proof}

We are aware of the fact that using $u_{\alpha_n}^h$ as a reference image in the LATV- and pLATV-algorithm to compute $\nu_\tau$ may commit errors. However, we recall that for Gaussian noise we set $\nu_2=\sigma^2$ and for salt-and-pepper noise with $r_1=r_2$ we have $\nu_1= r_1$. In these cases $\nu_\tau$ does not depend on the original image and hence we do not commit any error by computing $\nu_\tau$. For random-valued impulse noise corrupted images the situation is different and $\nu_1$ indeed depends on the true image. In this situation errors may be committed when $u_{\alpha_n}^h$ is used as a reference image for calculating $\nu_\tau$; see Figs. \ref{fig_nu_plot_03} and \ref{fig_nup_plot_03}. Hence, in order to improve the proposed algorithm, for such cases for future research it might be of interest to find the optimal reference image to obtain a good approximation of the real value $\nu_\tau$. 

In contrast to the SA-TV algorithm presented in \cite{DonHinRin,HinRin}, where the initial regularization parameter has to be chosen sufficiently small, in the LATV-algorithm as well as in the pLATV-algorithm the initial value $\alpha_0$ can be chosen arbitrarily positive. However, in the case $H_\tau(u_{\alpha_0}^h) > B_\tau(u_{\alpha_0}^h)$ we cannot guarantee in general that the solution $u_{\alpha}$ obtained by the pLATV-algorithm fulfills $H_\tau(u_{\alpha}^h) \leq B_\tau(u_{\alpha}^h)$, not even if $B_\tau(\cdot)$ is constant, due to the stopping criterion with respect to the power $p$. On the contrary, if $H_\tau(u_{\alpha_0}^h) \leq B_\tau(u_{\alpha_0}^h)$, then the pLATV-algorithm generates a sequence $(u_{\alpha_n}^h)_n$ such that $H_\tau(u_{\alpha_n}^h) \leq B_\tau(u_{\alpha_n}^h)$ for all $n$ and hence also for the solution of the algorithm. As a consequence we would wish to choose $\alpha_0>0$ such that $H_\tau(u_{\alpha_0}^h) \leq B_\tau(u_{\alpha_0}^h)$, which may be realized by the following simple automated procedure:

\noindent
\fbox{
\begin{minipage}{7.9cm}
\textbf{Algorithm 1:} Input: $\alpha_0 >0$ (arbitrary);
\begin{itemize}
\item[1)] Compute $u_{\alpha_0}^h \in \argmin_{u^h\in X}  J_\tau(u^h,\alpha_0)$.
\item[2)] If $H_\tau(u_{\alpha_0}^h) > B_\tau(u_{\alpha_0}^h)$ decrease $\alpha_0$ by setting $\alpha_0=c_{\alpha_0}  \alpha_0$, with $c_{\alpha_0}\in (0,1)$, and continue with step 1), otherwise stop and return $\alpha_0$.
\end{itemize}
\end{minipage}
}
%
%
%
%
%
%

\section{Total variation minimization}\label{Sec:TVmin}

In this section we are concerned with developing numerical methods for computing a minimizer of the discrete multi-scale $L^{\tau}$-TV model, i.e.,
\begin{equation}\label{minP1}
\min_{u^h\in X} \{J_\tau(u^h,\alpha):=H_\tau(u^h) + \tv_ \alpha(u^h)\}.
\end{equation}

\subsection{$L^2$-TV minimization}\label{SubSec:L2TV}
Here we consider the case $\tau=2$, i.e., the minimization problem
\begin{equation}\label{Eq:L2TVa}
\min_{u^h\in X} \frac{1}{2}\|T^h u^h - g^h\|_{2}^2 + \tv_\alpha(u^h),
\end{equation}
and present solution methods, first for the case $T^h=I$ and then for general linear bounded operators $T^h$.

\subsubsection{An algorithm for image denoising}\label{Sec:AlgfImagDen}
If $T^h=I$, then \eqref{Eq:L2TVa} becomes an image denoising problem, i.e., the minimization problem  

%
\begin{equation}\label{P:ImageDenoising}
\min_{u^h\in X} \| u^h - g^h \|_{2}^2 + 2 R_\alpha(u^h).
\end{equation} 
For solving this problem we use the algorithm of Chambolle and Pock \cite{ChaPoc}, which leads to the following iterative scheme:

\noindent
\fbox{
\begin{minipage}{7.9cm}
\textbf{Chambolle-Pock algorithm:}
Initialize $\tau, \sigma >0$, $\theta \in [0,1]$, $(\vec{p}_0^h,u_0^h)\in Y\times X$, set $\bar{u}_0^h=u_0^h$, and set $n=0$.
\begin{enumerate}
\item Compute 
$$
\vec p_{n+1}^h(x) =\frac{\vec p_n^h(x) + \sigma \nabla^h \bar{u}_n^h(x)}{\max\left\{\frac{1}{\alpha(x)} |\vec p_n^h(x) + \sigma \nabla^h \bar{u}_n^h(x)|,1\right\}},
$$
for all $x\in \Omega^h$.
\item Compute $u_{n+1}^h = \frac{u_n^h + \tau \Div^h \vec p_{n+1}^h + \tau g^h}{1+\tau}$.
\item Set $ \bar{u}_{n+1}^h= u_{n+1}^h + \theta (u_{n+1}^h - u_n^h)$.
\item Stop or set $n:=n+1$ and return to step 1).
\end{enumerate}

\end{minipage}
}\\

In our numerical experiments we choose $\theta =1$. In particular, in \cite{ChaPoc} it is shown that for $\theta=1$ and $\tau\sigma \|\nabla^h\|^2< 1$ the algorithm converges.

\subsubsection{An algorithm for linear bounded operators}

Assume, that $T^h$ is a linear bounded operator from $X$ to $X$, different to the identity $I$. Then instead of minimizing \eqref{Eq:L2TVa} directly, we introduce the surrogate functional
\begin{equation}\label{surrogateFct1}
\begin{split}
\mcS(u^h,a^h):&= \frac{1}{2}\|T^hu^h-g^h\|_{2}^2 + R_{\alpha}(u^h) + \frac{\delta}{2} \|u^h-a^h\|_{2}^2 \\
&\phantom{\frac{1}{2}\|T^hu^h-g^h\|_{2}^2}- \frac{1}{2} \| T^h(u^h-a^h)\|_{2}^2 \\
           &= \frac{\delta}{2} \|u^h - z(a^h)\|_{2}^2 +  R_{\alpha}(u^h) + \psi(a^h,g^h,T^h),
\end{split}
\end{equation}
with $a^h,u^h\in X$, $z(a^h)= a^h - \frac{1}{\delta}(T^h)^*(T^h a^h -g^h)$, $\psi$ a function independent of $u^h$, and where we assume $\delta > \|T^h\|^2$; see \cite{DauDefDeM,DauTesVes}. Note that
$$
\min_{u^h\in X} \mcS(u^h,a^h) \Leftrightarrow \min_{u\in X} \|u^h - z(a^h)\|_{2}^2 +  2 R_{\frac{\alpha}{\delta}}(u^h)
$$
and hence to obtain a minimizer amounts to solve a minimization problem of the type \eqref{P:ImageDenoising} and can be solved as described in Section \ref{Sec:AlgfImagDen}. Then an approximate solution of \eqref{Eq:L2TVa} can be computed by the following iterative algorithm: Choose $u_{0}^h \in X$ and iterate for $n\geq 0$
\begin{equation}\label{alg:surrogate}
u_{n+1}^h=\argmin_{u^h\in X} \mcS(u^h,u_n^h).
\end{equation}

%

For scalar $\alpha$ it is shown in \cite{ComWaj,DauDefDeM,DauTesVes} that this iterative procedure generates a sequence $(u_n^h)_n$ which converges to a minimizer of \eqref{Eq:L2TVa}. This convergence property can be easily extended to our non-scalar case yielding the following result.
\begin{theorem}
For $\alpha: \Omega \to \R^+$ the scheme in \eqref{alg:surrogate} generates a sequence $(u_n^h)_n$, which converges to a solution of \eqref{Eq:L2TVa} for any initial choice of $u_0^h\in X$.
\end{theorem}
\begin{proof}
A proof can be accomplished analogue to \cite{DauTesVes}. 
\qed
\end{proof}

\subsection{An algorithm for $L^1$-TV minimization}
The computation of a minimizer of 
\begin{equation}\label{Eq:L1TVa}
\min_{u^h\in X}  \|T^h u^h - g^h \|_{1} + \tv_\alpha(u^h),
\end{equation} 
is due to the non-smooth $\ell^1$-term in general more complicated than obtaining a solution of the $L^2$-TV model. Here we suggest to employ a trick, proposed in \cite{AujGilChaOsh} for $L^1$-TV minimization problems with a scalar regularization parameter, to solve \eqref{Eq:L1TVa} in two steps. In particular, we substitute the argument of the $\ell^1$-norm by a new variable $v$, penalize the functional by an $L^2$-term, which should keep the difference between $v$ and $Tu-g$ small, and minimize with respect to $v$ and $u$. That is, 
we replace the original minimization \eqref{Eq:L1TVa} by
\begin{equation}\label{reg_minP}
\min_{v^h,u^h\in X} \|v^h\|_{1} + \frac{1}{2 \gamma} \| T^h u^h - g^h - v^h \|_{2}^2 + R_{\alpha}(u^h),
\end{equation}
where $\gamma>0$ is small, so that we have $g^h\approx T^h u^h - v^h$. Actually, it can be shown that \eqref{reg_minP} converges to \eqref{Eq:L1TVa} as $\gamma\to 0$. In our experiments we actually choose $\gamma=10^{-2}$. This leads to the following alternating algorithm.

\noindent
\fbox{
\begin{minipage}{7.9cm}
\textbf{$L^1$-TV$_\alpha$ algorithm:} Initialize $\alpha >0$, $u_0^h \in X$ and set $n:=0$.
\begin{itemize}
\item[1)] Compute 
\begin{equation*}
v_{n+1}^h=\argmin_{v^h\in X} \|v^h\|_{1} + \frac{1}{2 \gamma} \| T^h u_{n}^h - g^h - v^h \|_{2}^2
\end{equation*}

\item[2)] 
Compute 
\begin{equation*}
u_{n+1}^h \in \argmin_{u^h\in X} \frac{1}{2 \gamma} \| T^h u^h - g^h - v_{n+1}^h \|_{2}^2 + R_{\alpha}(u^h)
\end{equation*} 

\item[3)] Stop or set $n:=n+1$ and return to step 1).
\end{itemize}

\end{minipage}
}\\

The minimizer $v_{n+1}^h$ in step 1) of the $L^1$-TV$_\alpha$ algorithm can be easily computed via a soft-thresholding, i.e., $v_{n+1}^h=\ST(T^h u_n^h - g^h,\gamma)$, where
$$
\ST(g^h,\gamma)(x) =
\begin{cases}
g^h(x) - \gamma  & \text{ if } g^h(x) > \gamma, \\
0                      & \text{ if } |g^h(x)| \leq \gamma, \\
g^h(x) + \gamma & \text{ if } g^h(x) < -\gamma \\
\end{cases}
$$
for all $x\in \Omega^h$. The minimization problem in step 2) is equivalent to 
\begin{equation}\label{equ:min}
\argmin_{u^h\in X} \frac{1}{2}\| T^h u^h - g^h - v_{n+1}^h \|_{2}^2 + R_{\gamma\alpha}(u^h)
\end{equation}
and hence is of the type \eqref{Eq:L2TVa}. Thus an approximate solution of \eqref{equ:min} can be computed as described above; see Section \ref{SubSec:L2TV}. 

\begin{theorem}
The sequence $(u_n^h,v_n^h)_n$ generated by the $L^1$-TV$_\alpha$ algorithm converges to a minimizer of  \eqref{reg_minP}.
\end{theorem}

\begin{proof}
The statement can be shown analogue to \cite{AujGilChaOsh}.
\qed
\end{proof}

\subsection{A primal-dual method for $L^1$-TV minimization}

For solving \eqref{LtauTVmodel} with $\tau=1$ we suggest, alternatively to the above method, to use the primal-dual method of \cite{HinRin} adapted to our setting, where a Huber regularisation of the gradient of $u$ is considered; see \cite{HinRin} for more details.
Denoting by $\bar{u}$ a corresponding solution of the primal problem and $\bar{\vec{p}}$ the solution of the associated dual problem, the optimality conditions due to the Fenchel theorem \cite{EkeTem} are given by
\begin{eqnarray*}
&&-\Div \bar{\vec{p}}(x) = -\kappa \Delta \bar{u}(x) + \frac{1}{\beta+\mu} T^*(T\bar{u}(x) - g(x) ) \\
&& \phantom{-\Div \bar{\vec{p}}(x) = -\kappa}+ \frac{\mu}{\beta+\mu} T^* \frac{T\bar{u}(x)  - g(x)}{\max\{\beta, |T \bar{u}(x) - g(x) |\}} \\
&&- \bar{\vec{p}}(x) =\frac{1}{\gamma} \nabla \bar{u}(x) \quad \text{if } |\bar{\vec{p}}(x)|_{l^2} <\alpha(x)\\
&&- \bar{\vec{p}}(x) =\alpha(x) \frac{\nabla \bar{u}(x)}{|\nabla \bar{u}(x)|_{l^2}} \quad \text{if } |\bar{\vec{p}}(x)|_{l^2} =\alpha(x),
\end{eqnarray*}
for all $x\in \Omega$, where $\kappa, \beta, \mu$, and $\gamma$ are fixed positive constants. The latter two conditions can be summarized to $-\bar{\vec{p}}(x) = \frac{\alpha(x) \nabla \bar{u}(x)}{\max\{\gamma \alpha(x),  |\nabla \bar{u}(x)|_{l^2} \}}$. Then setting $\bar{\vec{q}} = - \bar{\vec{p}}$ and $\bar{v} = \frac{T\bar{u}  - g}{\max\{\beta, |T \bar{u} - g |\}}$ leads to the following system of equation:
\begin{equation}\label{Eq:Op:cond:pd}
\begin{split}
&0 = -\max\{\beta, |T \bar{u}(x) - g(x) |\} \bar{v} + T\bar{u}(x)  - g(x) \\
&0=\Div \bar{\vec{q}}(x) + \kappa \Delta \bar{u}(x) - \frac{1}{\beta+\mu} T^*(T\bar{u}(x) - g(x) ) \\
&\phantom{0=} - \frac{\mu}{\beta+\mu} T^* \bar{v}(x)  \\
&0=\max\{\gamma \alpha(x),  |\nabla \bar{u}(x)|_{l^2} \} \bar{\vec{q}}(x) - \alpha(x) \nabla \bar{u}(x)
\end{split}
\end{equation}
for all $x\in \Omega$. This system can be solved efficiently by a semi-smooth Newton algorithm; see Appendix \ref{Appendix:Semismooth} for a description of the method and for the choice of the parameters $\kappa, \beta, \mu$, and $\gamma$.


Note, that different algorithms presented in the literature can also be adjusted to the case of a locally varying regularization parameter, such as \cite{Cha,ChaPoc2015,LorPoc}. However, it is not the scope of this paper to compare different algorithms in order to detect the most efficient one, although this is an interesting research topic in its own right.

\section{Numerical experiments}\label{Sec:Numerics}

In the following we present numerical experiments for studying the behavior of the proposed algorithms (i,e., APS-, pAPS-, LATV-, pLATV-algorithm) with respect to its image restoration capabilities and its stability concerning the choice of the initial value $\alpha_0$. The performance of these methods is compared quantitatively by means of the peak signal-to-noise-ratio (PSNR) \cite{Bov}, which is widely used as an image quality assessment measure, and the structural similarity measure (MSSIM) \cite{WanBovSheSim}, which relates to perceived visual quality better than PSNR. When an approximate solution of the $L^1$-TV model is calculated, we also compare the restorations by the mean absolute error (MAE), which is an $L^1$-based measure defined as
$$
\operatorname{MAE} = {\| u - \hat{u} \|_{L^1(\Omega)}}, 
$$
where $\hat{u}$ denotes the true image and $u$ represents the obtained restoration. In general, when comparing PSNR and MSSIM, large values indicate better reconstruction than smaller values, while the smaller MAE becomes the better the reconstruction results are.

Whenever an image is corrupted by Gaussian noise we compute a solution by means of the (multi-scale) $L^2$-TV model, while for images containing impulsive noise the (multi-scale) $L^1$-TV model is always considered.

In our numerical experiments the CPS-, APS-, and pAPS-algorithm are terminated as soon as 
$$
\frac{| \H_\tau(u_{\alpha_n};g) - \Bt(u_{\alpha_n}) |}{\Bt(u_{\alpha_n})} \leq \epsilon_B = 10^{-5}
$$
or the norm of the difference of two successive iterates $\alpha_{n}$ and $\alpha_{n+1}$ drops below the threshold $\epsilon_\alpha=10^{-10}$, i.e., $\| \alpha_{n} - \alpha_{n+1}\|< \epsilon_\alpha$. The latter stopping criterion is used to terminate the algorithms if $(\alpha_n)_n$ stagnates and only very little progress is to expect. In fact, if our algorithm converges at least linearly, i.e., if there exists an $\varepsilon_\alpha \in (0,1)$ and an $m>0$ such that for all $n\geq m$ we have $\| \alpha_{n+1} - \alpha_{\infty}\|< \varepsilon_\alpha \| \alpha_{n} - \alpha_{\infty}\|$, the second stopping criterion at least ensures that the distance between our obtained result $\alpha$ and $\alpha_{\infty}$ is $\| \alpha - \alpha_{\infty}\| \leq \frac{\epsilon_\alpha \varepsilon_\alpha}{1-\varepsilon_\alpha}$.


\subsection{Automatic scalar parameter selection}
For automatically selecting the scalar parameter $\alpha$ in \eqref{LtauTVmodel} we presented in Section \ref{Sec:APS} the APS- and pAPS-algorithm. Here we compare their performance for image denoising and image deblurring.

\subsubsection{Gaussian noise removal}

For recovering images corrupted by Gaussian noise with mean zero and standard deviation $\sigma$ we minimize the functional in \eqref{LtauTVmodel} by setting $\tau=2$ and $T=I$. Then $\B_2=\frac{\sigma^2}{2}|\Omega|$ is a constant independent of $u$. The automatic selection of a suitable regularization parameter $\alpha$ is here performed by the CPS-, APS-, and pAPS-algorithm, where the contained minimization problem is solved by the method presented in Section \ref{Sec:AlgfImagDen}. We recall, that by \cite[Theorem 4]{Cha} and Theorem \ref{Theorem:conv2} it is ensured that the CPS- and the pAPS-algorithm generate sequences $(\alpha_n)_n$ which converge to $\bar{\alpha}$ such that $u_{\bar{\alpha}}$ solves \eqref{conLtauTV}. In particular, in the pAPS-algorithm the value $p$ is chosen in dependency of $\alpha$, i.e., $p=p(\alpha)$, such that $\alpha \to \frac{(\H_\tau(u_\alpha;g))^{p(\alpha)}}{\alpha}$ is non-increasing, see Fig. \ref{fig_phantom1_decr}(a). This property is fundamental for obtaining convergence of this algorithm; see Theorem \ref{Theorem:conv2}. For the APS-algorithm such a monotonic behavior is not guaranteed and hence we cannot ensure its convergence. Nevertheless, if the APS-algorithm generates $\alpha$'s such that the function $\alpha \to \frac{\H_\tau(u_\alpha;g)}{\alpha}$ is non-increasing, then it indeed converges to the desired solution, see Theorem \ref{Theorem:conv}. Unfortunately, the non-increase of the function $\alpha \to \frac{\H_\tau(u_\alpha;g)}{\alpha}$ does not hold always, see Fig. \ref{fig_phantom1_decr}(b).

\graphicspath{{./graphics/}}
\begin{figure}[htbp!]
\begin{center}
\hspace{0cm}
    \subfigure[phantom]{\includegraphics[height=4.1cm]{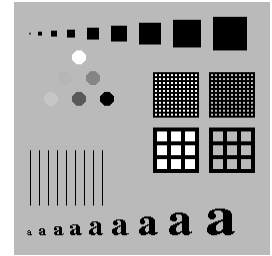}}
    \subfigure[cameraman]{\includegraphics[height=4.1cm]{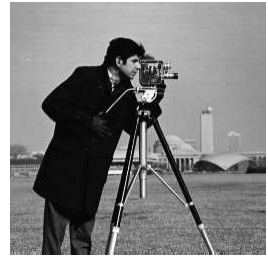}}
    \subfigure[barbara]{\includegraphics[height=4.1cm]{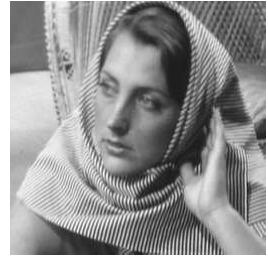}}
    \subfigure[lena]{\includegraphics[height=4.1cm]{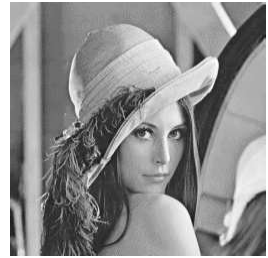}}
\end{center}    
\caption{\small \it Original images.  }
\label{fig_org1}
\end{figure}

\graphicspath{{./graphics/}}
\begin{figure}[htbp!]
\begin{center}
\hspace{0cm}
    \subfigure[]{\includegraphics[height=4.8cm]{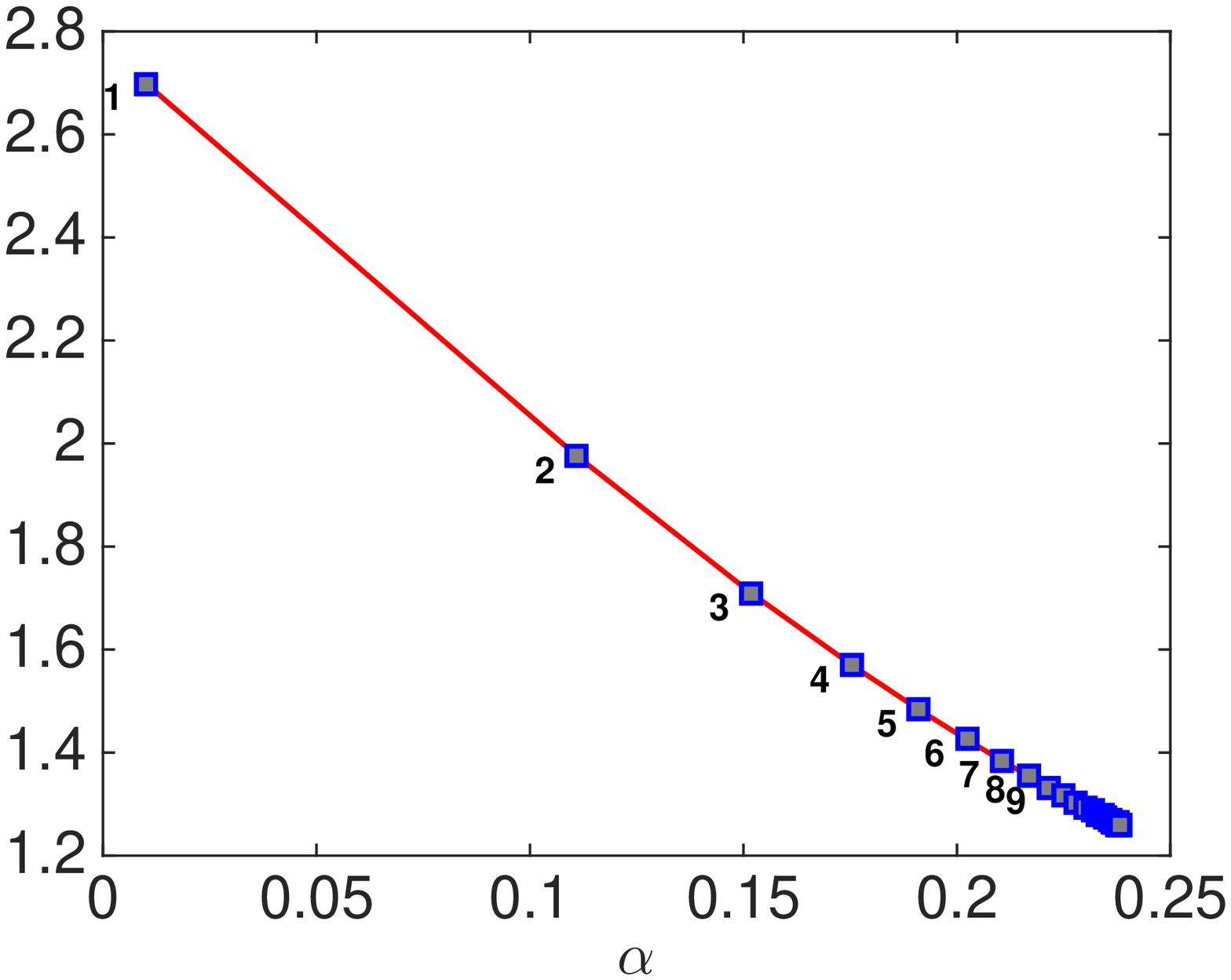}}
    \subfigure[]{\includegraphics[height=4.8cm]{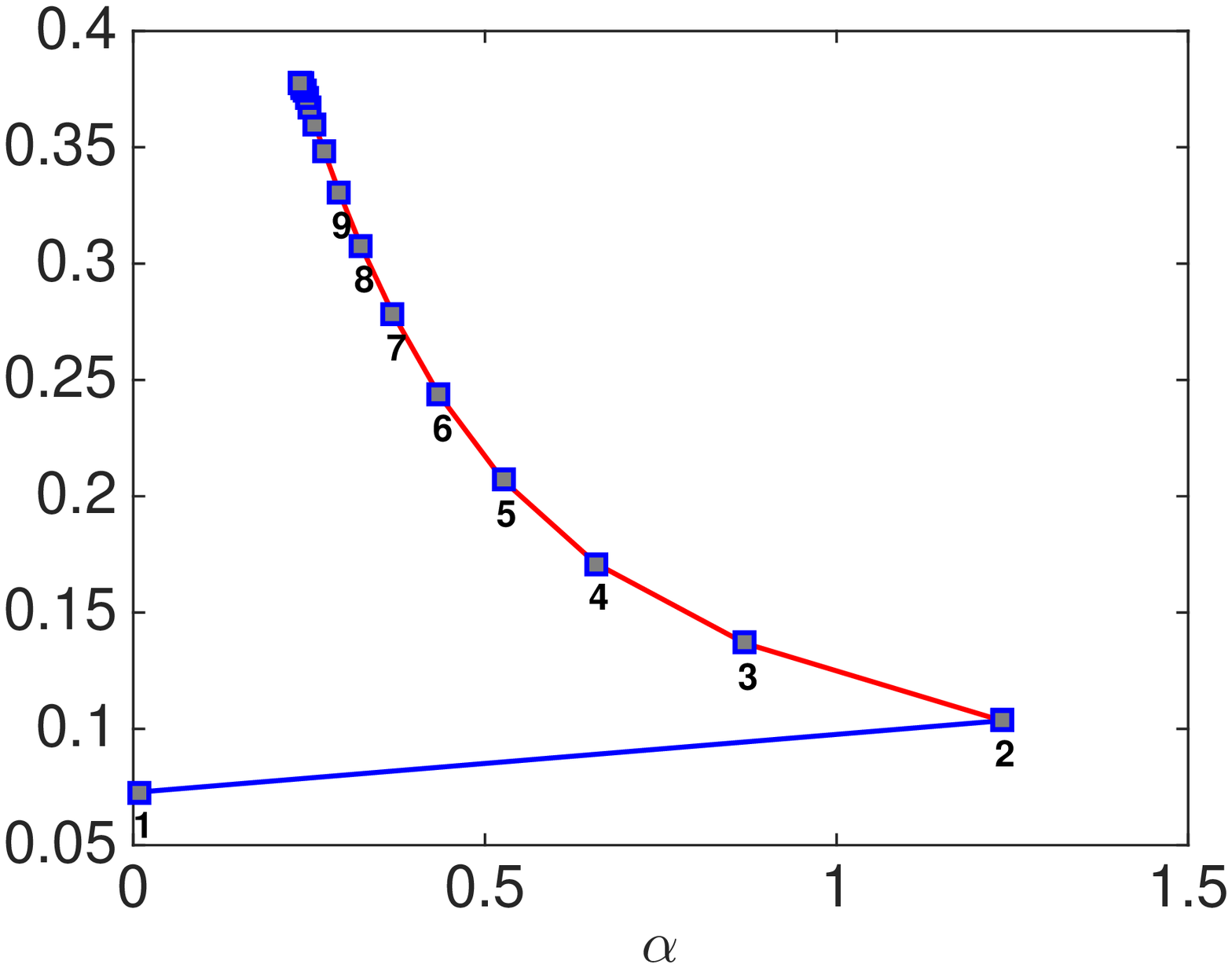}}
\end{center}    
\caption{\small \it Denoising of the phantom-image corrupted with Gaussian white noise with $\sigma=0.03$. (a) Plot of the function $\alpha \to \frac{\H_\tau(u_\alpha;g)^p}{\alpha}$ of the pAPS-algorithm with $\alpha_0=10^{-2}$. (b) Plot of the function $\alpha \to \frac{\H_\tau(u_\alpha;g)}{\alpha}$ of the APS-algorithm with $\alpha_0=10^{-2}$. The desired monotone behavior is observed after the second iteration (red part of the curve).  }
\label{fig_phantom1_decr}
\end{figure}

For a performance-comparison here we 
 consider the phantom-image of size $256 \times 256$ pixels, see Fig. \ref{fig_org1}(a), corrupted only by Gaussian white noise with $\sigma=0.3$. 
 In the pAPS-algorithm we set $p_0=32$. The behavior of the sequence $(\alpha_n)_n$ for different initial $\alpha_0$, i.e., $\alpha_0 \in \{1, 10^{-1}, 10^{-2}\}$ is depicted in Fig. \ref{fig_alpha_plot1}.  We observe that all three methods for arbitrary $\alpha_0$ converge to the same regularization parameter $\alpha$ and hence generate results with the same PSNR and MSSIM (i.e., PSNR$=19.84$ and MSSIM$=0.7989$). Hence, in these experiments, despite the lack of theoretical convergence, also the APS-algorithm seems to converge to the desired solutions. We observe the same behavior for different $\sigma$ as well.  

\graphicspath{{./graphics/}}
\begin{figure*}[htbp!]
\begin{center}
\hspace{0cm}
    \subfigure[CPS-algorithm]{\includegraphics[height=4.2cm]{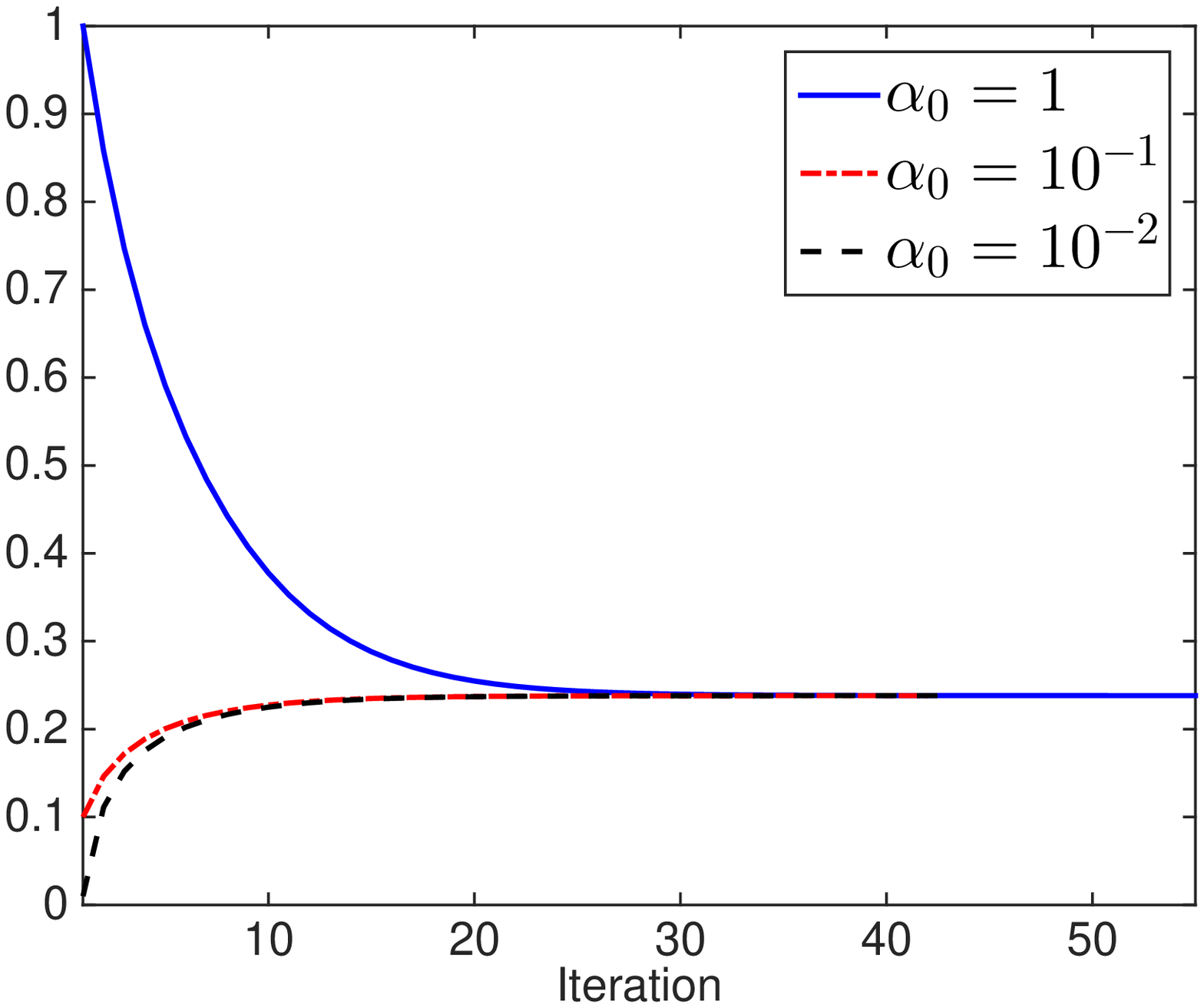}}
    \subfigure[APS-algorithm]{\includegraphics[height=4.2cm]{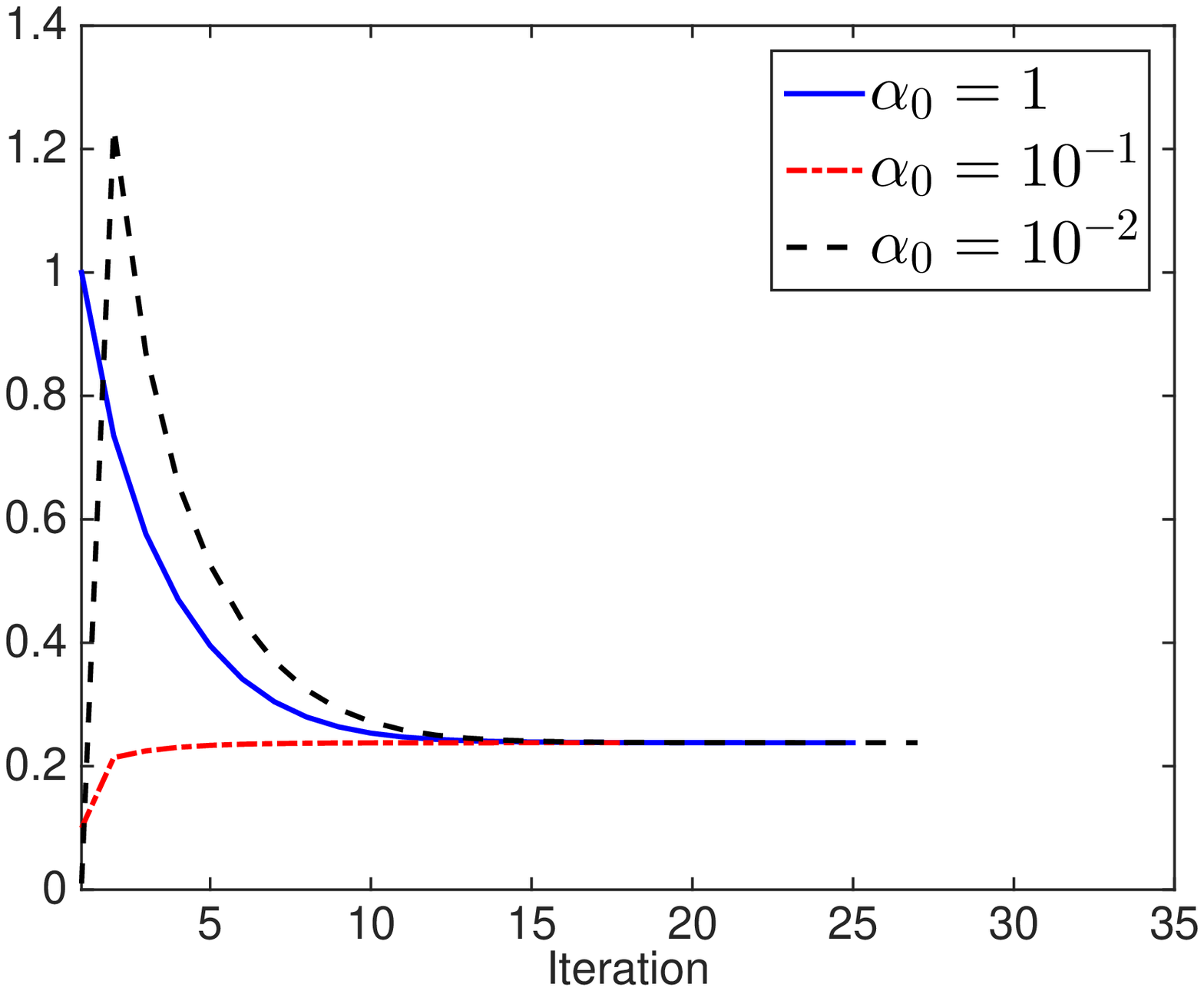}}
    \subfigure[pPAS-algorithm]{\includegraphics[height=4.2cm]{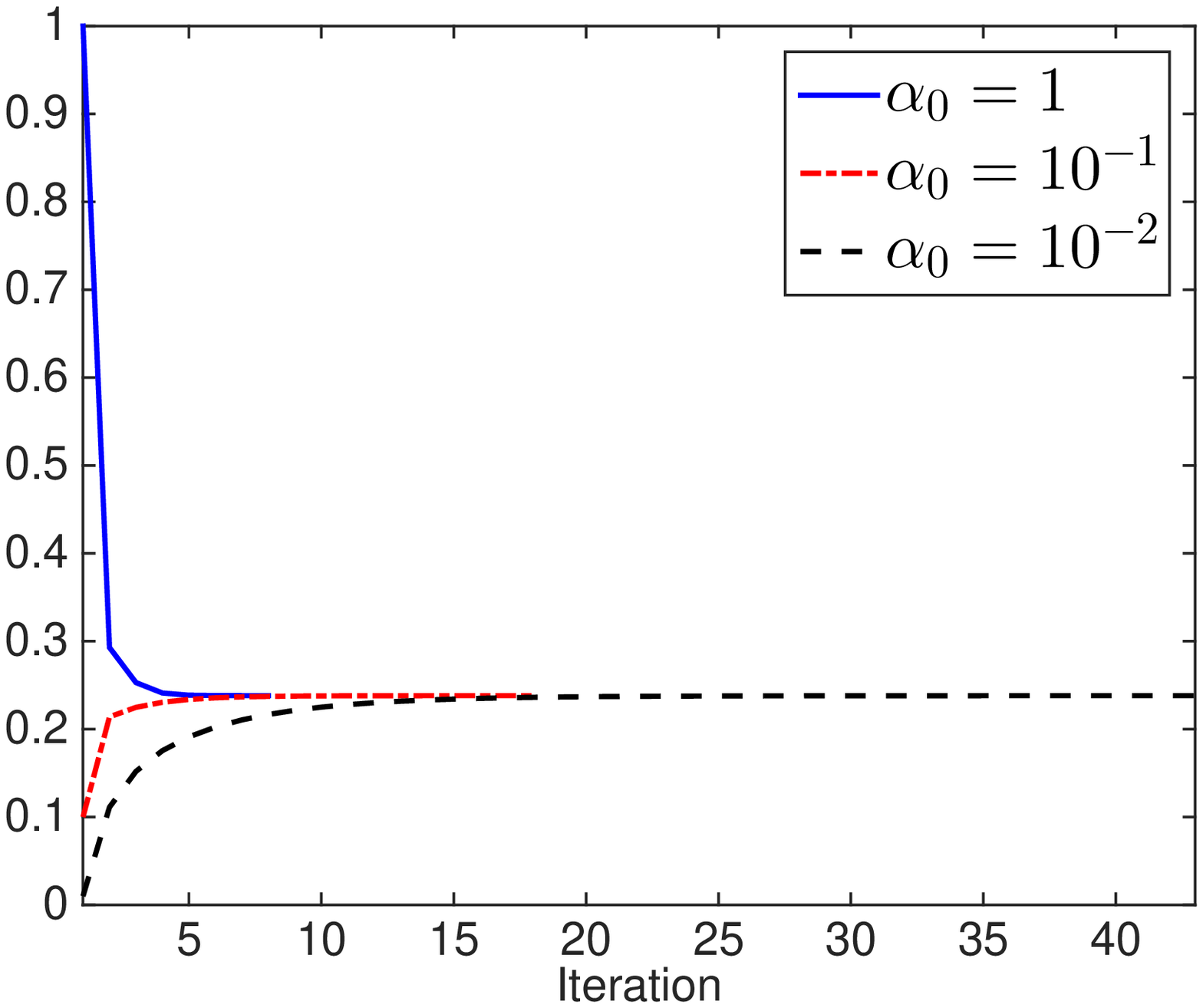}}
\end{center}    
\caption{\small \it Denoising of the phantom-image corrupted with Gaussian white noise with $\sigma=0.03$.}
\label{fig_alpha_plot1}
\end{figure*}

\begin{table*}
\footnotesize
\caption{\small \it Number of iterations needed for the reconstruction of the phantom-image corrupted by Gaussian white noise with different standard deviations $\sigma$. In the pAPS-algorithm we set $p_0=32$.}
\label{table_phantom1_scalar:1}
\begin{center}
\begin{tabular}{c||c|c|c||c|c|c||c|c|c||c|c|c}
\toprule
\multicolumn{1}{c||}{}  &  \multicolumn{3}{|c||}{$\sigma=0.3$} & \multicolumn{3}{c||}{$\sigma=0.1$} & \multicolumn{3}{c||}{$\sigma=0.05$} & \multicolumn{3}{c}{$\sigma=0.01$} \\
$\alpha_0$ & CPS & APS & pAPS  & CPS & APS & pAPS  & CPS & APS & pAPS & CPS & APS & pAPS\\ \hline
$1$          &  $55$ & $25$ & $8$  & $47$ & $21$ & $21$ & $46$ & $20$ & $20$ & $47$ & $17$ & $47$ \\
$10^{-1}$ & $42$ & $18$ & $18$ & $38$ & $17$ & $6$   & $44$ & $20$ & $20$ & $46$ & $18$ & $46$  \\
$10^{-2}$ & $43$ & $27$ & $43$ & $40$ & $20$ & $40$ & $40$ & $18$ & $40$ & $39$ & $17$ & $6$ \\
$10^{-3}$ & $43$ & $31$ & $43$ & $41$ & $22$ & $41$ & $41$ & $19$ & $41$ & $41$ & $20$ & $41$ \\
$10^{-4}$ & $43$ & $35$ & $43$ & $41$ & $23$ & $41$ & $41$ & $22$ & $41$ & $41$ & $19$ & $41$ \\ \hline 
\end{tabular}

\end{center}
\end{table*}

Looking at the number of iterations needed till termination, we observe from Table \ref{table_phantom1_scalar:1} that the APS-algorithm always needs significantly less iterations than the CPS-algorithm till termination. This is attributed to the different updates of $\alpha$. Recall, that for a fixed $\alpha_n$ in the CPS-algorithm we set $\alpha^{CPS}_{n+1}:=\sqrt{\frac{\nu_2 |\Omega|}{\tau \H_2(u_{\alpha_n};g)}}\alpha_n$, while in the APS-algorithm the update is performed as $\alpha^{APS}_{n+1}:=\frac{\nu_2 |\Omega|}{\tau \H_2(u_{\alpha_n};g)}\alpha_n$. Note, that $ \sqrt{\frac{\nu_2 |\Omega|}{\tau \H_2(u_{\alpha_n};g)}} \leq \frac{\nu_2 |\Omega|}{\tau \H_2(u_{\alpha_n};g)} $, if $\frac{\nu_2 |\Omega|}{\tau \H_2(u_{\alpha_n};g)}\geq 1$ and $ \sqrt{\frac{\nu_2 |\Omega|}{\tau \H_2(u_{\alpha_n};g)}} \geq \frac{\nu_2 |\Omega|}{\tau \H_2(u_{\alpha_n};g)} $, if $\frac{\nu_2 |\Omega|}{\tau \H_2(u_{\alpha_n};g)}\leq 1$. Hence, we obtain $\alpha_n \leq \alpha^{CPS}_{n+1} \leq \alpha^{APS}_{n+1}$ if $\frac{\nu_2 |\Omega|}{\tau \H_2(u_{\alpha_n};g)}\geq 1$ and $\alpha_n\geq \alpha^{CPS}_{n+1} \geq \alpha^{APS}_{n+1}$ if $\frac{\nu_2 |\Omega|}{\tau \H_2(u_{\alpha_n};g)}\leq 1$. That is, in the APS-algorithm $\alpha$ changes more significantly in each iteration than in the CPS-algorithm, which leads to a faster convergence with respect to the number of iterations. Nevertheless, exactly this behavior allows the function $\alpha \to \frac{\H_2(u_\alpha;g)}{\alpha}$ to increase which is responsible that the convergence of the APS-algorithm is not guaranteed in general. However, in our experiments we observed that the function $\alpha \to \frac{\H_2(u_\alpha;g)}{\alpha}$ only increases in the first iterations, but non-increases (actually even decreases) afterwards, see Fig. \ref{fig_phantom1_decr}(b). This is actually enough to guarantee convergence, as discussed in Section \ref{Sec:APS}, since we can consider the solution of the last step in which the desired monotonic behavior is not fulfilled as a ``new'' initial value. Since from this point on the non-increase holds, we get convergence of the algorithm. 

The pAPS-algorithm is designed to ensure the non-increase of the function $\alpha \to \frac{(\H_\tau(u_\alpha;g))^{p(\alpha)}}{\alpha}$ by choosing $p(\alpha)$ in each iteration accordingly, which is done by the algorithm automatically. If $p(\alpha)=p=1/2$ in each iteration, then the pAPS-algorithm becomes the CPS-algorithm, as it happens sometimes in practice (indicated by the same number of iterations in Table \ref{table_phantom1_scalar:1}). Since the CPS-algorithm converges \cite{Cha}, the pAPS-algorithm always yields $p\geq 1/2$. In particular, we observe that if the starting value $\alpha_0$ is larger than the requested regularization parameter $\alpha$, less iteration till termination are needed than with the CPS-algorithm. On the contrary, if $\alpha_0$ is smaller than the desired $\alpha$, $p=1/2$ is chosen by the algorithm to ensure the monotonicity. The obtained result of the pAPS-algorithm is independent on the choice of $p_0$ as visible from Fig. \ref{fig_p_test}. In this plot we also specify the number of iterations needed till termination. On the optimal choice of $p_0$ with respect to the number of iterations, we conclude from Fig. \ref{fig_p_test} that $p_0=32$ seems to do a good job, although the optimal value may depend on the noise-level. 

\graphicspath{{./graphics/}}
\begin{figure}[htbp!]
\begin{center}
\hspace{0cm}
    \includegraphics[height=4.8cm]{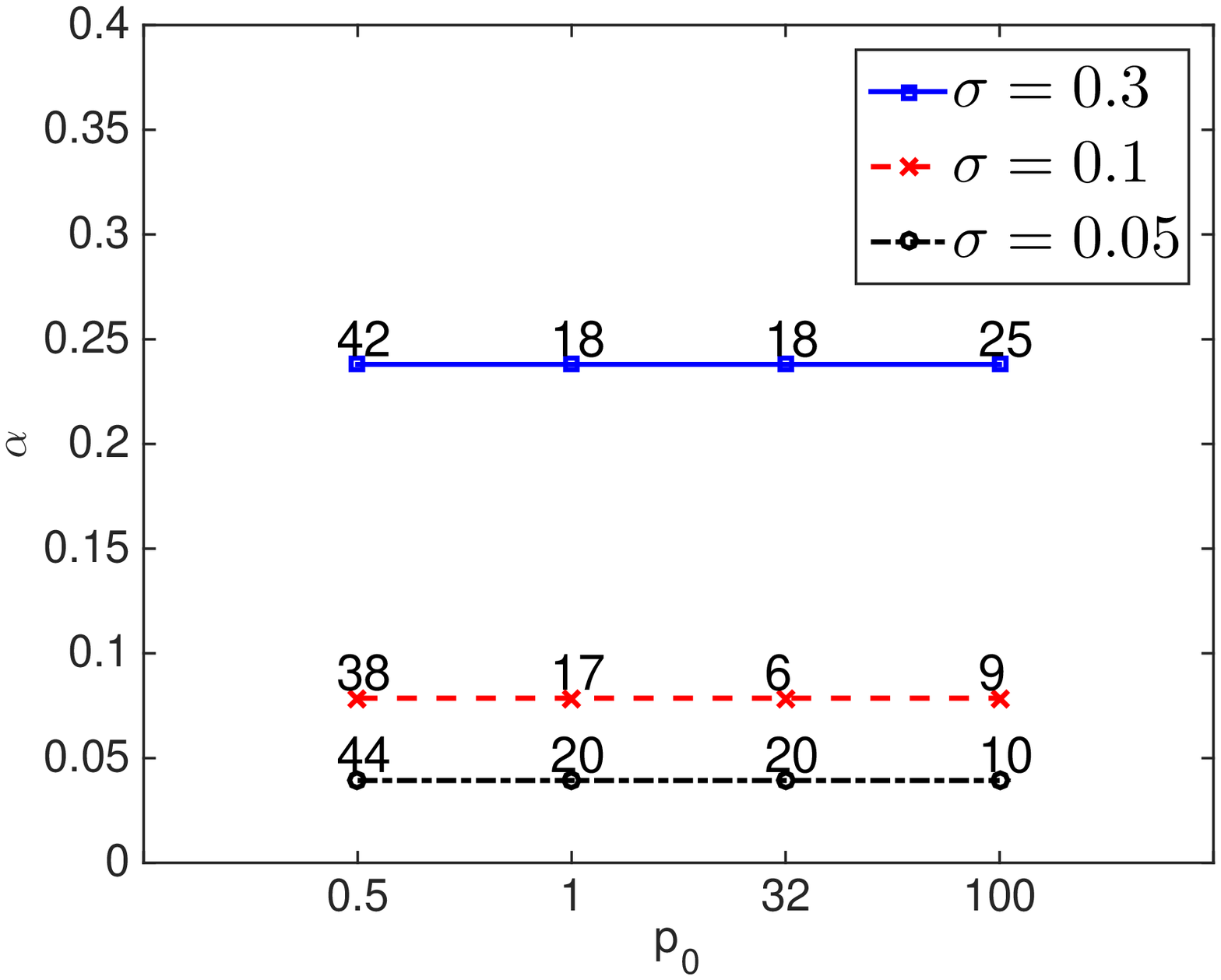}
\end{center}    
\caption{\small \it Regularization parameter $\alpha$ obtained by the pAPS-algorithm with different $p_0$ for denoising the phantom-image corrupted with Gaussian white noise for different $\sigma$.}
\label{fig_p_test}
\end{figure}

Similar behaviors as described above are also observed for denoising other and real images as well.


\subsubsection{Image deblurring}

Now, we consider the situation when an image is corrupted by some additive Gaussian noise and additionally blurred. Then the operator $T$ is chosen according to the blurring kernel, which we assume here to be known. For testing the APS- and pAPS-algorithm in this case we take the cameraman-image of Fig. \ref{fig_org1}(b), which is of size $256\times 256$ pixels, blur it by a Gaussian blurring kernel of size $5\times 5$ pixels and standard deviation $10$ and additionally add some Gaussian white noise with variance $\sigma^2$. The minimization problem in the APS- and pAPS-algorithm is solved approximately by the algorithm in \eqref{alg:surrogate}. In Fig. \ref{fig_alpha_G} the progress of $\alpha_n$ for different $\sigma$'s, i.e., $\sigma \in \{0.3, 0.1, 0.05\}$, and different $\alpha_0$'s, i.e., $\alpha_0 \in  \{1, 10^{-1}, 10^{-2}\}$ are presented. In these tests both algorithms converge to the same regularization parameter and minimizer. From the figure we observe, that the pAPS-algorithm needs much less iterations than the APS-algorithm till termination. This behavior might be attributed to the choice of the power $p$ in the pAPS-algorithm, since we observe in all our experiments that $p>1$ till termination.

\graphicspath{{./graphics/}}
\begin{figure*}[htbp!]
\begin{center}
\hspace{0cm}\subfigure[$\sigma=0.3$; pAPS-algorithm]{\includegraphics[height=4.2cm]{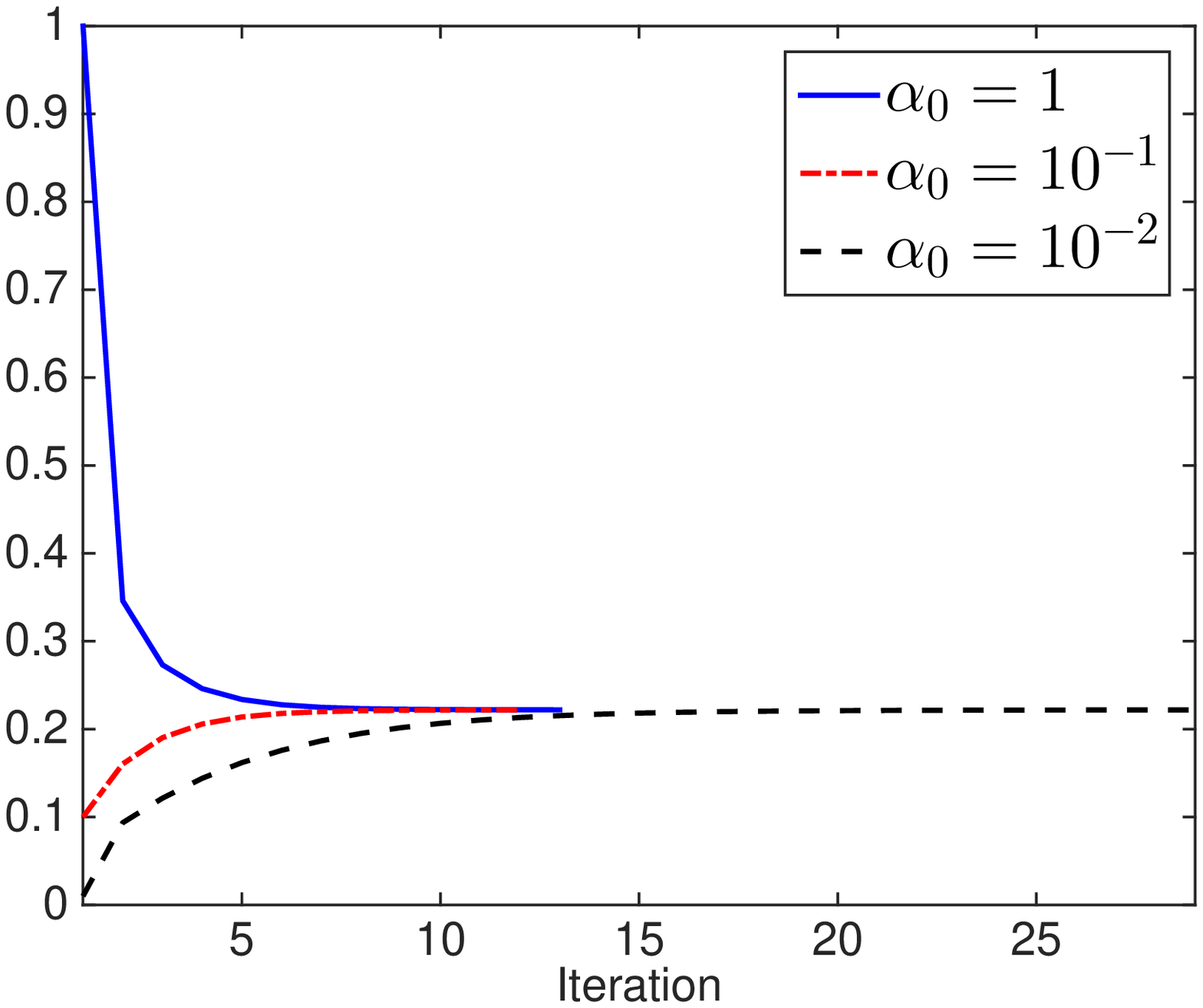}}
\subfigure[$\sigma=0.1$; pAPS-algorithm]{\includegraphics[height=4.2cm]{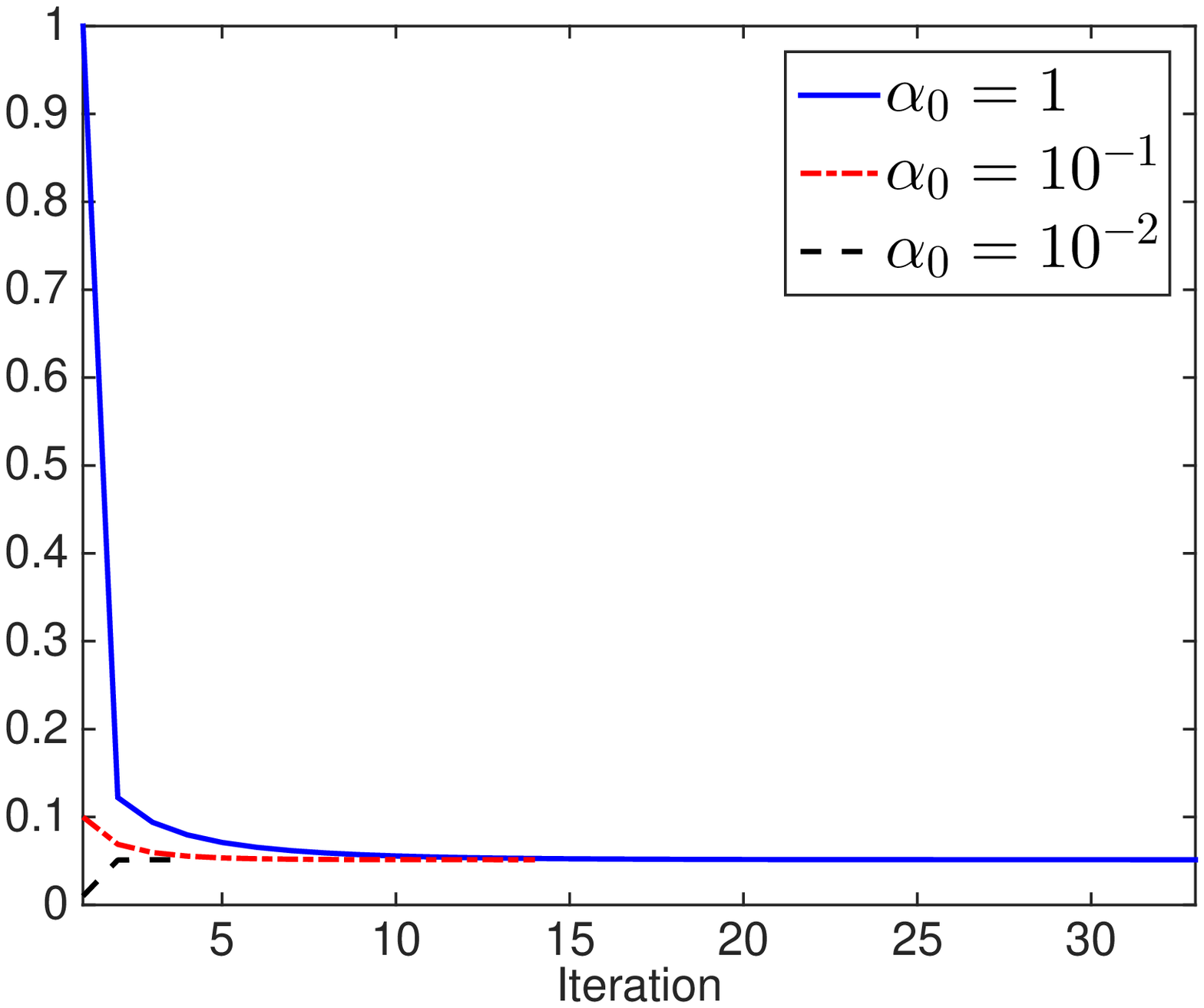}}
\subfigure[$\sigma=0.05$; pAPS-algorithm]{\includegraphics[height=4.2cm]{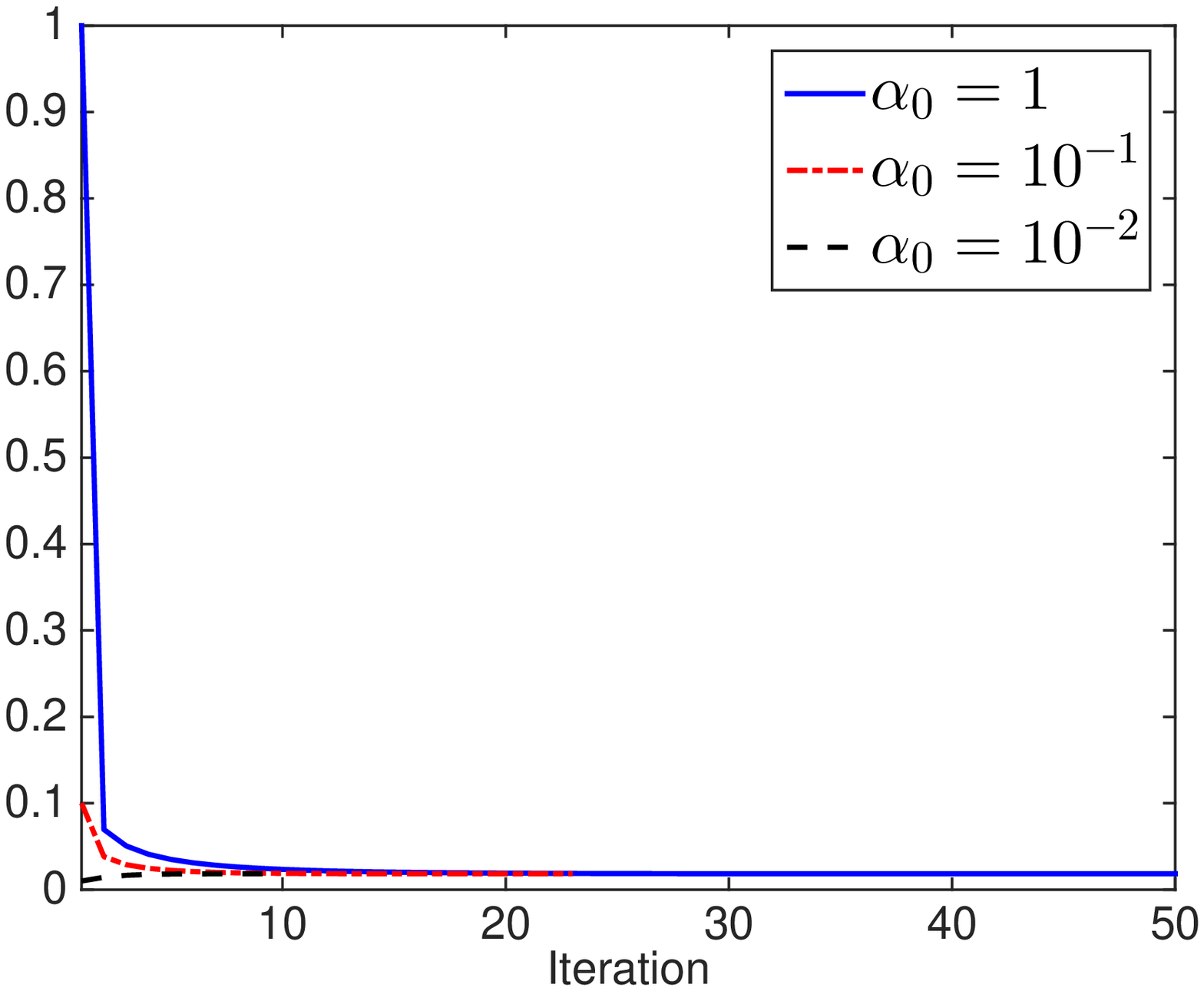}}
\subfigure[$\sigma=0.3$; APS-algorithm]{\includegraphics[height=4.2cm]{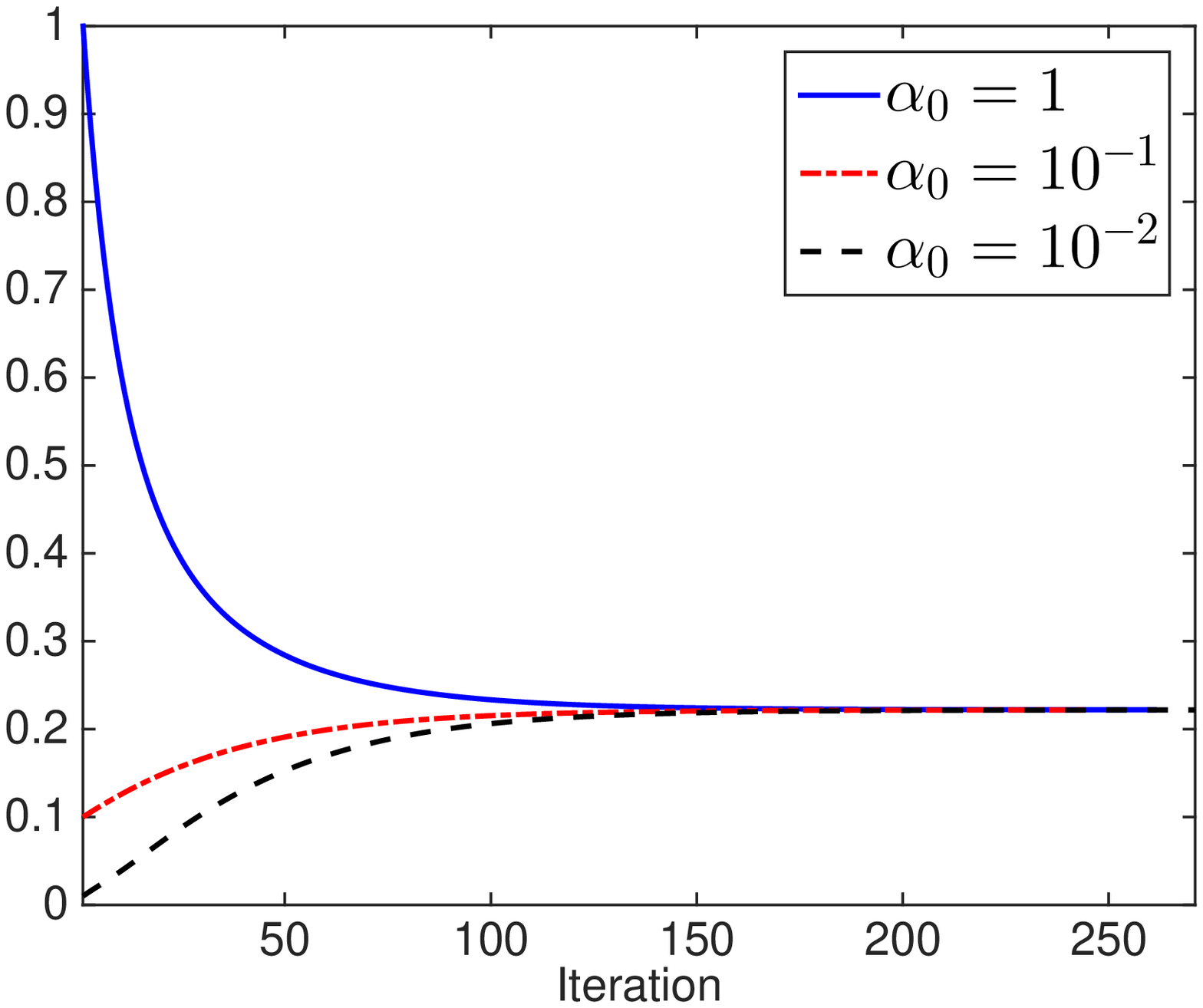}}
\subfigure[$\sigma=0.1$; APS-algorithm]{\includegraphics[height=4.2cm]{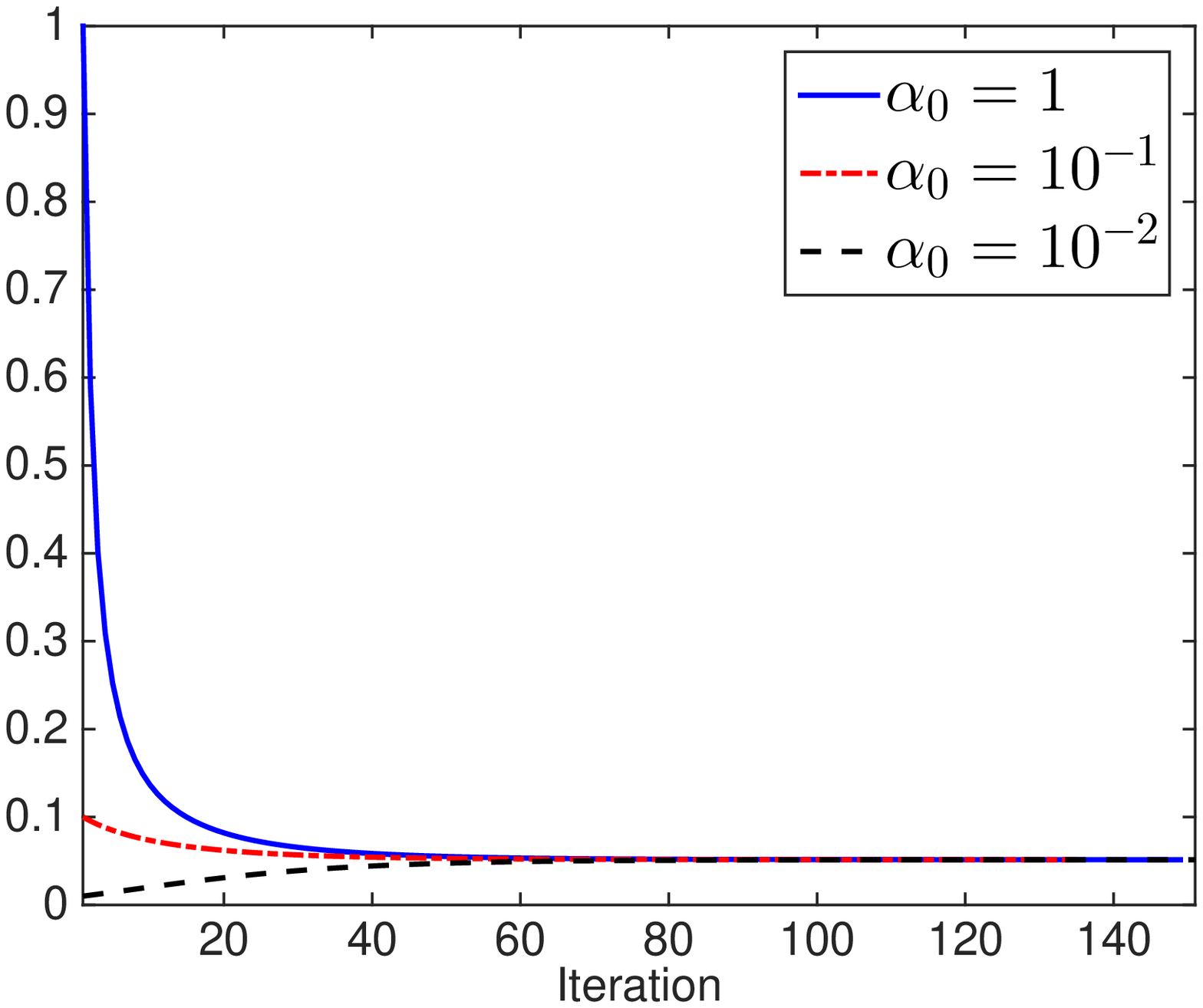}}
\subfigure[$\sigma=0.05$; APS-algorithm]{\includegraphics[height=4.2cm]{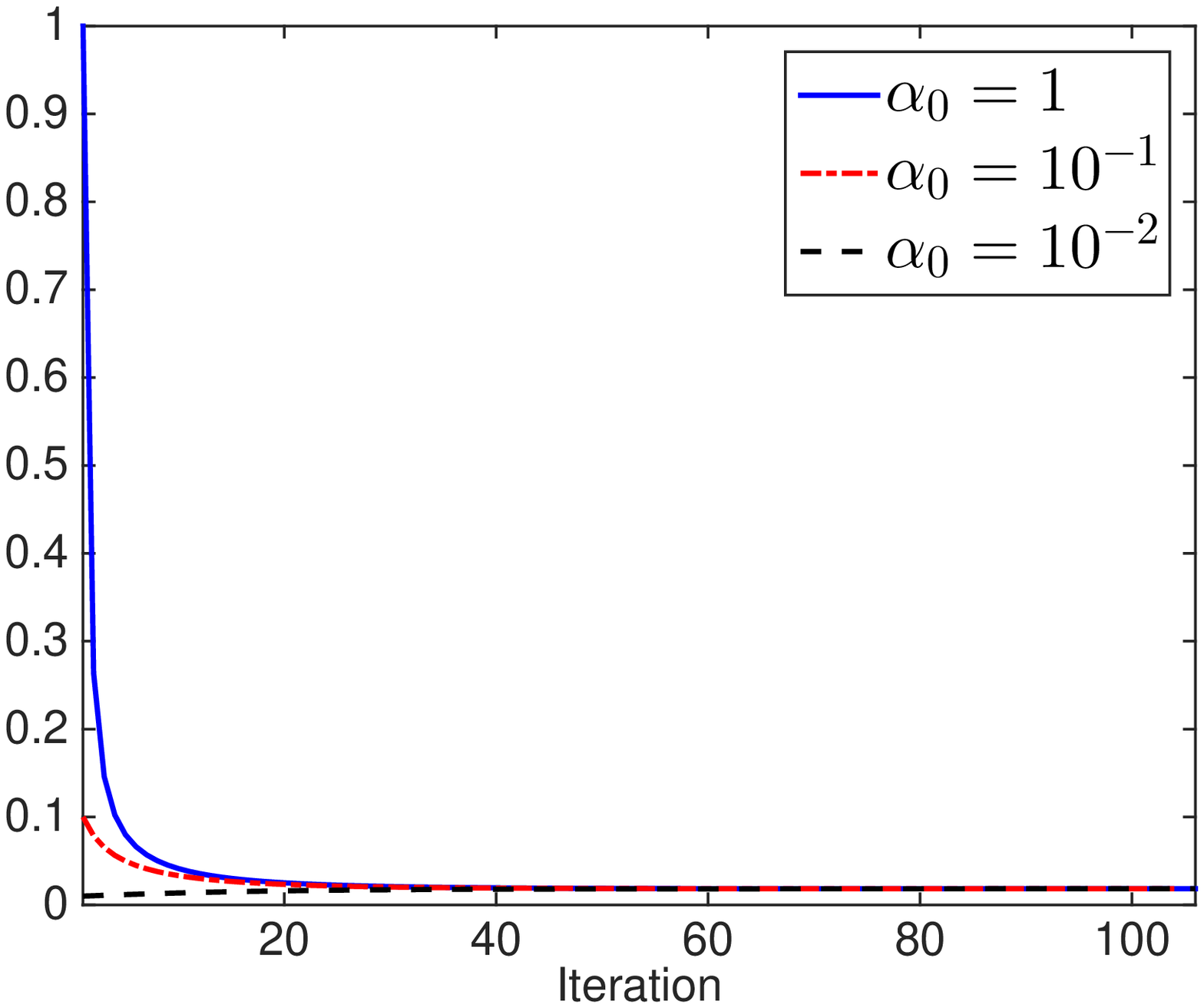}}
\end{center}    
\caption{\small \it Reconstruction of the cameraman-image corrupted by Gaussian blurring kernel of size $5\times 5$ and standard deviation $10$. In the pAPS-algorithm we set $p_0=32$.}
\label{fig_alpha_G}
\end{figure*}

\subsubsection{Impulsive noise removal}

It has been demonstrated that for removing impulsive noise in images one should minimize the $L^1$-TV model rather than the $L^2$-TV model. Then for calculating a suitable regularization parameter $\alpha$ in the $L^1$-TV model we use the APS-  and pAPS-algorithm, in which the minimization problems are solved approximately by the $L^1$-TV$_\alpha$-algorithm. Here, we consider the cameraman-image corrupted by salt-and-pepper noise or random-valued impulse noise with different noise-levels, i.e., $r_1=r_2 \in \{0.3, 0.1, 0.05\}$ and $r \in \{0.3, 0.1, 0.05\}$ respectively. The obtained results for different $\alpha_0$'s are depicted in Fig. \ref{fig_alpha_SP} and Fig. \ref{fig_alpha_RV}. For the removal of salt-and-pepper noise we observe from Fig. \ref{fig_alpha_SP} similar behaviors of the APS- and pAPS-algorithm as above for removing Gaussian noise. In particular, both algorithms converge to the same regularization parameter. However, in many cases the APS-algorithm needs significantly less iterations than the pAPS-algorithm. These behaviors are also observed in Fig.~\ref{fig_alpha_RV} for removing random-valued impulse noise as long as the APS-algorithm finds a solution. In fact, for $r=0.05$ it actually does not converge but oscillates as depicted in Fig.~\ref{fig_alpha_RVc}. 

\graphicspath{{./graphics/}}
\begin{figure*}[htbp!]
\begin{center}
\hspace{0cm}\subfigure[$r_1=r_2=0.3$; pAPS-algorithm]{\includegraphics[height=4.2cm]{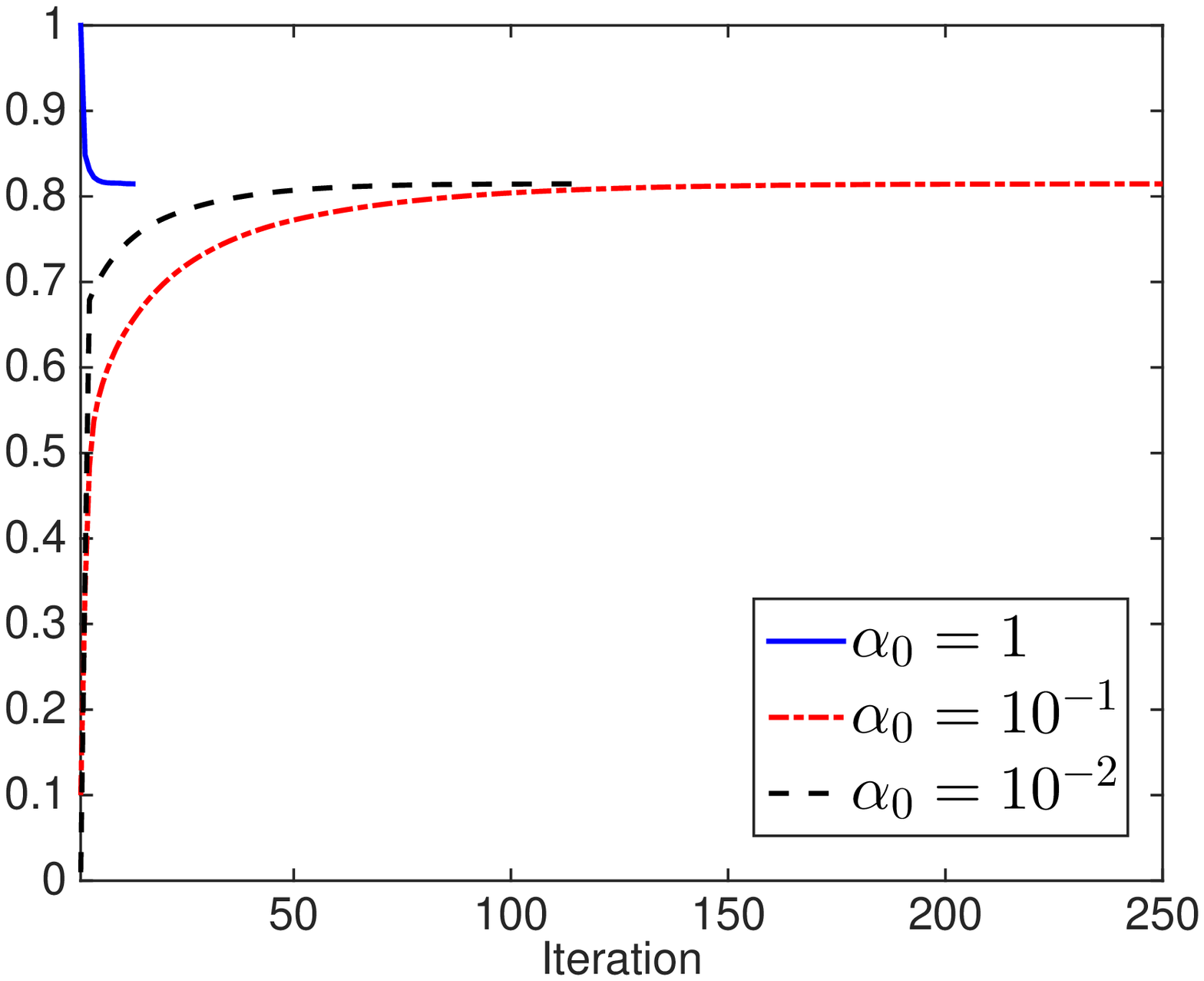}}
\subfigure[$r_1=r_2=0.1$; pAPS-algorithm]{\includegraphics[height=4.2cm]{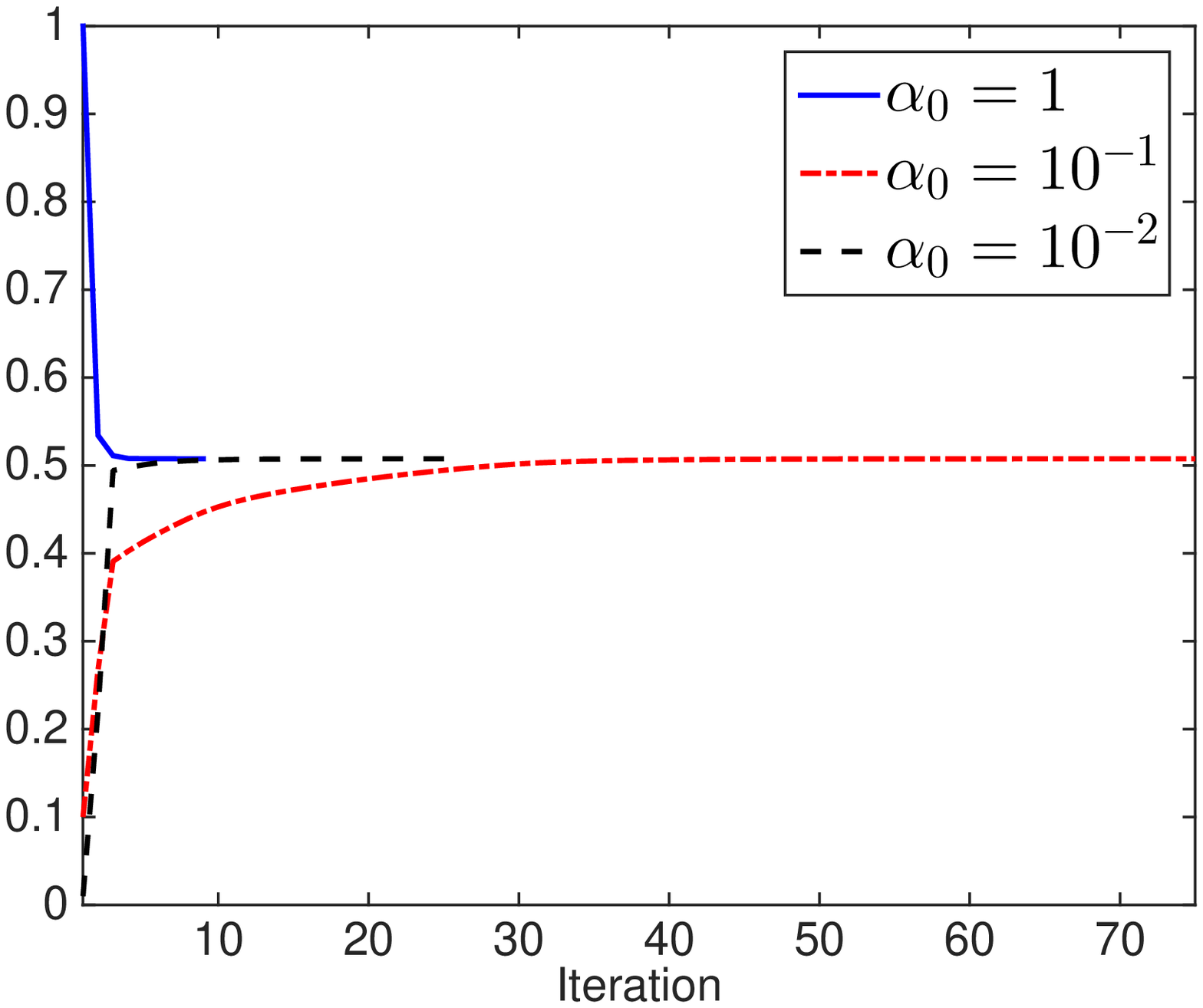}}
\subfigure[$r_1=r_2=0.05$; pAPS-algorithm]{\includegraphics[height=4.2cm]{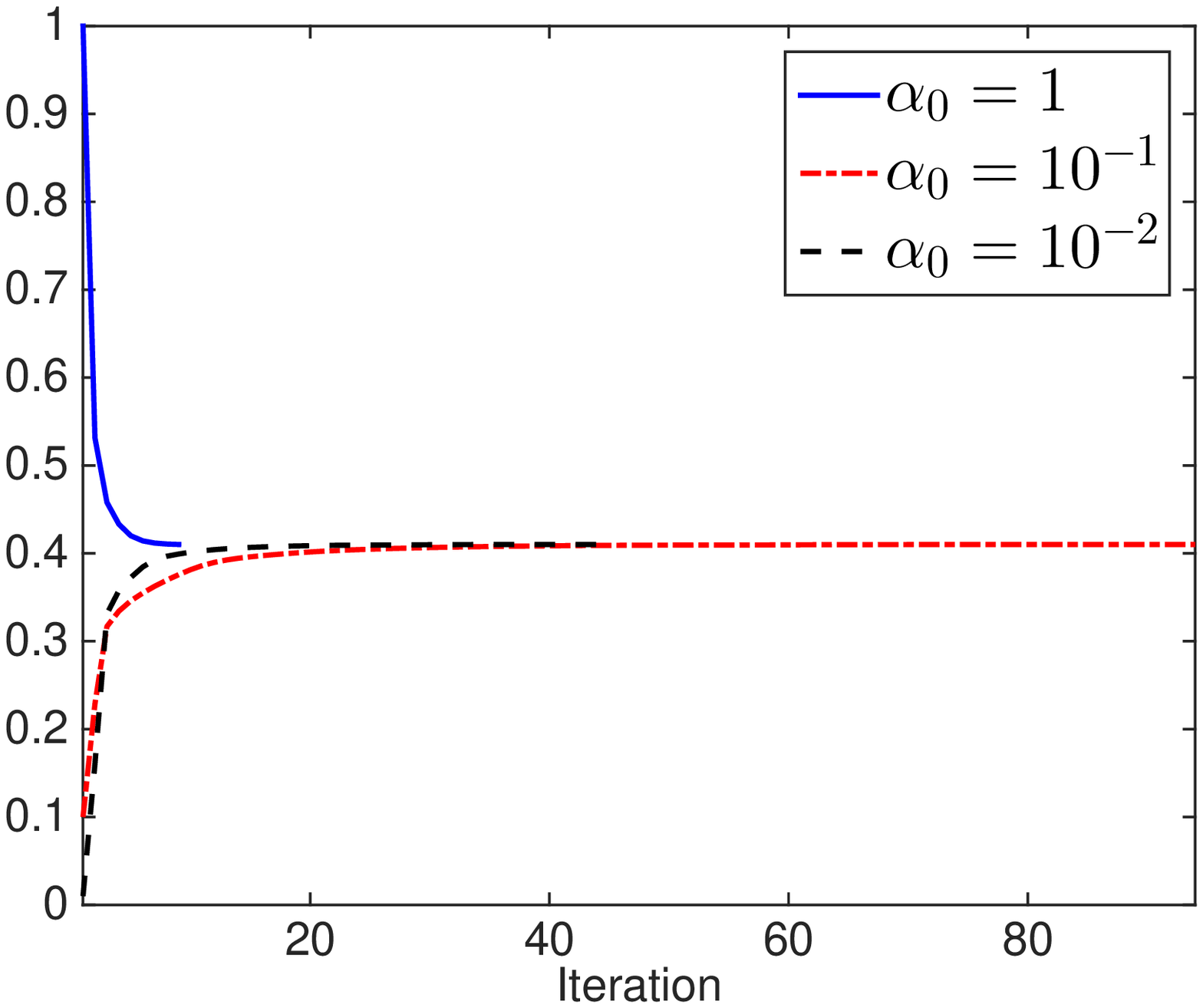}}
\subfigure[$r_1=r_2=0.3$; APS-algorithm]{\includegraphics[height=4.2cm]{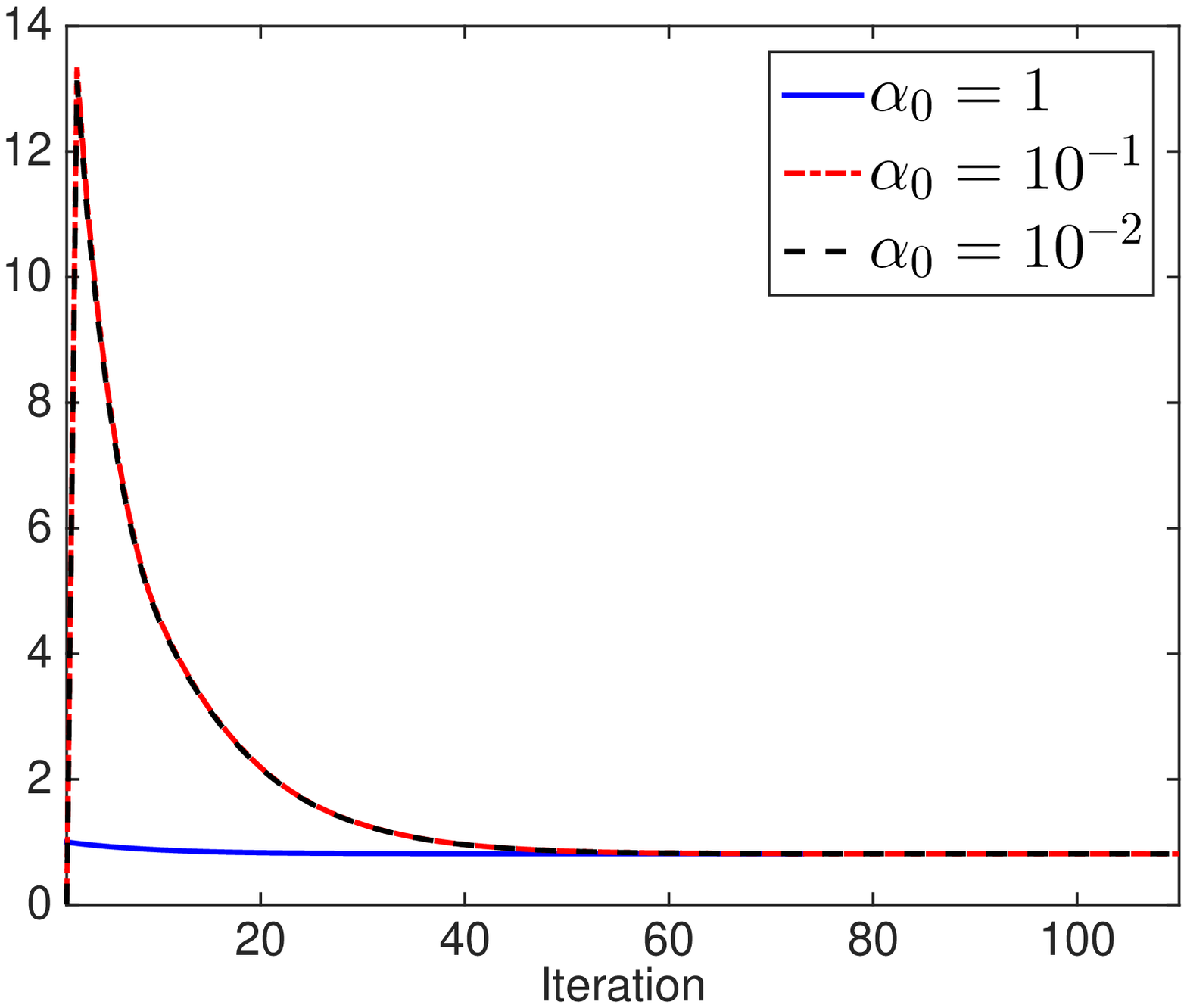}}
\subfigure[$r_1=r_2=0.1$; APS-algorithm]{\includegraphics[height=4.2cm]{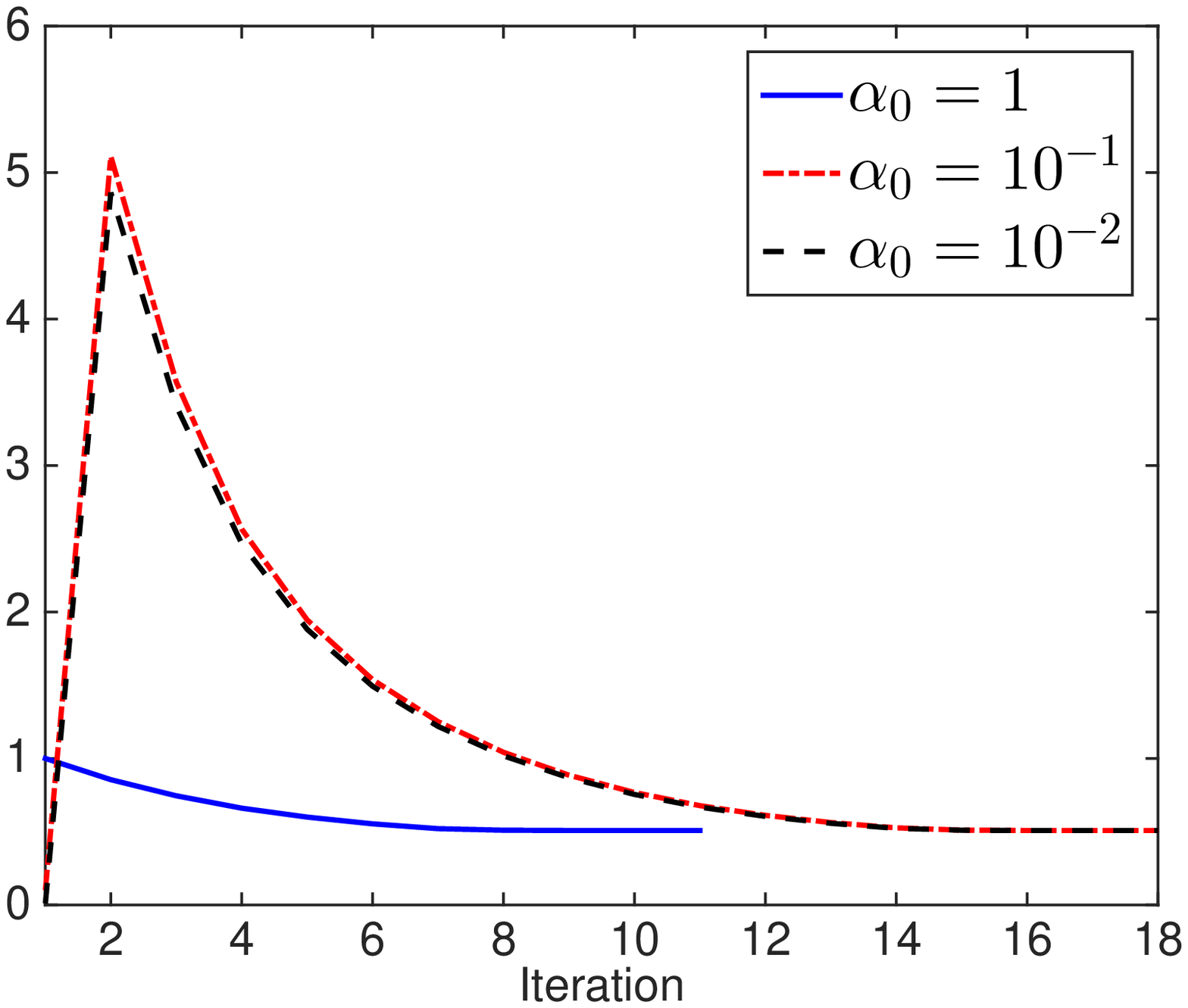}}
\subfigure[$r_1=r_2=0.05$; APS-algorithm]{\includegraphics[height=4.2cm]{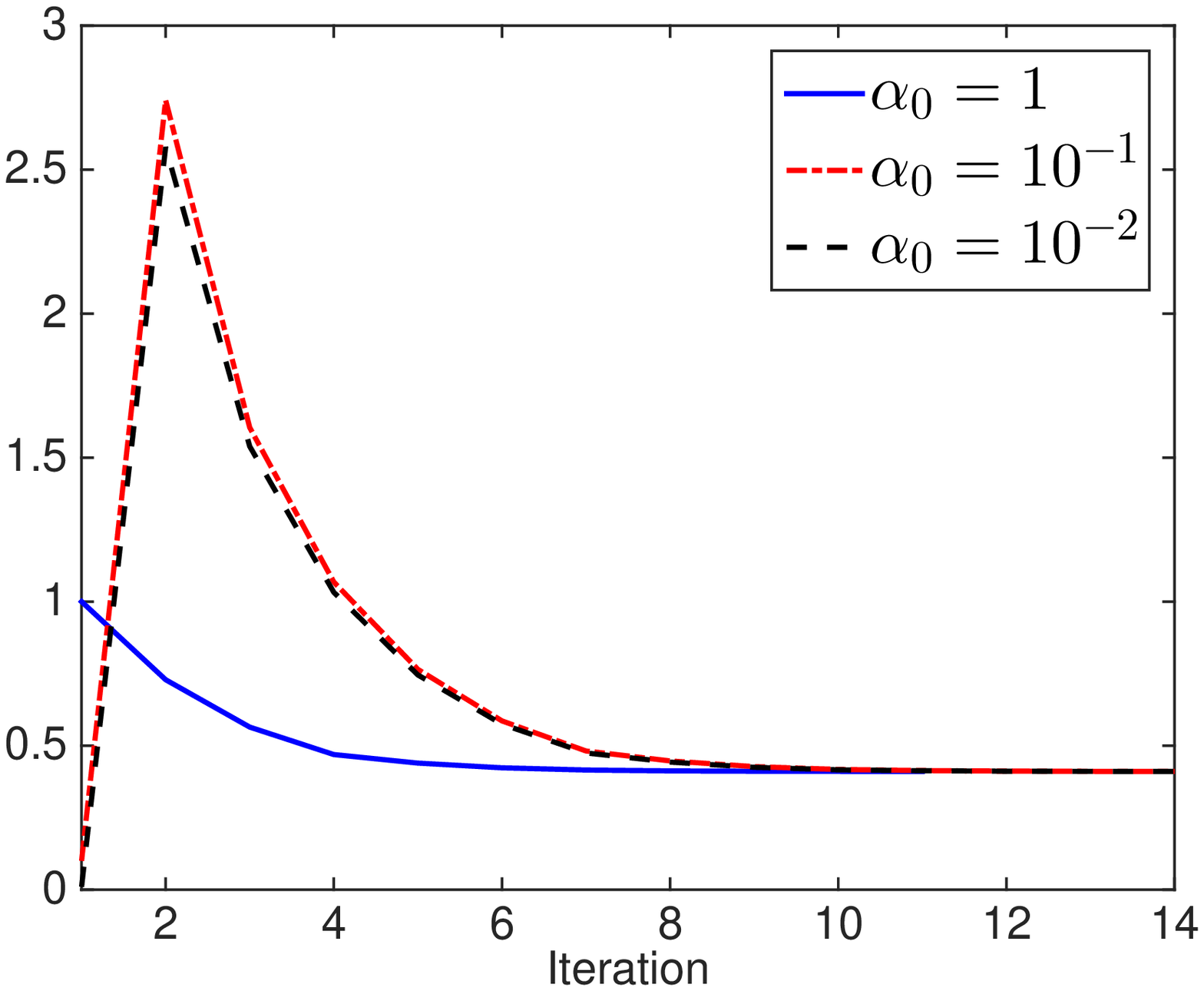}}
\end{center}    
\caption{\small \it Denoising of the cameraman-image corrupted by salt-and-pepper noise. In the pAPS-algorithm we set $p_0=32$.  }
\label{fig_alpha_SP}
\end{figure*}

\graphicspath{{./graphics/}}
\begin{figure*}[htbp!]
\begin{center}
\hspace{0cm}
\subfigure[$r=0.3$; pAPS-algorithm]{\includegraphics[height=4.2cm]{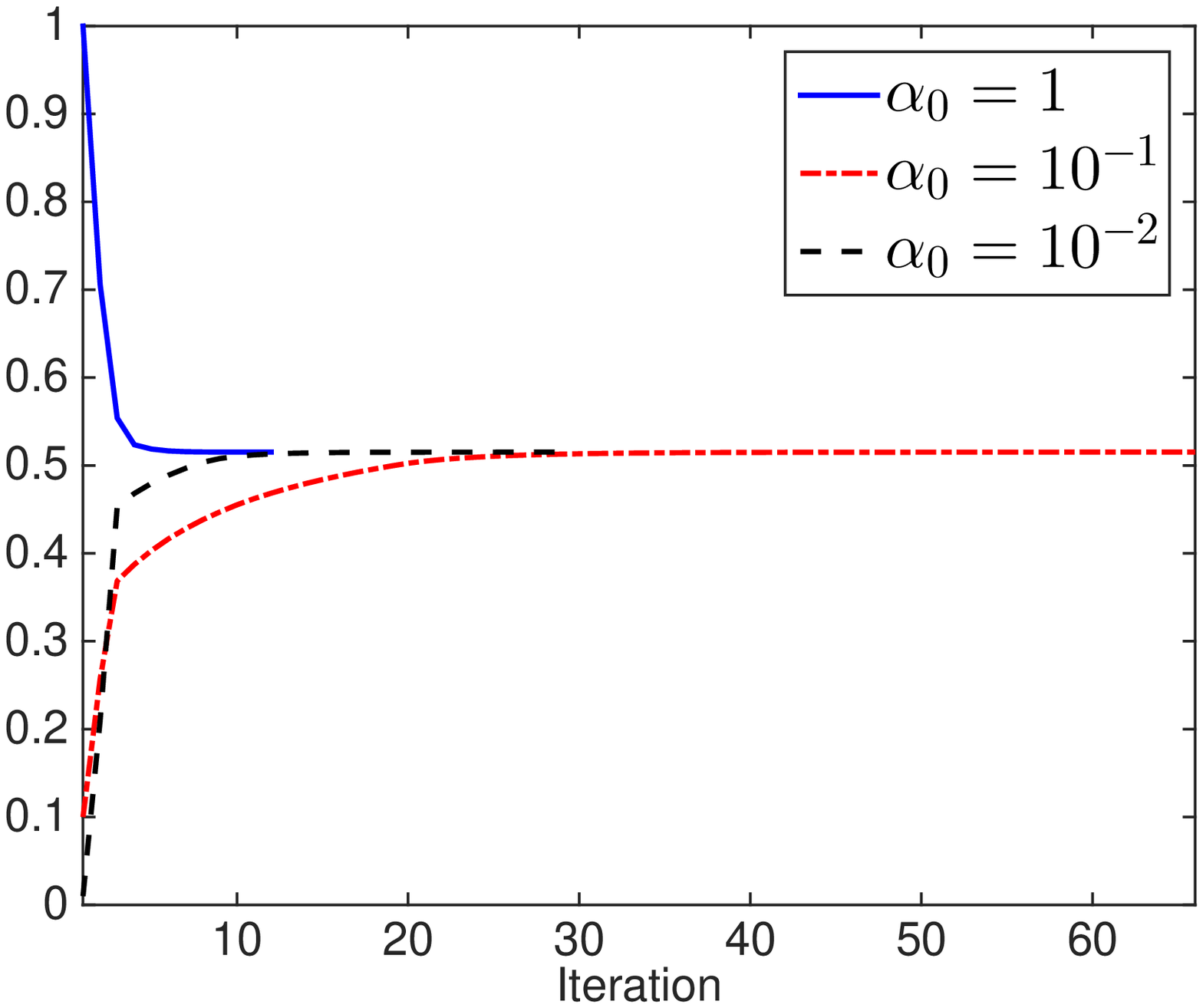}}
\subfigure[$r=0.1$; pAPS-algorithm]{\includegraphics[height=4.2cm]{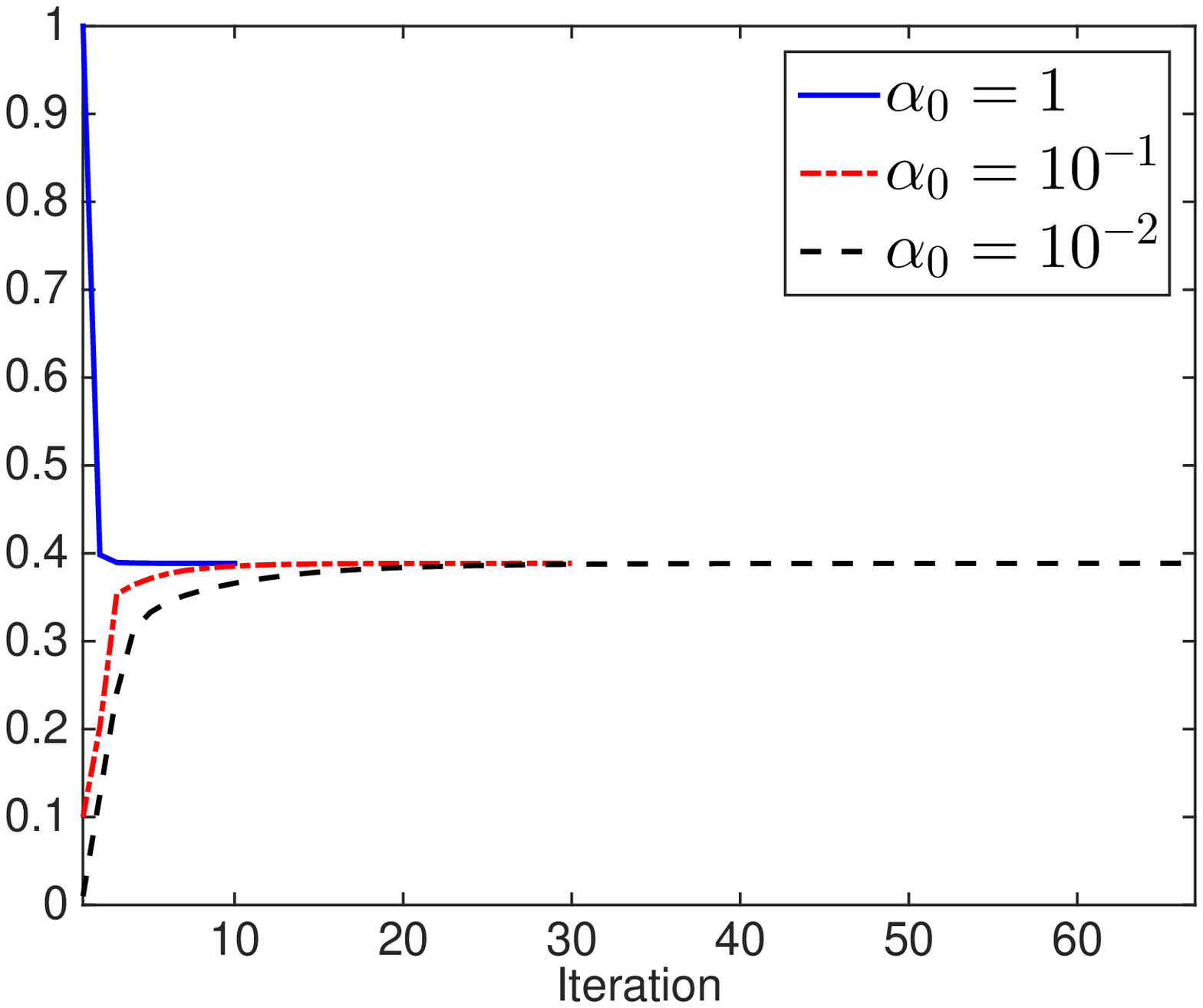}}
\subfigure[$r=0.05$; pAPS-algorithm \label{fig_alpha_RVc}]{\includegraphics[height=4.2cm]{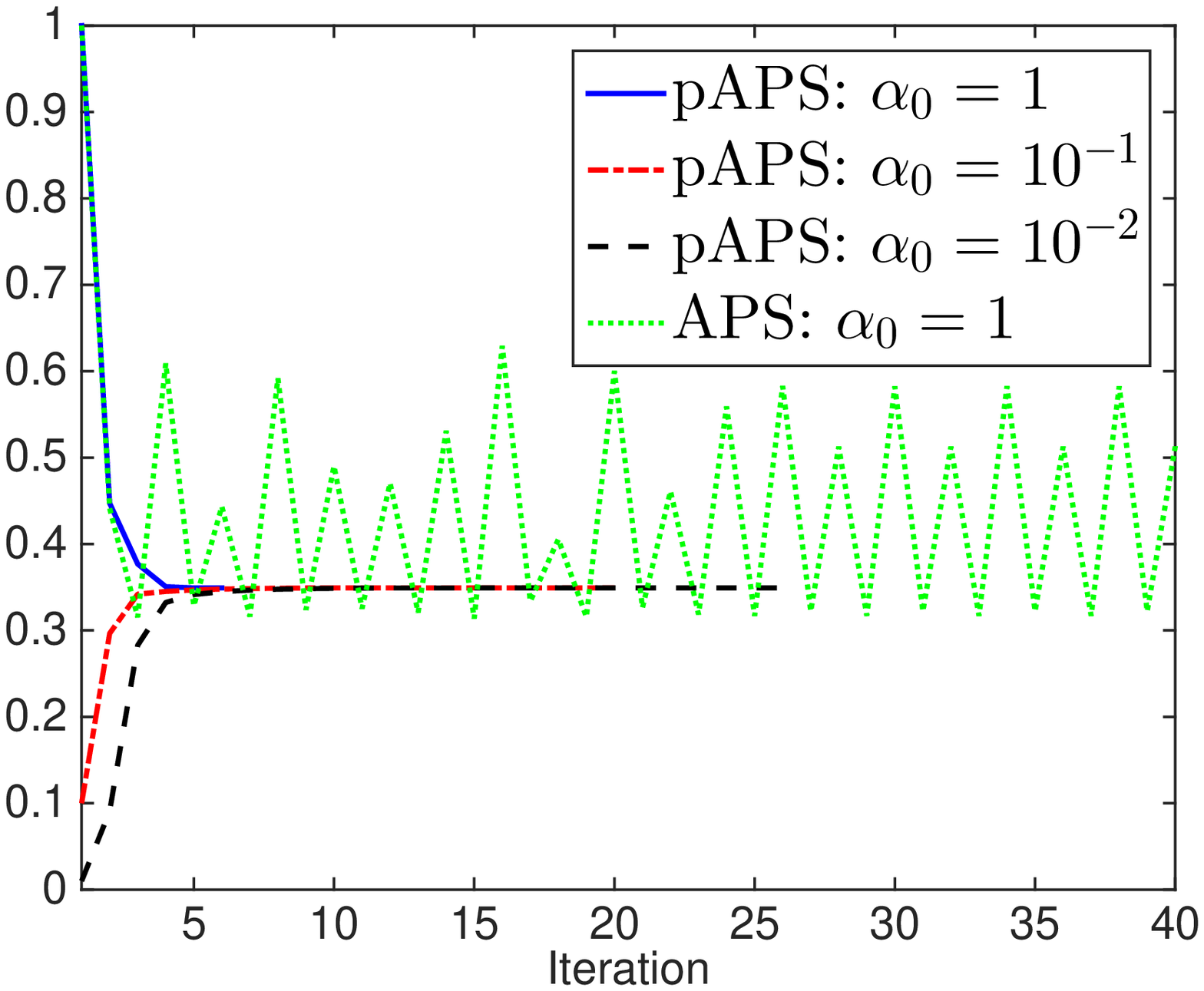}}
\subfigure[$r=0.3$; APS-algorithm]{\includegraphics[height=4.2cm]{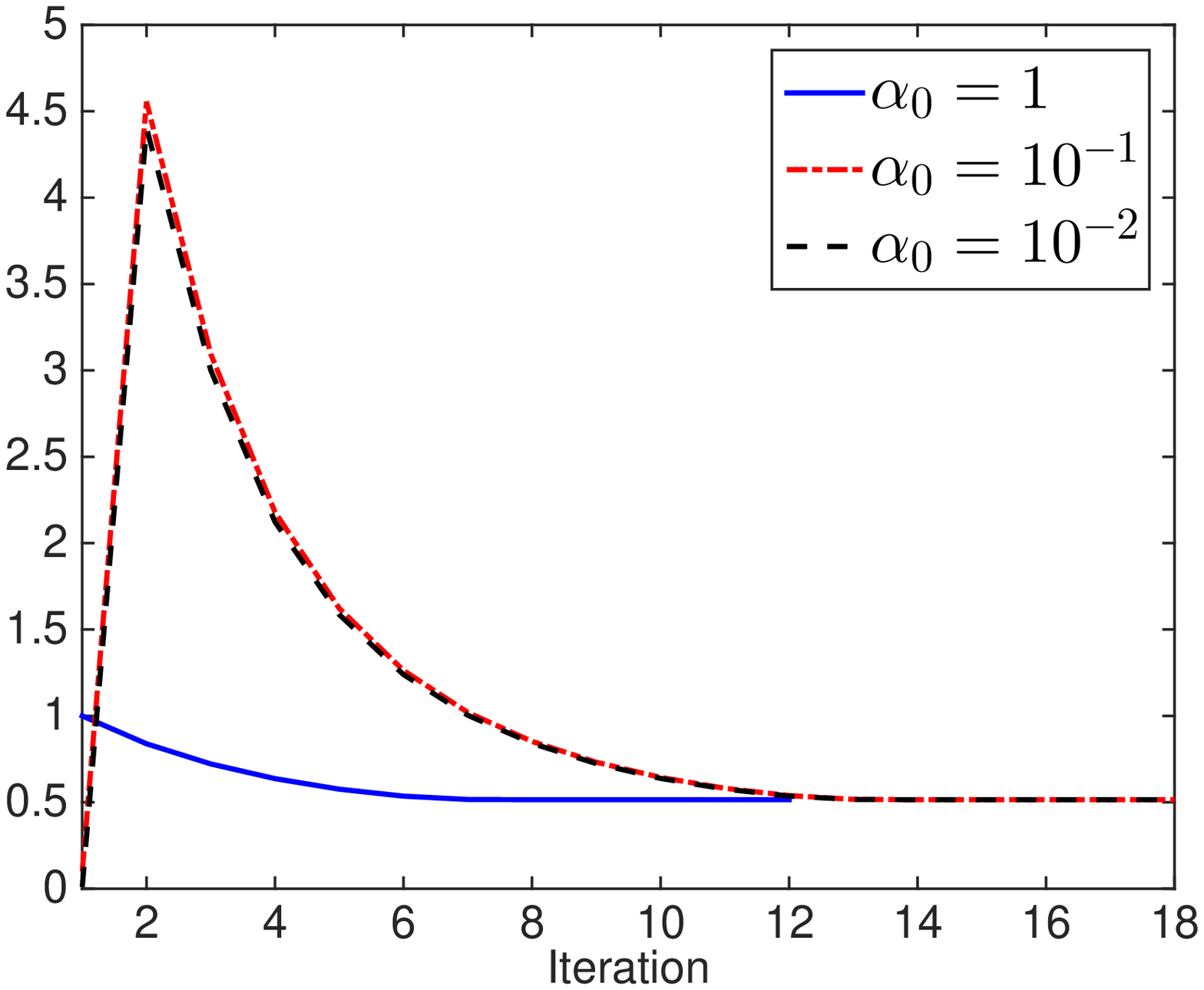}}
\subfigure[$r=0.1$; APS-algorithm]{\includegraphics[height=4.2cm]{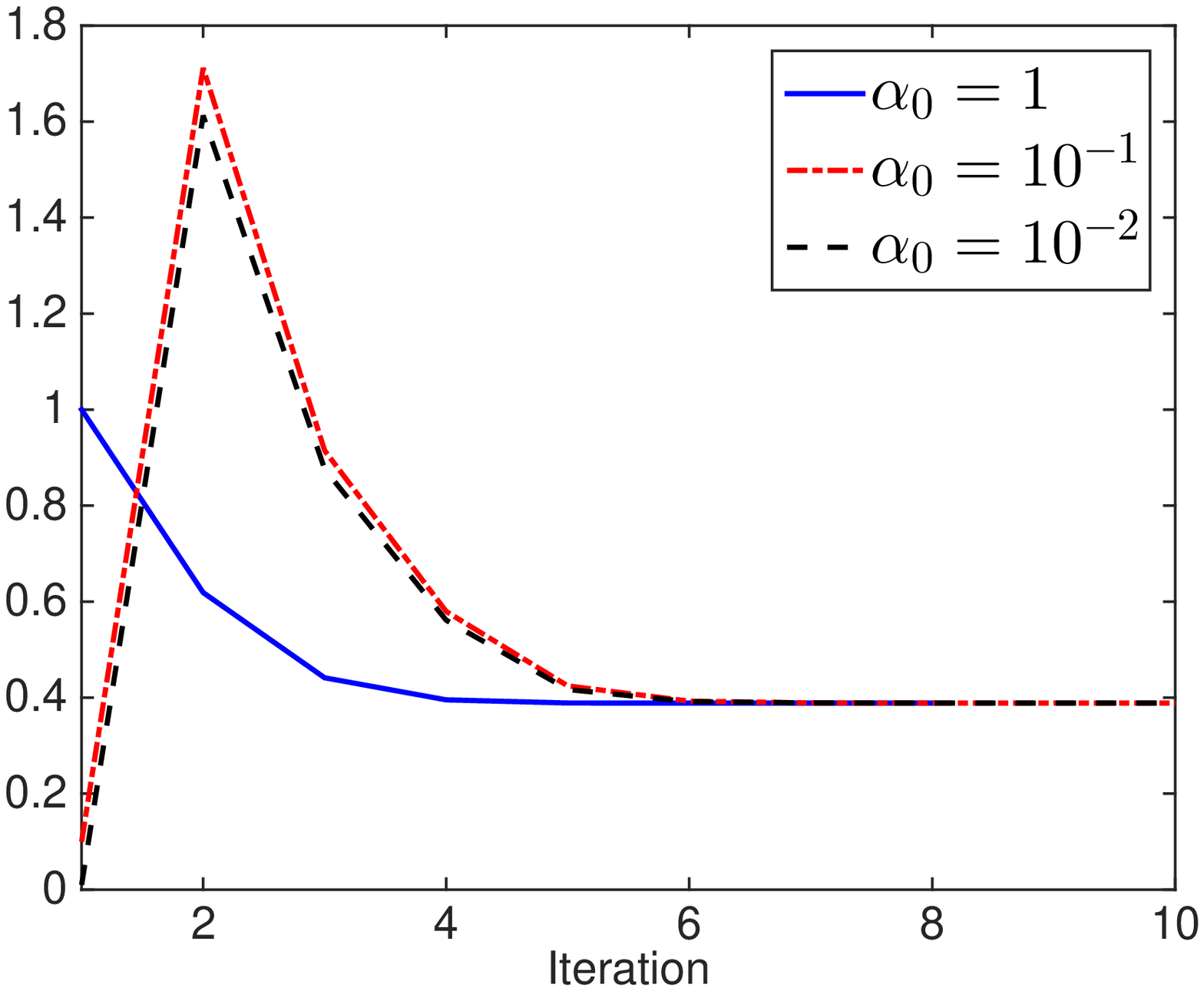}}
\end{center}    
\caption{\small \it Denoising of the cameraman-image corrupted by random-valued impulse noise. In the pAPS-algorithm we set $p_0=32$. }
\label{fig_alpha_RV}
\end{figure*}

\subsection{Locally adaptive total variation minimization}
 In this section various experiments are presented to evaluate the performance of the  LATV- and pLATV-algorithm presented in Section \ref{Sec:locallyTV}. Their performance is compared with the proposed pAPS-algorithm as well as with the SA-TV-algorithm introduced in \cite{DonHinRin} for $L^2$-TV minimization and in \cite{HinRin} for $L^1$-TV minimization. We recall that the SA-TV methods perform an approximate solution for the optimization problem in \eqref{LtauTVmultimodel2}, respectively, and compute automatically a spatially varying $\lambda$ based on a local variance estimation. However, as pointed out in \cite{DonHinRin,HinRin}, they only perform efficiently when the initial $\lambda$ is chosen sufficiently small, as we will do in our numerics. On the contrary, for the LATV- and pLATV-algorithm any positive initial $\alpha_0$ is sufficient. 
 
 For the comparison we consider four different images, shown in Fig. \ref{fig_org1}, which are all of size $256 \times 256$ pixels. In all our experiments for the SA-TV-algorithm we use $\mathcal{I}_{i,j}= \Omega_{i,j}^\omega$, see \cite{DonHinRin}, and we set the window-size to $11 \times 11$ pixels in the case of Gaussian noise and to $21 \times 21$ pixels in case of impulse noise. For the LATV- and pLATV-algorithm we use the window-size $\omega=11$, if not otherwise specified, and choose $p_0=\frac{1}{2}$.
 
\subsection{Gaussian noise removal}

\subsubsection{Dependency on the initial regularization parameter}
We start this section by investigating the stability of the SA-TV-, LATV-, and pLATV-algorithm with respect to the initial regularization parameter, i.e., $\lambda_0$ for the SA-TV-algorithm and $\alpha_0$ for the other algorithms, by denoising the cameraman-image corrupted by Gaussian white noise with standard deviation $\sigma=0.1$. In this context we also compare the difference of the pLATV-algorithm with and without using Algorithm 1 for computing automatically an initial parameter, where we set $c_{\alpha_0}=\frac{1}{5}$. The minimization problems contained in the LATV- and pLATV-algorithm are solved as described in Section \ref{Sec:AlgfImagDen}. For comparison reasons we define the values PSNR$_{\text{diff}}:= \max_{\alpha_0}$ PSNR$(\alpha_0) - \min_{\alpha_0}$ PSNR$(\alpha_0)$ and MSSIM$_{\text{diff}}:= \max_{\alpha_0}$ MSSIM$(\alpha_0) - \min_{\alpha_0}$ MSSIM$(\alpha_0)$ to measure the variation of the considered quality measures. Here PSNR$(\alpha_0)$ and MSSIM$(\alpha_0)$ are the  PSNR and MSSIM values of the reconstructions, which are obtained from the considered algorithms when the initial regularization parameter is set to $\alpha_0$. From Table \ref{table_comparison_adaptive_reg} we observe that the pLATV-algorithm with and without Algorithm 1 are more stable with respect to the initial regularization parameter than the LATV-algorithm and the SA-TV-algorithm. This stable performance of the pLATV-algorithm is reasoned by the adaptivity of the value $p$, which allows the algorithm to reach the desired residual (at least very closely) for any $\alpha_0$. As expected, the pLATV-algorithm with Algorithm 1 is even more stable with respect to $\alpha_0$ than the pLATV-algorithm alone, since, due to Algorithm 1, the difference of the actually used initial parameters in the pLATV-algorithm is rather small leading to very similar results. Note, that if $\alpha_0$ is sufficiently small, then the pLATV-algorithm with and without Algorithm 1 coincide, see Table \ref{table_comparison_adaptive_reg} for $\alpha_0\in \{10^{-2},10^{-3},10^{-4}\}$. Actually in the rest of our experiments we choose $\alpha_0$ always so small that Algorithm 1 returns the inputted $\alpha_0$.

\begin{table*}
\caption{\small \it PSNR and MSSIM of the reconstruction of the cameraman-image corrupted by Gaussian white noise with standard deviation $\sigma=0.1$ via the LATV- and pLATV-algorithm with different $\alpha_0$ and via the SA-TV-algorithm with different $\lambda_0$. In the LATV- and pLATV-algorithm we use $\mathcal{I}_{i,j}=\tilde \Omega_{i,j}$ with window-size $11 \times 11$ pixels in the interior and we set $p_0=\frac{1}{2}$.}
\label{table_comparison_adaptive_reg}
\begin{center}
\begin{tabular}{c|c|c|c|c|c|c|c|c}
\toprule
\multicolumn{1}{c|}{}  & \multicolumn{2}{c|}{SA-TV}& \multicolumn{2}{c|}{LATV}& \multicolumn{2}{c|}{pLATV} & \multicolumn{2}{c}{pLATV with Algorithm 1}\\
$\alpha_0$/$\lambda_0$   & PSNR    & MSSIM     & PSNR    & MSSIM    & PSNR    & MSSIM & PSNR    & MSSIM   \\ \hline
$1$           & 27.82 & 0.8155 & 27.44 & 0.8258 & 27.37 & 0.8260 & 27.37 & 0.8168\\ 
$10^{-1}$ & 27.77 & 0.8123 & 27.59 & 0.8211  & 27.41 & 0.8189 & 27.38 & 0.8166 \\
$10^{-2}$ & 27.71 & 0.8107 & 27.39 & 0.8167  & 27.37 & 0.8167 & 27.37 & 0.8167\\
$10^{-3}$ & 27.42 & 0.8007 & 27.40 & 0.8167  & 27.38 & 0.8168 & 27.38 & 0.8168\\ 
$10^{-4}$ & 27.56 & 0.7792 & 27.40 & 0.8168  & 27.38 & 0.8168 & 27.38 & 0.8168\\ \hline
\multicolumn{1}{c|}{PSNR$_{\text{diff}}$} &  \multicolumn{2}{|c|}{$0.39646$} & \multicolumn{2}{|c|}{$0.20257$}& \multicolumn{2}{|c}{$0.044473$} & \multicolumn{2}{|c}{$0.012704$}\\
\multicolumn{1}{c|}{MSSIM$_{\text{diff}}$} &  \multicolumn{2}{|c|}{$0.036322$} & \multicolumn{2}{|c|}{$0.0091312$} & \multicolumn{2}{|c}{$0.0092963$} & \multicolumn{2}{|c}{$0.00019843$}\\
\end{tabular}

\end{center}
\end{table*}

\subsubsection{Dependency on the local window}
In Table \ref{table_comparison_qLATV_Omega} we report on the performance-tests of the pLATV-algorithm with respect to the chosen type of window, i.e., $\mathcal{I}_{i,j}=\tilde{\Omega}_{i,j}^\omega$ and $\mathcal{I}_{i,j}={\Omega}_{i,j}^\omega$. We observe that independently which type of window is used the algorithm finds nearly the same reconstruction. This may be attributed to the fact that the windows in the interior are the same for both types of window. Nevertheless, the boundaries are treated differently, which leads to different theoretical results, but seems not to have significant influence on the practical behavior. A similar behavior is observed for the LATV-algorithm, as the LATV- and pLATV-algorithm return nearly the same reconstructions as observed below in Table \ref{table_comparison_adaptive}. Since for both types of windows nearly the same results are obtained, in the rest of our experiments we limit ourselves to always set $\mathcal{I}_{i,j}=\tilde \Omega_{i,j}^\omega$ in the LATV- and pLATV-algorithm.


\begin{table*}
\caption{\small \it PSNR and MSSIM of the reconstruction of different images corrupted by Gaussian white noise with standard deviation $\sigma$ via the pLATV-algorithm with $\alpha_0=10^{-4}$.}
\label{table_comparison_qLATV_Omega}
\begin{center}
\begin{tabular}{c|c|c|c|c|c} 
\toprule
\multicolumn{2}{c|}{}  & \multicolumn{2}{c|}{pLATV with $\mathcal{I}=\tilde{\Omega}$}& \multicolumn{2}{c}{pLATV with $\mathcal{I}={\Omega}$} \\
Image & $\sigma$ & PSNR & MSSIM& PSNR& MSSIM  \\ \hline
cameraman    & $0.3$      & 22.47  & 0.6807 & 22.47 & 0.6809  \\ 
                       & $0.1$      & 27.38  & 0.8168 & 27.37 & 0.8165   \\
                       & $0.05$    & 30.91 & 0.8875 & 30.92 & 0.8875   \\
                       & $0.01$    & 40.69 & 0.9735 & 40.68 & 0.9735  \\  \hline
lena    & $0.3$      & 22.31  & 0.5947 & 22.30 & 0.5950  \\ 
           & $0.1$      & 26.85  & 0.7447 & 26.87 & 0.7448   \\
           & $0.05$    & 30.15 & 0.8301 & 30.15 & 0.8300   \\
           & $0.01$    & 39.69 & 0.9699 & 39.68 & 0.9699  \\  \hline
\end{tabular}

\end{center}
\end{table*}

Next, we test the pLATV-algorithm for different values of the window-size varying from 3 to 15. Fig. \ref{fig_window_test} shows the PSNR and MSSIM of the restoration of the cameraman-image degraded by different types of noise (i.e., Gaussian noise with $\sigma=0.3$ or $\sigma=0.1$, salt-and-pepper noise with $r_1=r_2=0.3$ or $r_1=r_2=0.1$, or random-valued impulse noise with $r=0.3$ or $r=0.1$), where the pLATV-algorithm with $\alpha_0=10^{-2}$ and $p_0=1/2$ is used. We observe that the PSNR and MSSIM are varying only slightly with respect to changing window-size. However, in the case of Gaussian noise elimination the PSNR and MSSIM increases very slightly with increasing window-size, while in the case of impulse noise contamination such a behavior cannot be observed. In Fig. \ref{fig_window_test} we also specify the number of iterations needed till termination of the algorithm. From this we observe that a larger window-size results in most experiments in more iterations.


\graphicspath{{./graphics/}}
\begin{figure}[htbp!]
\begin{center}
\hspace{0cm}
\includegraphics[height=4.2cm]{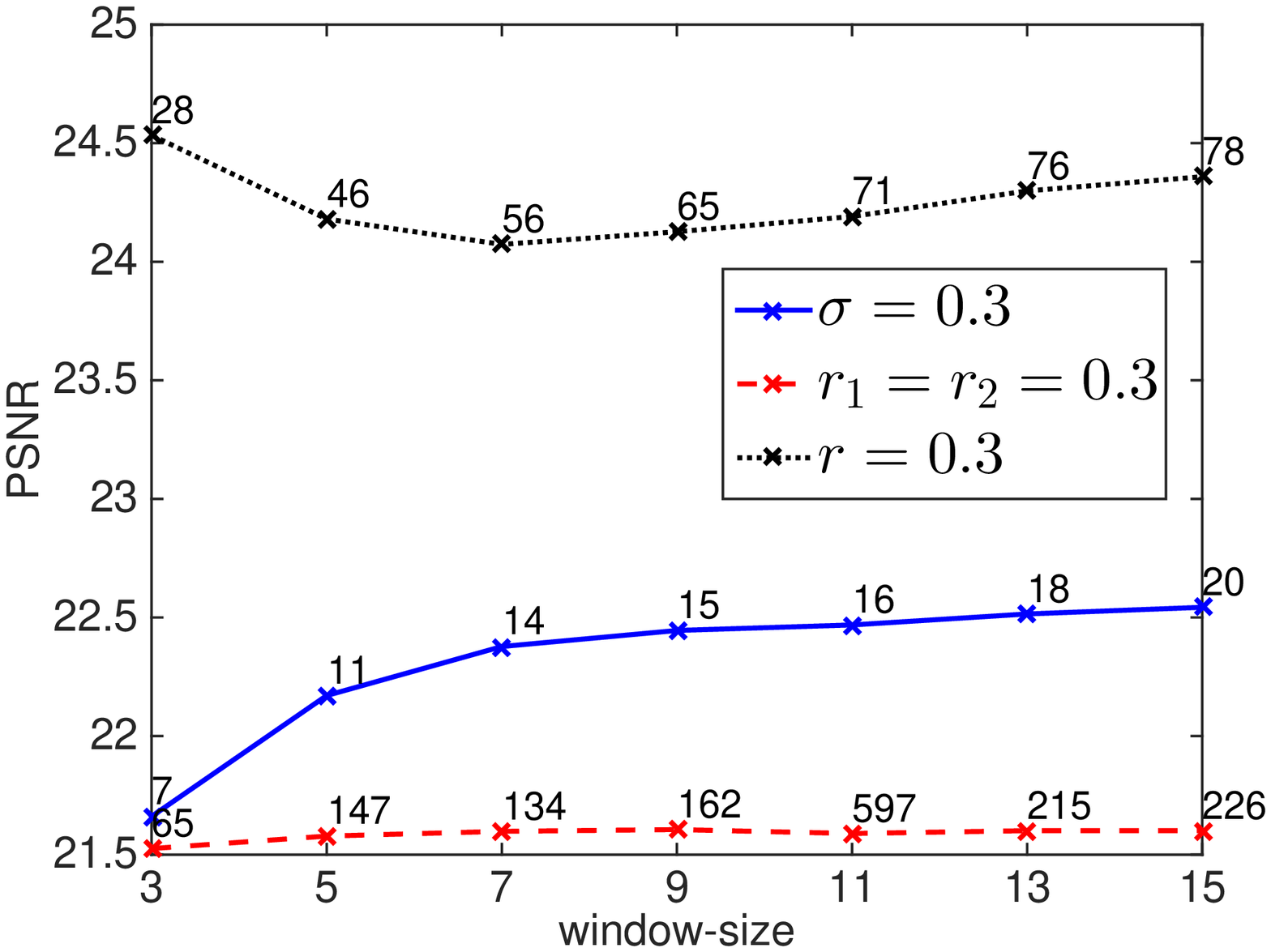}
\includegraphics[height=4.2cm]{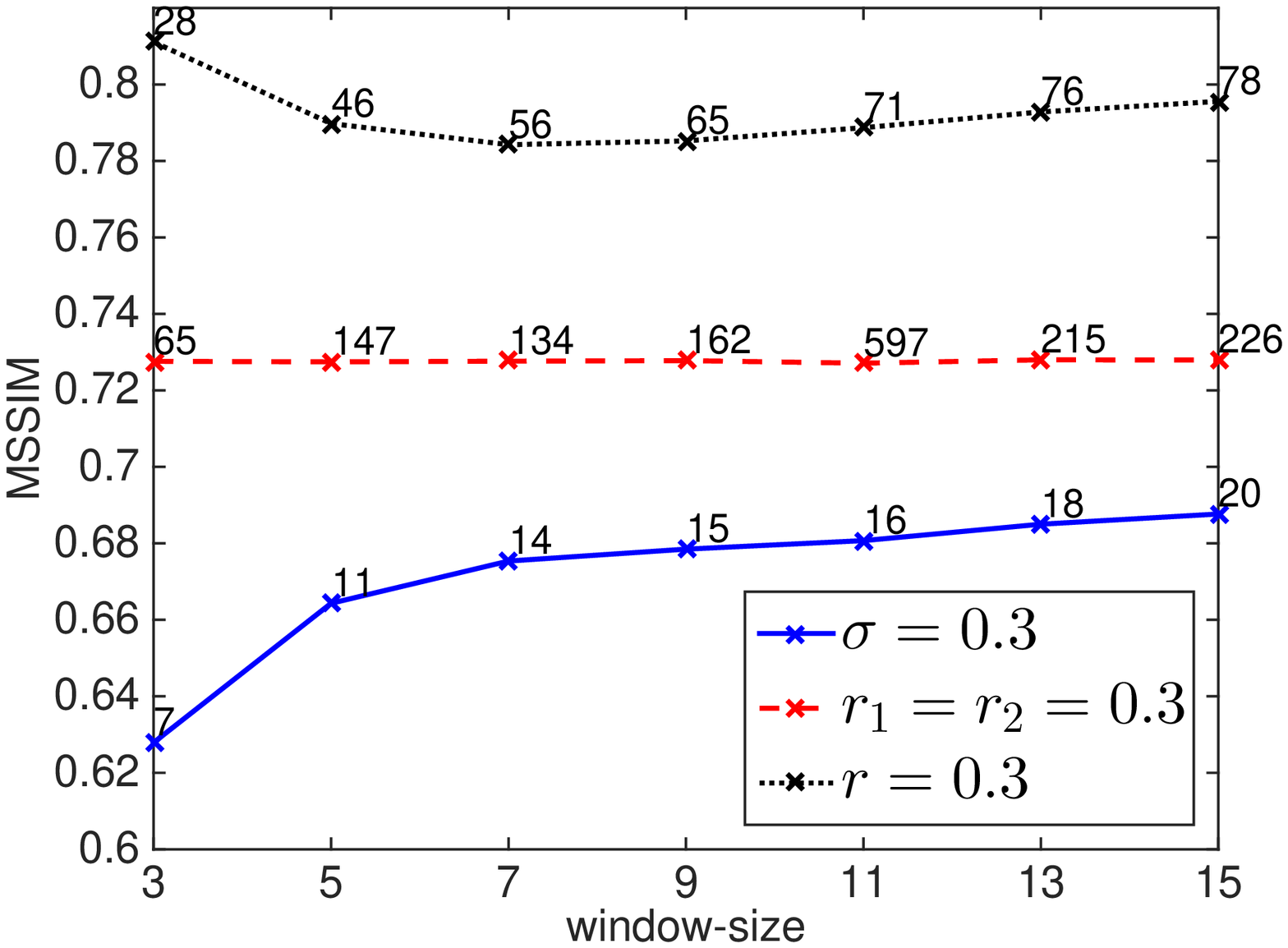}
\includegraphics[height=4.2cm]{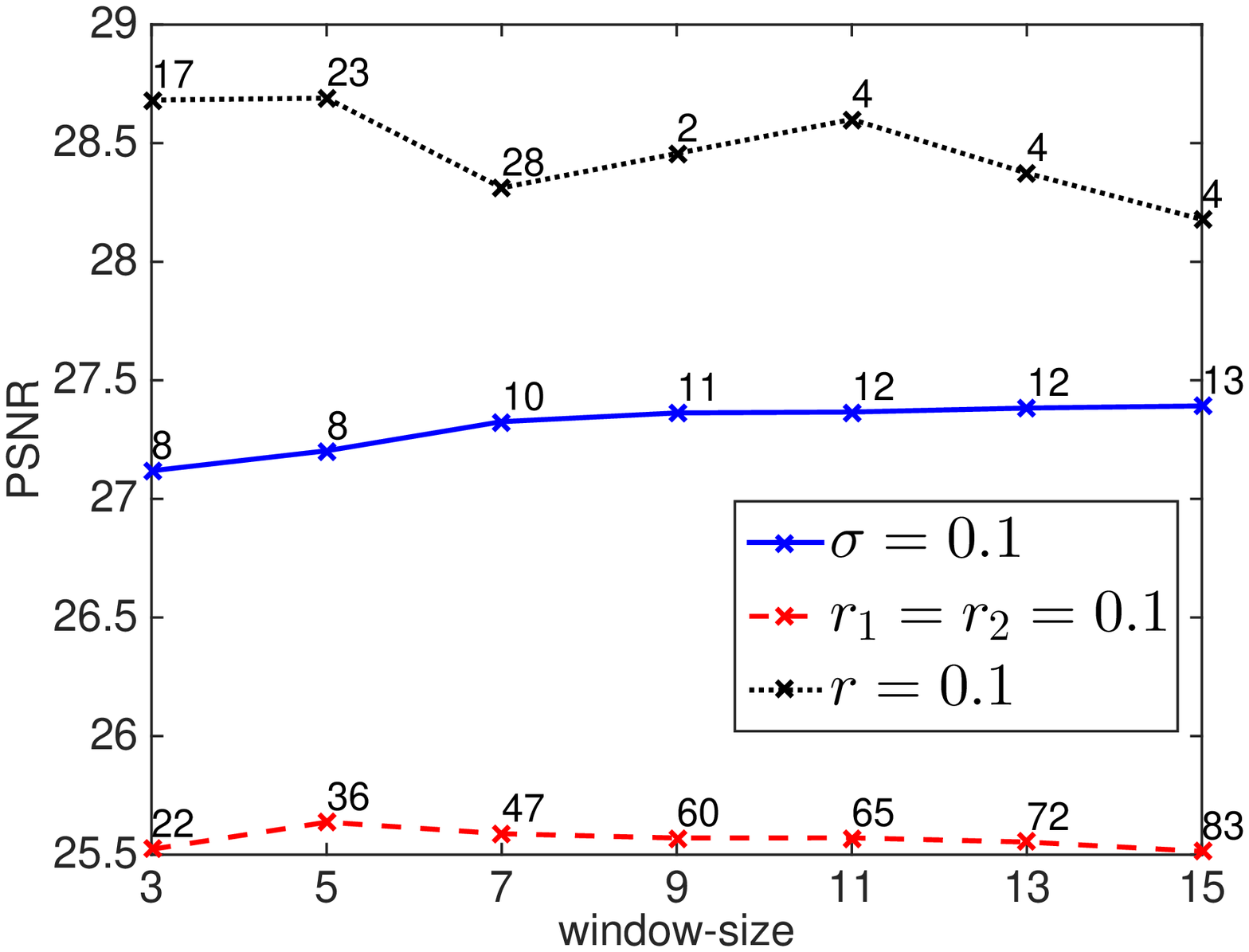}
\includegraphics[height=4.2cm]{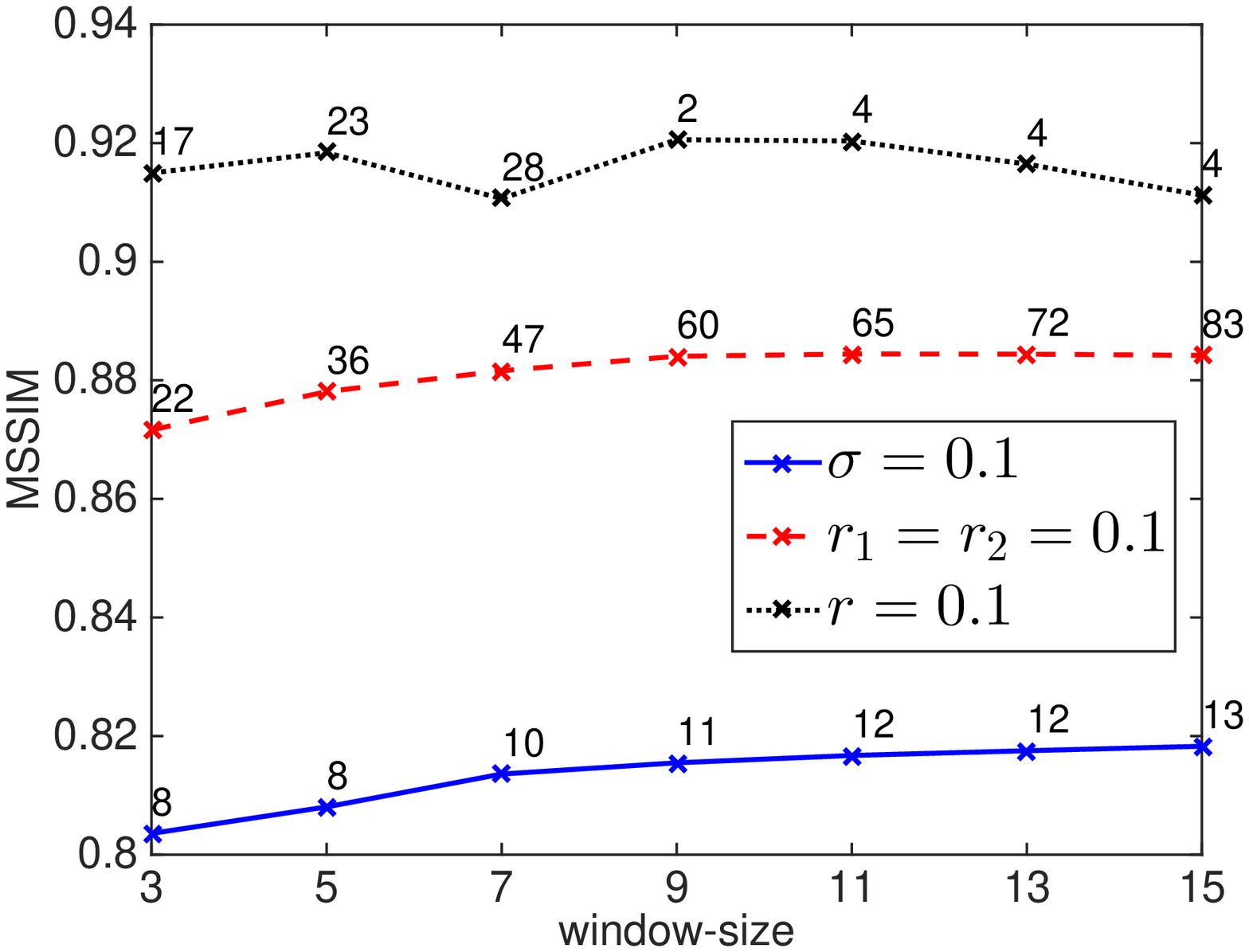}
\end{center}    
\caption{\small \it Restoration of the cameraman-image corrupted by different types of noise via the pLATV-method with different window-sizes. }
\label{fig_window_test}
\end{figure}

\subsubsection{Homogeneous noise}

Now we test the algorithms for different images corrupted by Gaussian noise with zero mean and different standard deviations $\sigma$, i.e., $\sigma \in \{0.3, 0.1, 0.05, 0.01\}$. The initial regularization parameter $\alpha_0$ is set to $10^{-4}$ in the pAPS-, LATV-, and pLATV-algorithm. In the SA-TV-algorithm we choose $\lambda_0=10^{-4}$, which seems sufficiently small. From Table \ref{table_comparison_adaptive} we observe that all considered algorithms behave very similar. However, for $\sigma\in \{ 0.1, 0.05, 0.01\}$ the SA-TV-algorithm most of the times performs best with respect to PSNR and MSSIM, while sometimes the LATV- and pLATV-algorithm have larger PSNR and MSSIM. That is, looking at these quality measures a locally varying regularization weight is preferred  to a scalar one, as long as $\sigma$  is sufficiently small. In Fig. \ref{fig_cam_G01} we present the reconstructions obtained via the considered algorithms and we observe that the LATV- and pLATV-algorithm generate visually the best results, while the result of the SA-TV-algorithm seems in some parts over-smoothed. For example, the very left tower in the SA-TV-reconstruction is completely vanished. This object is in the other restorations still visible. For large standard deviations, i.e. $\sigma=0.3$, we observe from Table \ref{table_comparison_adaptive} that the SA-TV method performs clearly worse than the other methods, while the pAPS-algorithm usually has larger PSNR  and the LATV- and pLATV-algorithm have larger MSSIM. Hence, whenever the noise-level is too large and details are considerably lost due to noise, the locally adaptive methods are not able to improve the restoration quality.

\begin{table*}
\scriptsize
\vspace{2mm}
\caption{\small \it PSNR- and MSSIM-values of the reconstruction of different images corrupted by Gaussian white noise with standard deviation $\sigma$ via pAPS-, LATV-, pLATV-algorithm with $\alpha_0=10^{-4}$ and SA-TV-algorithm with $\lambda_0=10^{-4}$. In the LATV- and pLATV-algorithm we use $\mathcal{I}_{i,j}=\tilde \Omega_{i,j}$ with window-size $11 \times 11$ pixels in the interior and we set $p_0=\frac{1}{2}$.}
\label{table_comparison_adaptive}
\begin{center}
\begin{tabular}{c|c|c|c|c|c|c|c|c|c}
\toprule
\multicolumn{2}{c|}{}  &  \multicolumn{2}{c|}{pAPS (scalar $\alpha$)} & \multicolumn{2}{c|}{SA-TV}& \multicolumn{2}{c|}{LATV}& \multicolumn{2}{c}{pLATV} \\
Image          & $\sigma$ & PSNR    & MSSIM     & PSNR    & MSSIM    & PSNR    & MSSIM    & PSNR    & MSSIM \\ \hline
phantom     
                    & $0.3$     & $19.84$ & $0.7989$ & $19.83$ & $0.8319$ & ${\bf 20.35}$ & $0.8411$ & $20.31$ & ${\bf 0.8432}$ \\
					&   $0.1$     & $28.97$ & $0.9644$ & $28.97$ & $0.9648$ & ${\bf 29.50}$ & ${\bf 0.9680}$ & ${\bf 29.50}$ & ${\bf 0.9680}$ \\
					&   $0.05$   & $34.97$ & $0.9887$ & $33.77$ & $0.9867$ & ${\bf 35.51}$ & ${\bf 0.9882}$ & ${\bf 35.51}$ & ${\bf 0.9882}$ \\
					&   $0.01$   & $48.88$ & ${\bf 0.9994}$ & $47.38$ & $0.9987$ & $49.46$ & $0.9993$ & ${\bf 49.53}$ & $0.9993$ \\ \hline
cameraman 
                    & $0.3$     & ${\bf 22.62}$ & ${\bf 0.6911}$ & $22.03$ & $0.6806$ & $22.47$ & $0.6807$ & $22.47$ & $0.6807$\\
					&   $0.1$     & $27.31$ & $0.8109$ & ${\bf 27.56}$ & $0.7792$ & $27.40$ & $0.8168$ & $27.38$ & ${\bf 0.8168}$ \\
					&   $0.05$   & $30.75$ & $0.8788$ & ${\bf 31.60}$ & ${\bf 0.8929}$ & $30.95$ & $0.8878$ & $30.91$ & $0.8875$ \\
					&   $0.01$   & $40.51$ & $0.9731$ & ${\bf 40.92}$ & $0.9649$ & $40.73$ & ${\bf 0.9737}$ & $40.69$ & $0.9735$ \\ \hline	
barbara		
                    &   $0.3$     & ${\bf 21.22}$ & $0.5022$ & $19.78$ & $0.4470$ & $21.05$ & ${\bf 0.5032}$ & $21.05$ & ${\bf 0.5032}$\\
					&   $0.1$     & $24.70$ & $0.7145$ & ${\bf 25.53}$ & ${\bf 0.7292}$ & $24.93$ & $0.7278$ & $24.93$ & $0.7278$ \\
					&   $0.05$   & $28.22$ & $0.8514$ & ${\bf 29.94}$ & ${\bf 0.8801}$ & $28.49$ & $0.8584$ & $28.49$ & $0.8584$\\
					&   $0.01$   & $38.91$ & $0.9791$ & ${\bf 40.56}$ & ${\bf 0.9809}$ & $39.08$ & $0.9788$ & $39.08$ & $0.9788$ \\ \hline
lena     		
                    &   $0.3$     & ${\bf 22.42}$ & $0.5930$ & $21.09$ & $0.5474$ & $22.33$ & ${\bf 0.5951}$ & $22.31$ & $0.5947$ \\
					&   $0.1$     & $26.84$ & $0.7393$ & ${\bf 27.31}$ & ${\bf 0.7528}$ & $26.85$ & $0.7447$ & $26.85$ & $0.7447$ \\
					&   $0.05$   & $30.06$ & $0.8261$ & ${\bf 30.92}$ & ${\bf 0.8385}$ & $30.16$ & $0.8307$ & $30.15$ & $0.8301$ \\
					&   $0.01$   & $39.62$ &$0.9685$  & ${\bf 39.81}$ & $0.9660$ & $39.76$ & ${\bf 0.9708}$ & $39.69$ & $0.9699$ \\ \hline
\end{tabular}

\end{center}
\end{table*}

\graphicspath{{./graphics/}}
\begin{figure}[htbp!]
\begin{center}
\hspace{0cm}
    \subfigure[pAPS (PSNR: $27.31$; MSSIM: $0.8109$)]{\includegraphics[height=4.1cm]{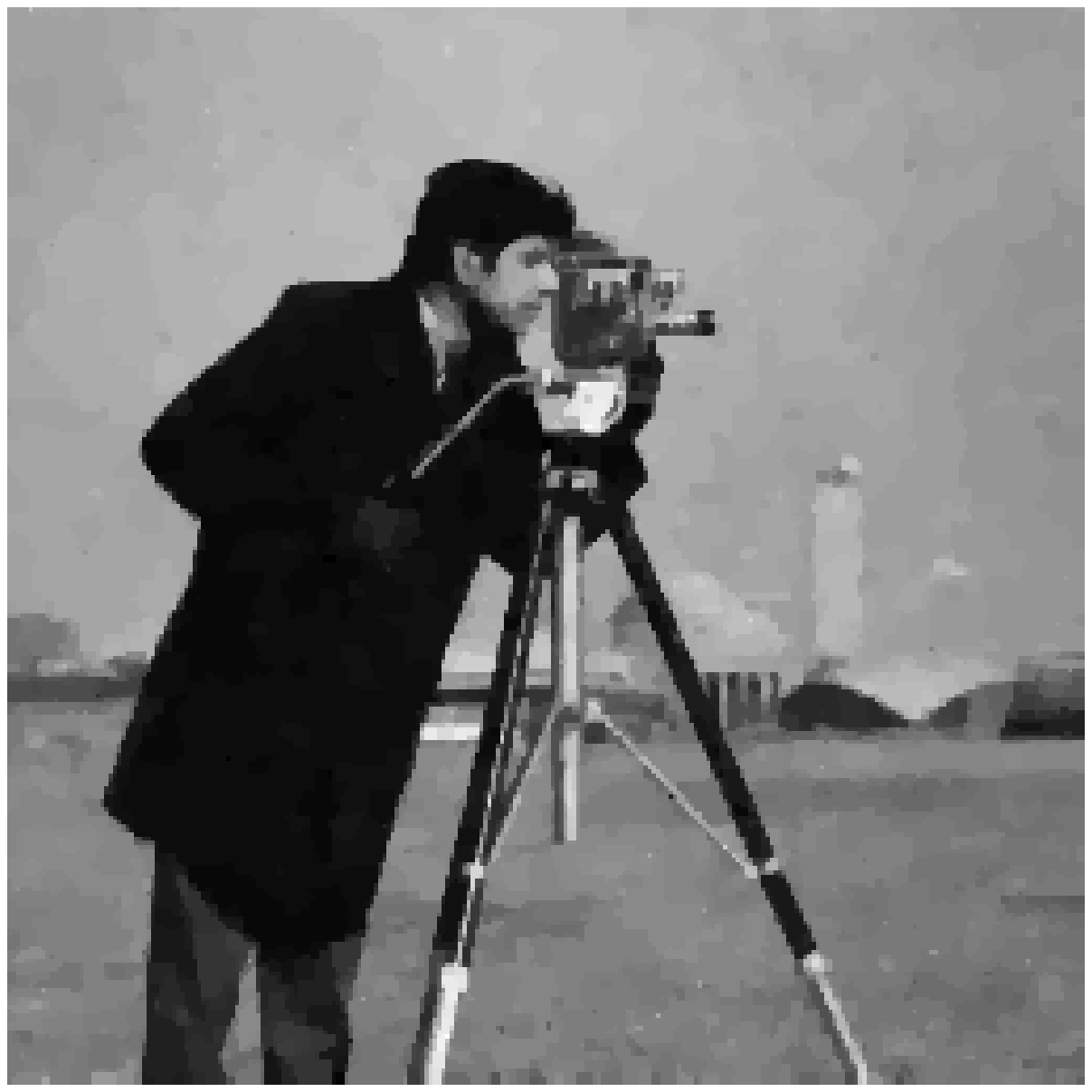}}
    \subfigure[SA-TV (PSNR: $27.56$; MSSIM: $0.7792$)]{\includegraphics[height=4.1cm]{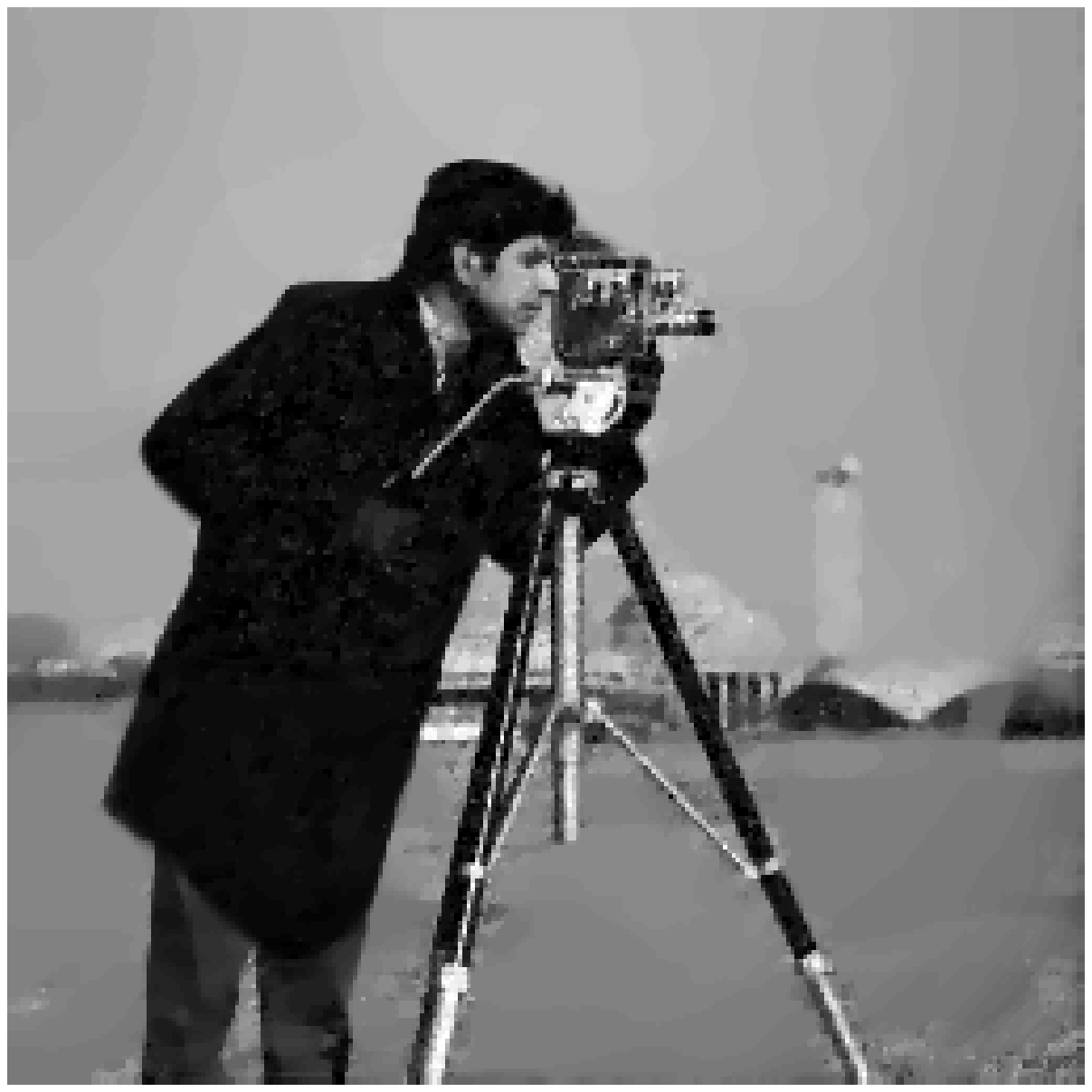}}
    \subfigure[ LATV (PSNR: $27.40$; MSSIM: $0.8170$)]{\includegraphics[height=4.1cm]{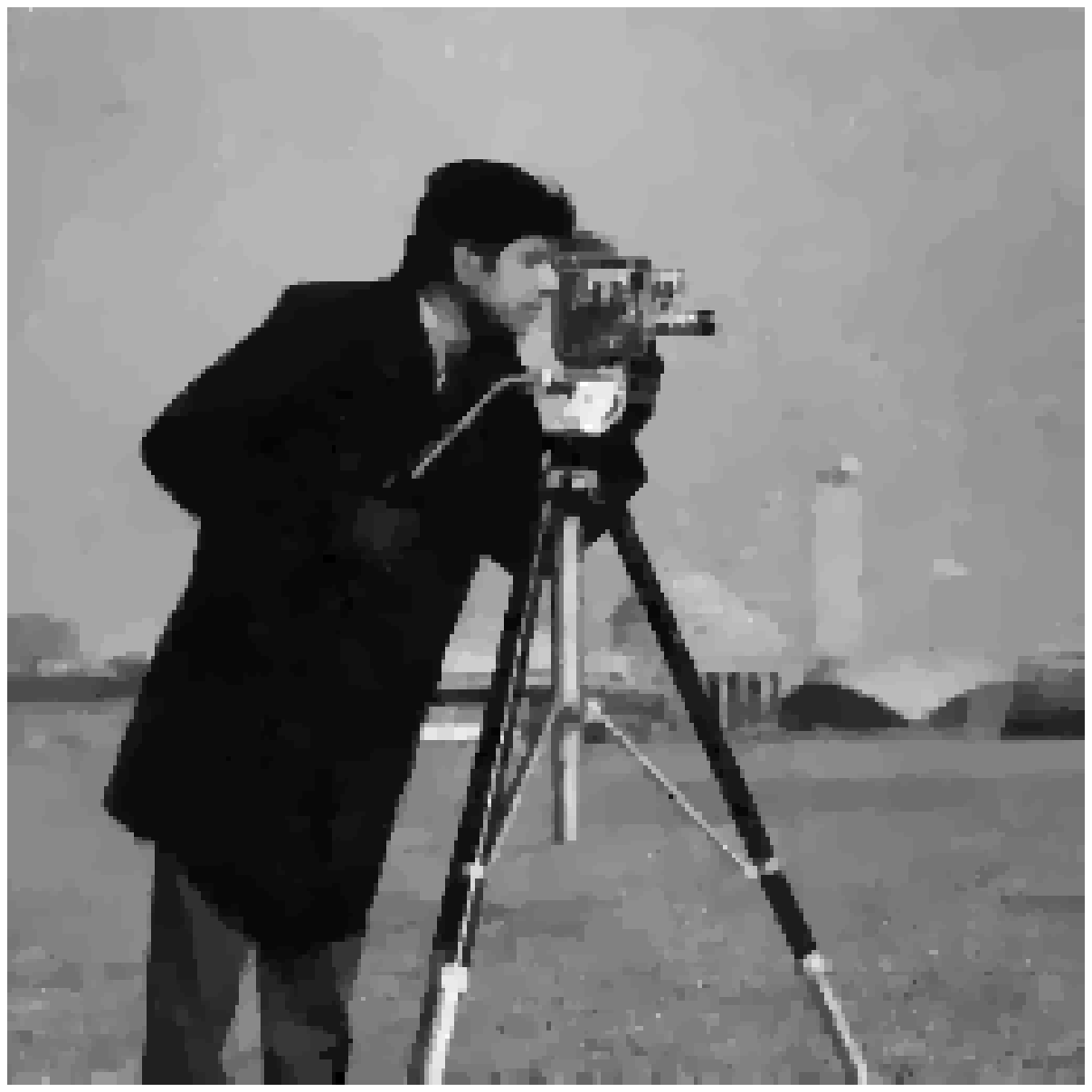}}
    \subfigure[pLATV (PSNR: $27.38$; MSSIM: $0.8171$)]{\includegraphics[height=4.1cm]{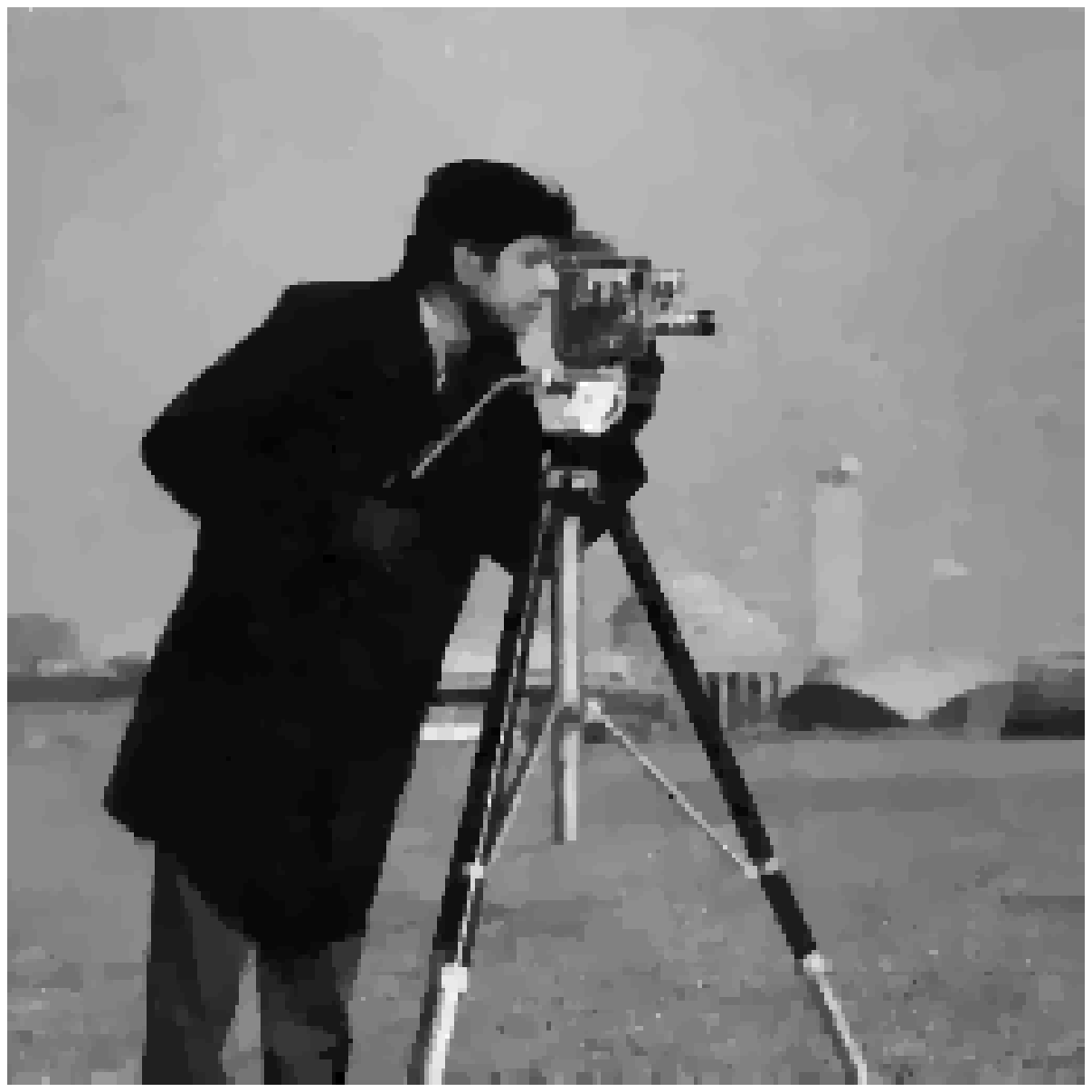}}
\end{center}    
\caption{\small \it Reconstruction of the cameraman-image corrupted by Gaussian white noise with $\sigma=0.1$.} 
\label{fig_cam_G01}
\end{figure} 
 

\subsubsection{Non-homogeneous noise}

For this experiment we consider the cameraman-image degraded by Gaussian white noise with variance $\sigma^2 = 0.0025$ in the whole domain $\Omega$ except a rather small area (highlighted in red in Fig. \ref{fig_inhomo_a}), denoted by $\tilde{\Omega}$, where the variance is 6 times larger, i.e., $\sigma^2=0.015$ in this part. Since the noise-level is in this application not homogeneous, the pLATV-algorithm presented in Section \ref{Sec:locallyTV} has to be adjusted to this situation accordingly. This can be done by making $\nu_\tau$ (here $\tau=2$) locally dependent and we write $\nu_\tau=\nu_\tau(\hat{u})(x)$ to stress the dependency on the true image $\hat{u}$ and on the location $x\in \Omega$ in the image. In particular, for our experiment we set $\nu_2=0.015$ in $\tilde \Omega$, while $\nu_2=0.0025$ in $\Omega\setminus\tilde{\Omega}$. Since $\nu_\tau$ is now varying, we also have to adjust the definition of $\Bt$ and $B_\tau$ to
$$
\Bt(u):= \int_\Omega \nu_\tau(u)(x) \ \diff x \ \text{and} \ B_\tau(u^h):= \sum_{x\in {\Omega^h}} \nu_\tau(u^h)(x),
$$
respectively for the continuous and discrete setting. Making these adaptations allows us to apply the pLATV-algorithm as well as the pAPS-algorithm to the application of removing non-uniform noise. 

The reconstructions obtained by the pAPS-algorithm (with $p_0=32$ and $\alpha_0=10^{-2}$) and by the pLATV-algorithm (with $p_0=1/2$ and $\alpha_0=10^{-2}$) are shown in Figs. \ref{fig_inhomo_b} and \ref{fig_inhomo_c} respectively. Due to the adaptive choice of $\alpha$, see Fig. \ref{fig_inhomo_d} where light colors indicate a large value, the pLATV-algorithm is able to remove all the noise considerably, while the pAPS-algorithm returns a restoration, which still retains noise in $\tilde{\Omega}$.

\graphicspath{{./graphics/}}
\begin{figure}[htbp!]
\begin{center}
\hspace{0cm}
    \subfigure[noisy observation \label{fig_inhomo_a}]{\includegraphics[height=4.1cm]{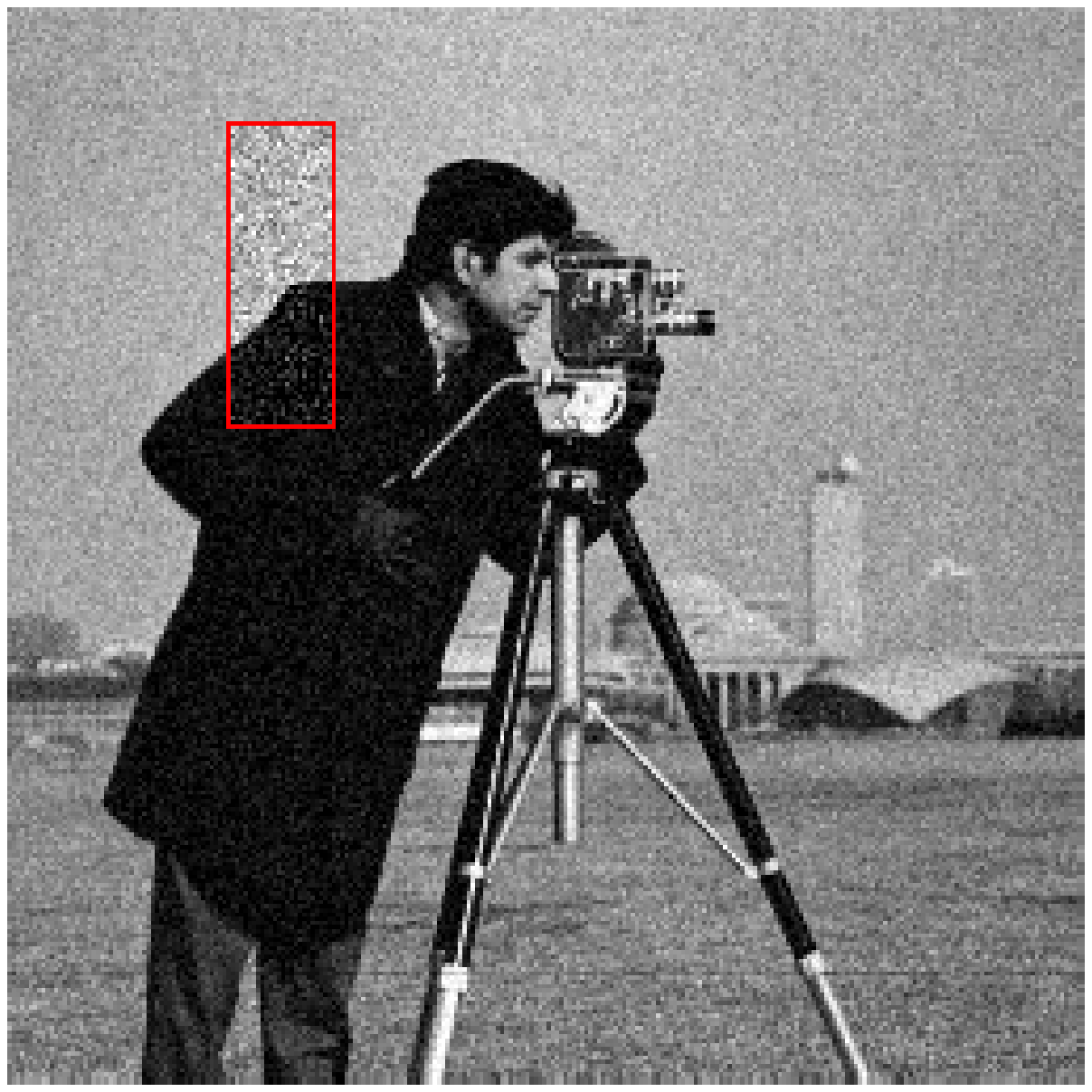}}
    \subfigure[pAPS (PSNR: $30.01$; MSSIM: $0.8551$)\label{fig_inhomo_b}]{\includegraphics[height=4.1cm]{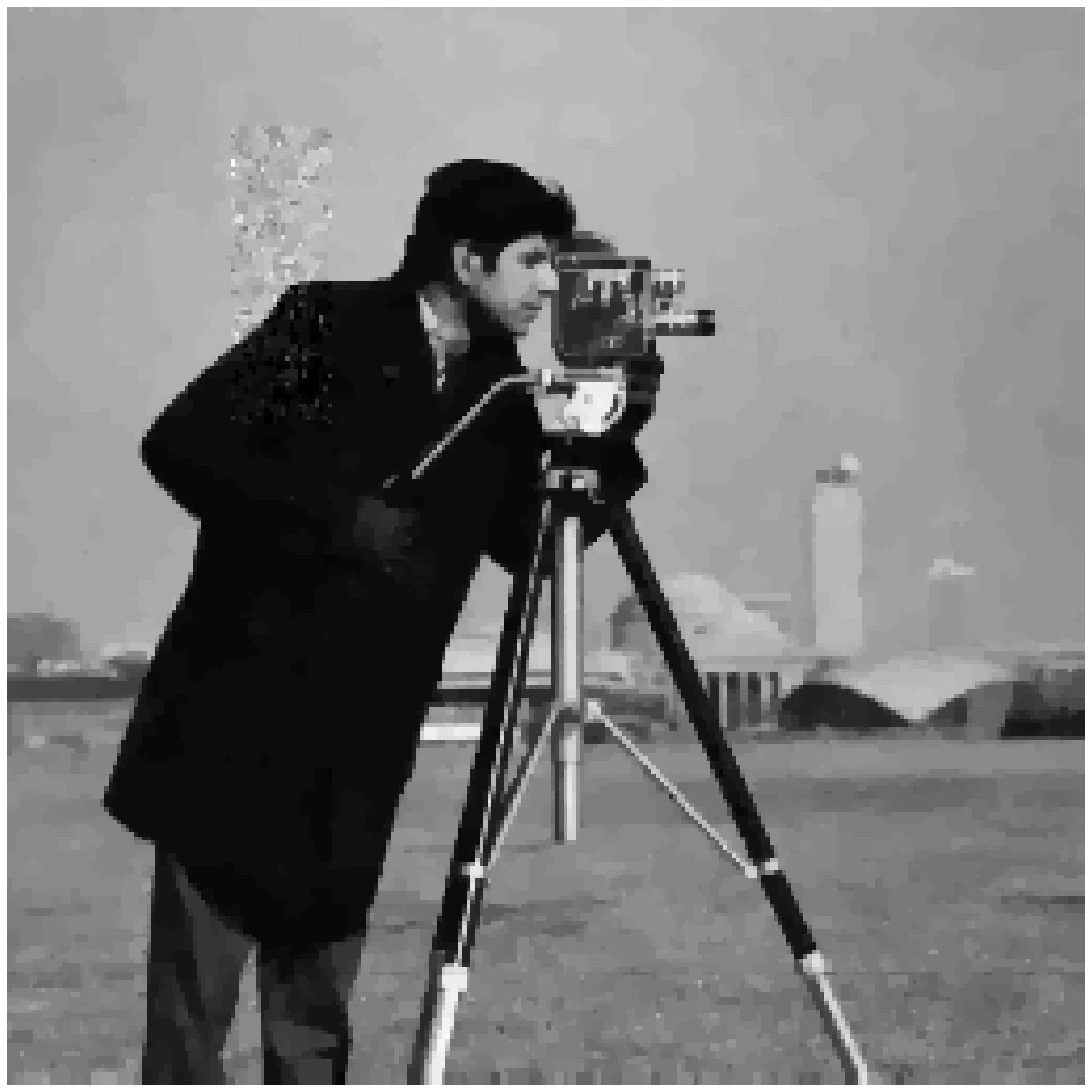}}
    \subfigure[ pLATV (PSNR: $30.85$; MSSIM: $0.8859$)\label{fig_inhomo_c}]{\includegraphics[height=4.1cm]{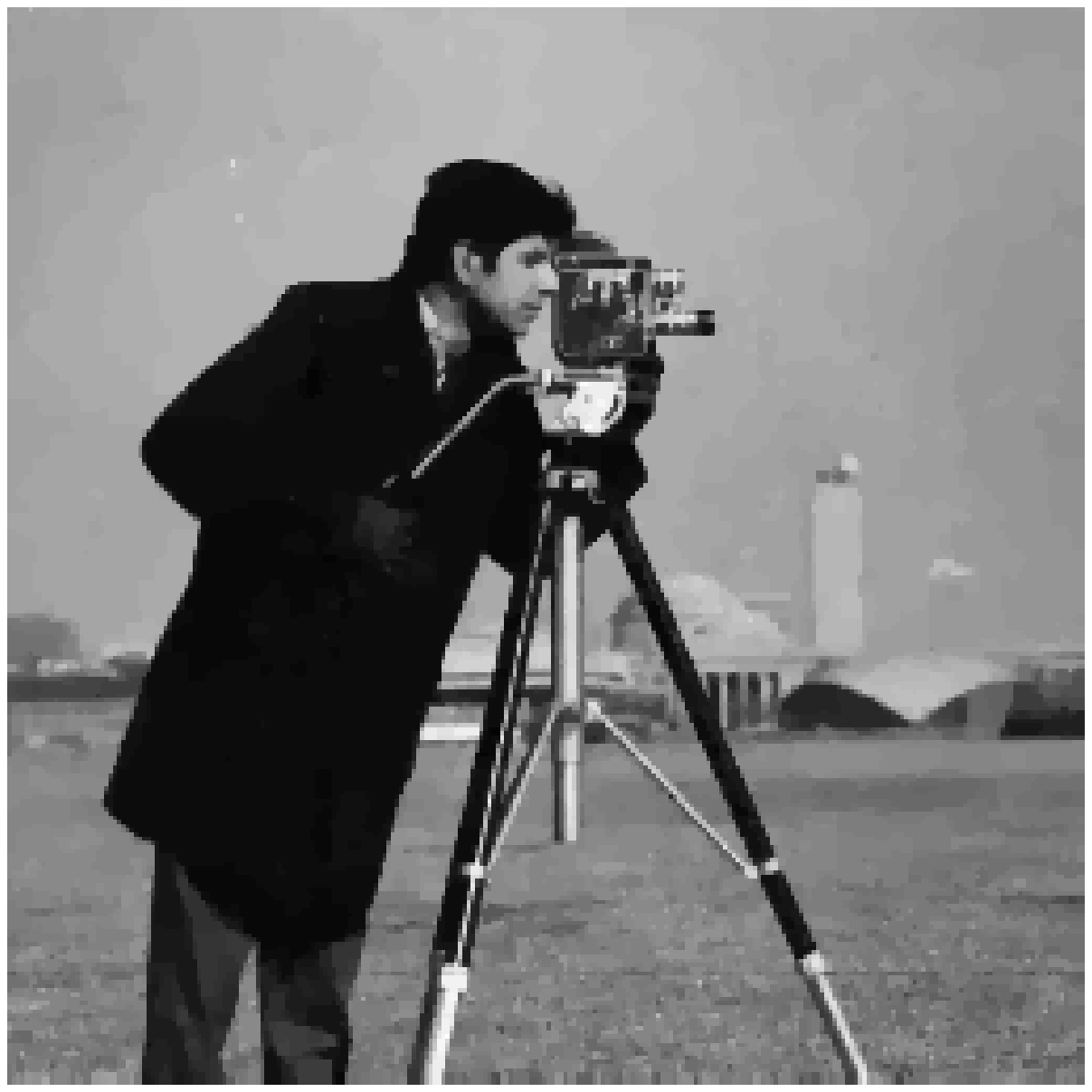}}
    \subfigure[Locally varying $\alpha$\label{fig_inhomo_d}]{\includegraphics[height=4.1cm]{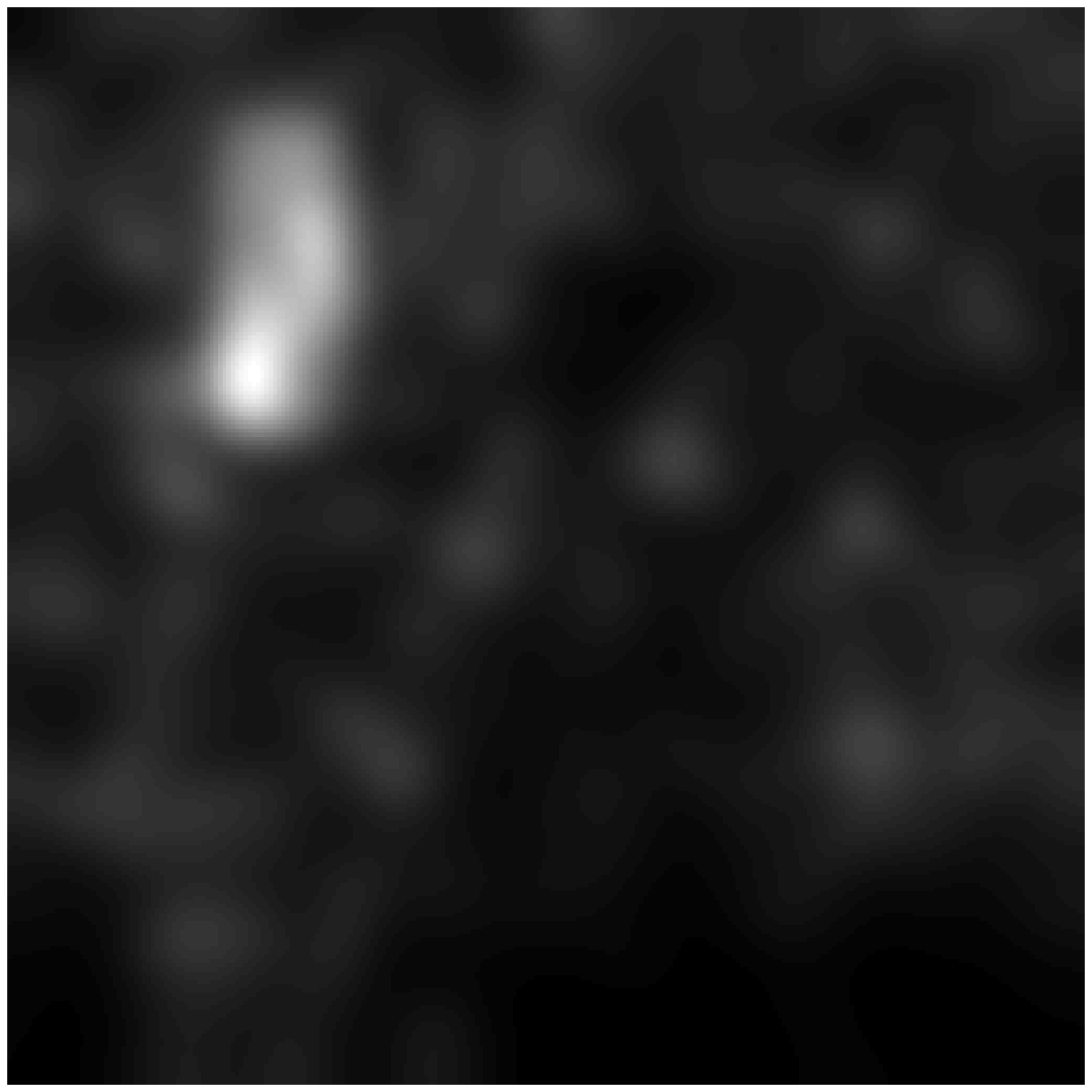}}
\end{center}    
\caption{\small \it Reconstruction of the cameraman-image corrupted by Gaussian white noise with $\sigma^2=0.0025$ except in the in (a) highlighted area where $\sigma^2=0.015$.} 
\label{fig_inhomo}
\end{figure}

\subsection{Deblurring and Gaussian noise removal}

The performance of the algorithms for restoring images corrupted by Gaussian blur with blurring kernel of size $5\times 5$ pixels and standard deviation $10$ and additive Gaussian noise with standard deviation $\sigma$ is reported in Table \ref{table_comparison_adaptive_deblurring}. Here we observe that the LATV- as well as the pLATV-algorithm outperform the SA-TV-algorithm for nearly any example. This observation is also clearly visible in Fig. \ref{fig_cam_Gb01}, where the SA-TV-algorithm produces a still blurred output. The pAPS-algorithm generates very similar reconstructions as the LATV- and pLATV-algorithm, which is also reflected by similar PSNR and MSSIM. Similarly as before, the pAPS-algorithm performs best when $\sigma=0.3$, while for smaller $\sigma$ the LATV-algorithm has always the best PSNR.

\begin{table*}
\scriptsize
\caption{\small \it PSNR- and MSSIM-values of the reconstruction of different images corrupted by Gaussian blur (blurring kernel of size $5 \times 5$ pixels with standard deviation 10) and additive Gaussian noise with standard deviation $\sigma$  via pAPS-, LATV-, pLATV-algorithm with $\alpha_0=10^{-2}$ and SA-TV-algorithm with $\lambda_0=10^{-4}$. In the LATV- and pLATV-algorithm we use $\mathcal{I}_{i,j}=\tilde \Omega_{i,j}$ with window-size $11 \times 11$ pixels in the interior and set $p_0=\frac{1}{2}$.}
\label{table_comparison_adaptive_deblurring}
\begin{center}
\begin{tabular}{c|c|c|c|c|c|c|c|c|c}
\toprule
\multicolumn{2}{c|}{}  &  \multicolumn{2}{c|}{pAPS (scalar $\alpha$)} & \multicolumn{2}{c|}{SA-TV}& \multicolumn{2}{c|}{LATV}& \multicolumn{2}{c}{pLATV} \\
Image          & $\sigma$ & PSNR    & MSSIM     & PSNR    & MSSIM    & PSNR    & MSSIM    & PSNR    & MSSIM \\ \hline
phantom     
                    & $0.3$ & $16.23$ & $0.6958$ & 15.65 & 0.6632 & {\bf 16.32} &0.6995 &  16.31 &  {\bf 0.6997} \\  
					&   $0.1$     & 17.86 & 0.7775 & 17.31 & 0.7442 & {\bf 17.98} & 0.7914 & 17.97 & {\bf 0.7923} \\
					&   $0.05$   & 18.89 & 0.7784 & 19.20 & 0.8193 & {\bf 19.49} & {\bf 0.8343} & 19.39 & 0.8279 \\ \hline
cameraman 
                    & $0.3$ & {\bf 21.04} & {\bf 0.6410} &  19.26 & 0.5990 & 20.86 & 0.6272  & 20.86 & 0.6272\\
					& $0.1$ & 23.11 & {\bf 0.7175} & 22.64 & 0.6957 & {\bf 23.17} & 0.7157 & 23.15 & 0.7156 \\
					& $0.05$ & 24.14 & 0.7562 & 23.75 & 0.7393 &  {\bf 24.22} & {\bf 0.7573} & 24.21 & 0.7570\\ \hline
barbara		
                    & $0.3$     & {\bf 20.58} & {\bf 0.4556} & 18.95 & 0.4314 & 20.42 & 0.4517 & 20.42 & 0.4515 \\
					&   $0.1$     &  {\bf 22.16} & 0.5589 & 22.09 & {\bf 0.5687} & {\bf 22.16} & 0.5597 & {\bf 22.16} & 0.5597 \\
					&   $0.05$   & 22.87 & 0.6245 & 22.88 & 0.6268 & {\bf 22.92} & {\bf 0.6273} & 22.90 & 0.6255 \\ \hline
lena     		
                    & $0.3$     & {\bf 21.75} & {\bf 0.5542} & 20.10& 0.5278 & 21.71 & 0.5529 & 21.69 & 0.5528 \\
					&   $0.1$     & 24.44 & 0.6496 & 24.39 & {\bf 0.6574} & {\bf 24.50} & 0.6514 & 24.49 & 0.6510 \\
					&   $0.05$   & 25.83 & 0.7047 & 25.81 & {\bf 0.7091} & {\bf 25.92} & 0.7066 & 25.91 & 0.7062 \\ \hline
\end{tabular}
\end{center}
\end{table*}

\graphicspath{{./graphics/}}
\begin{figure}[htbp!]
\begin{center}
\hspace{0cm}
    \subfigure[pAPS (PSNR: $23.11$; MSSIM: $0.7175$)]{\includegraphics[height=4.1cm]{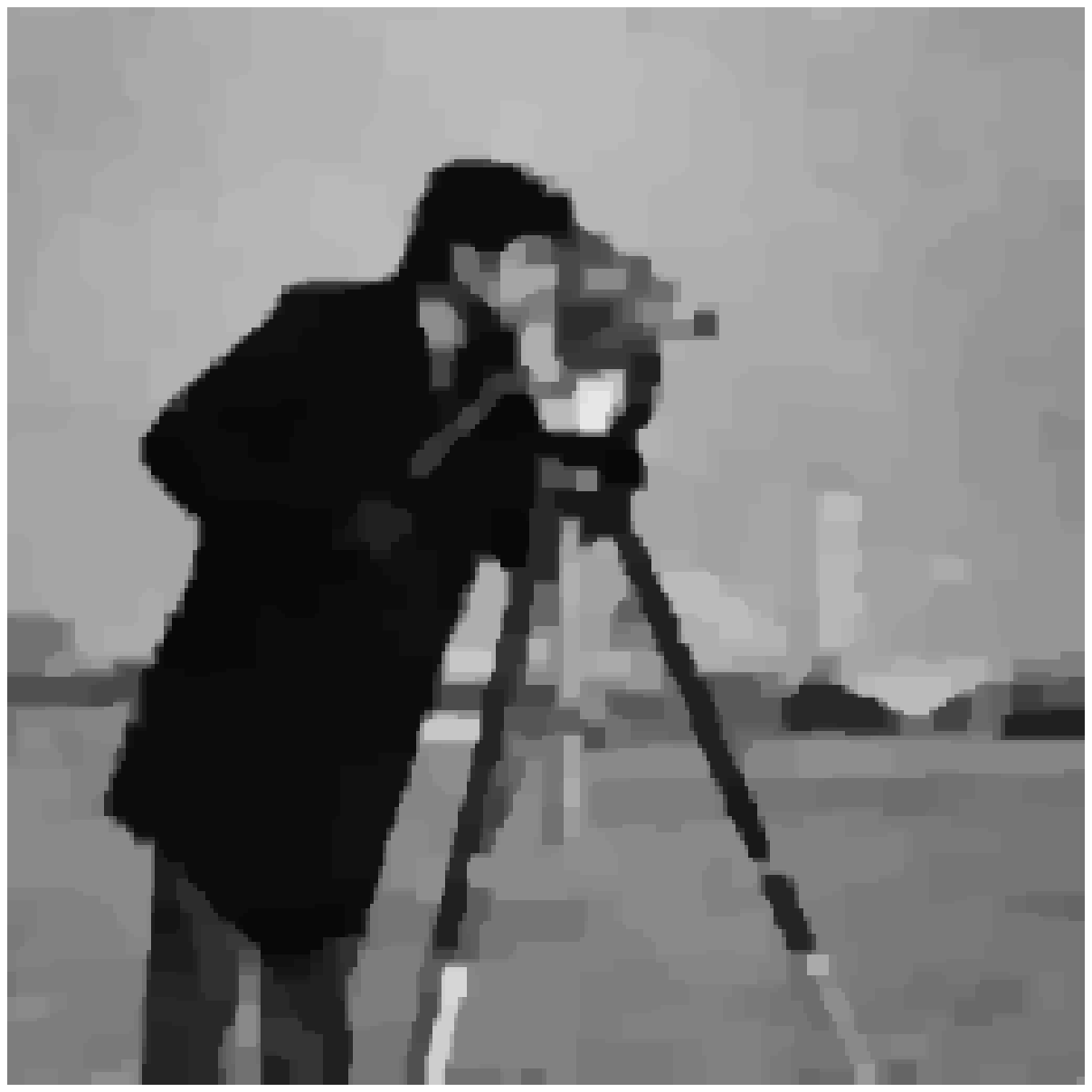}}
    \subfigure[SA-TV (PSNR: $22.64$; MSSIM: $0.6957$)]{\includegraphics[height=4.1cm]{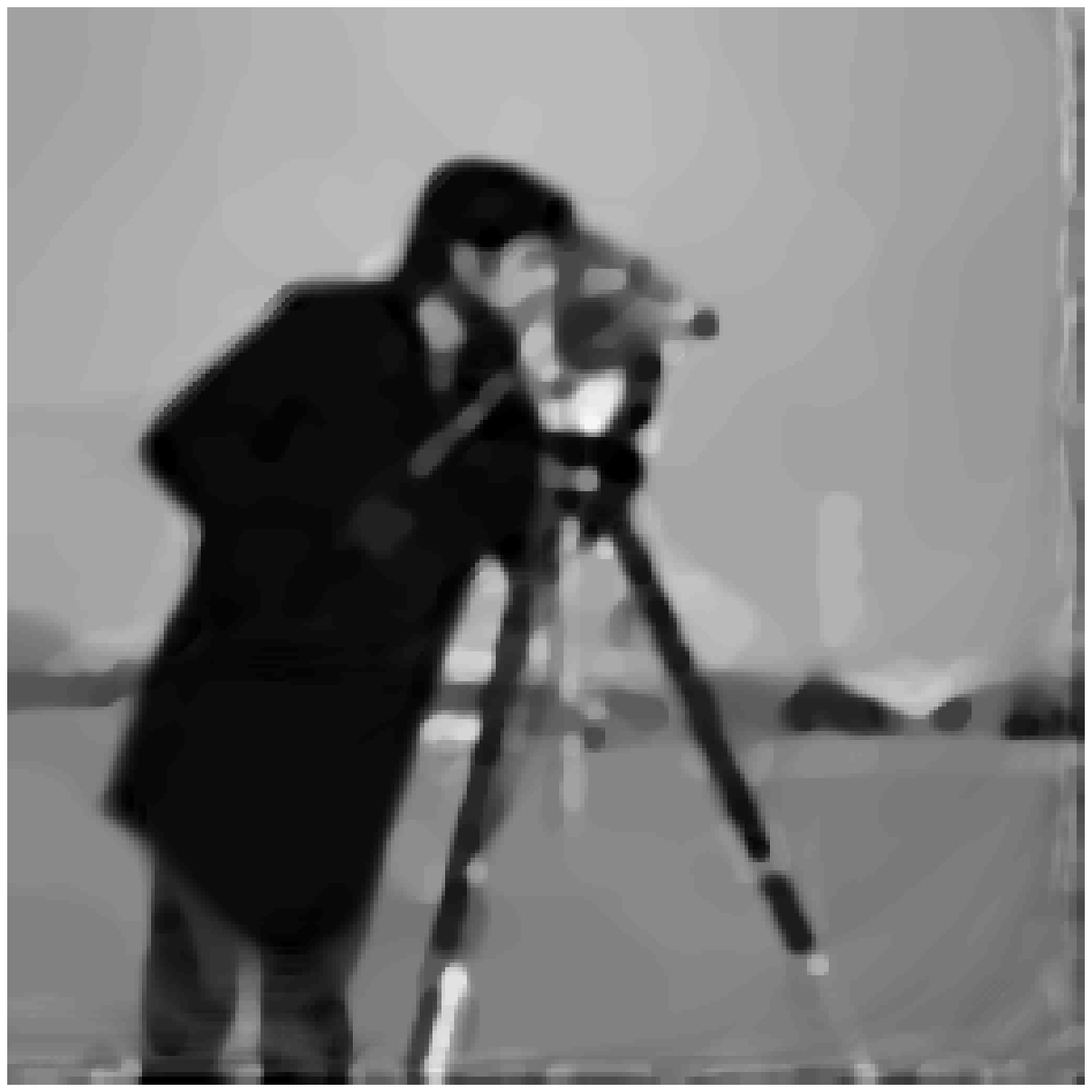}}
    \subfigure[ LATV (PSNR: $23.17$; MSSIM: $0.7160$)]{\includegraphics[height=4.1cm]{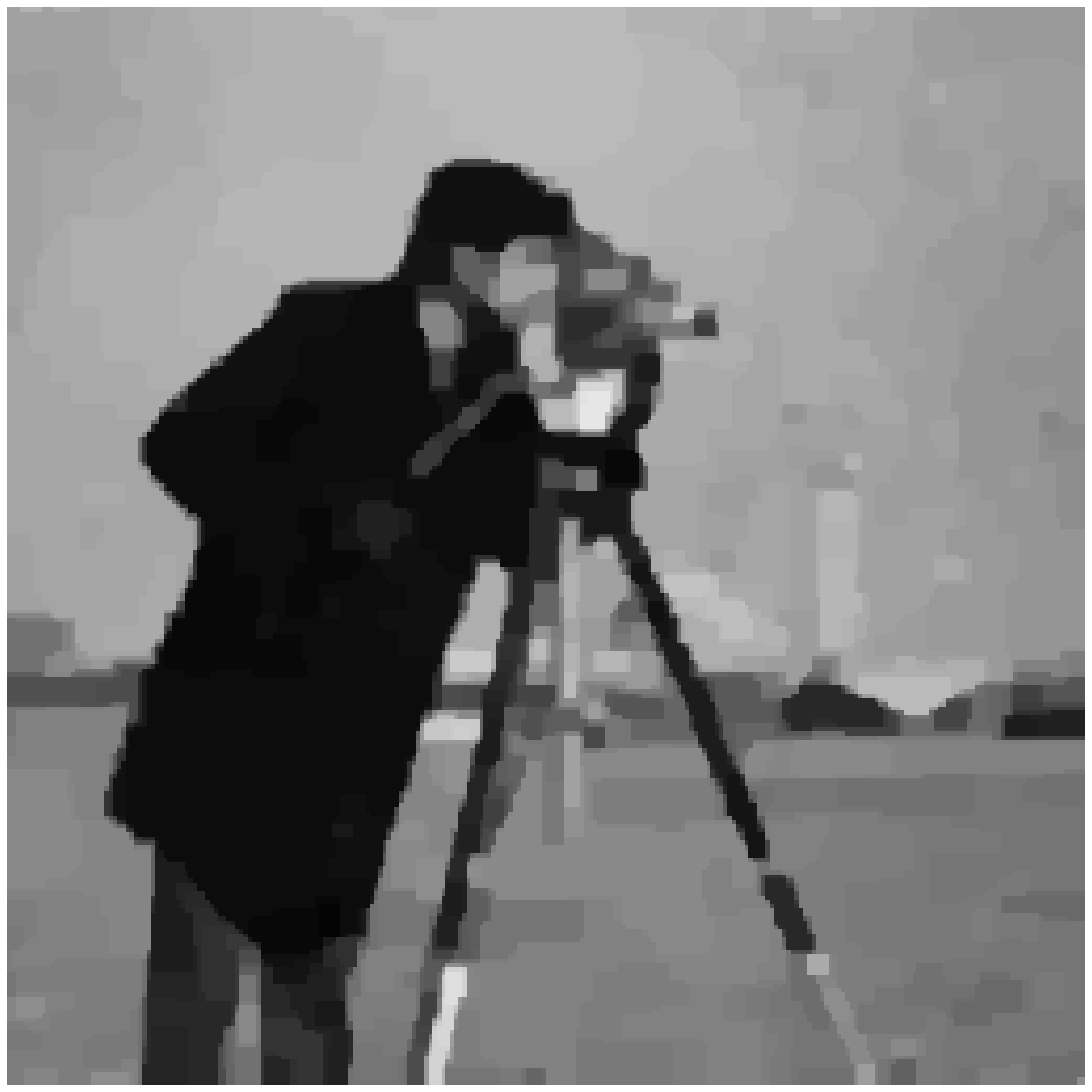}}
    \subfigure[pLATV (PSNR: $23.15$; MSSIM: $0.7160$)]{\includegraphics[height=4.1cm]{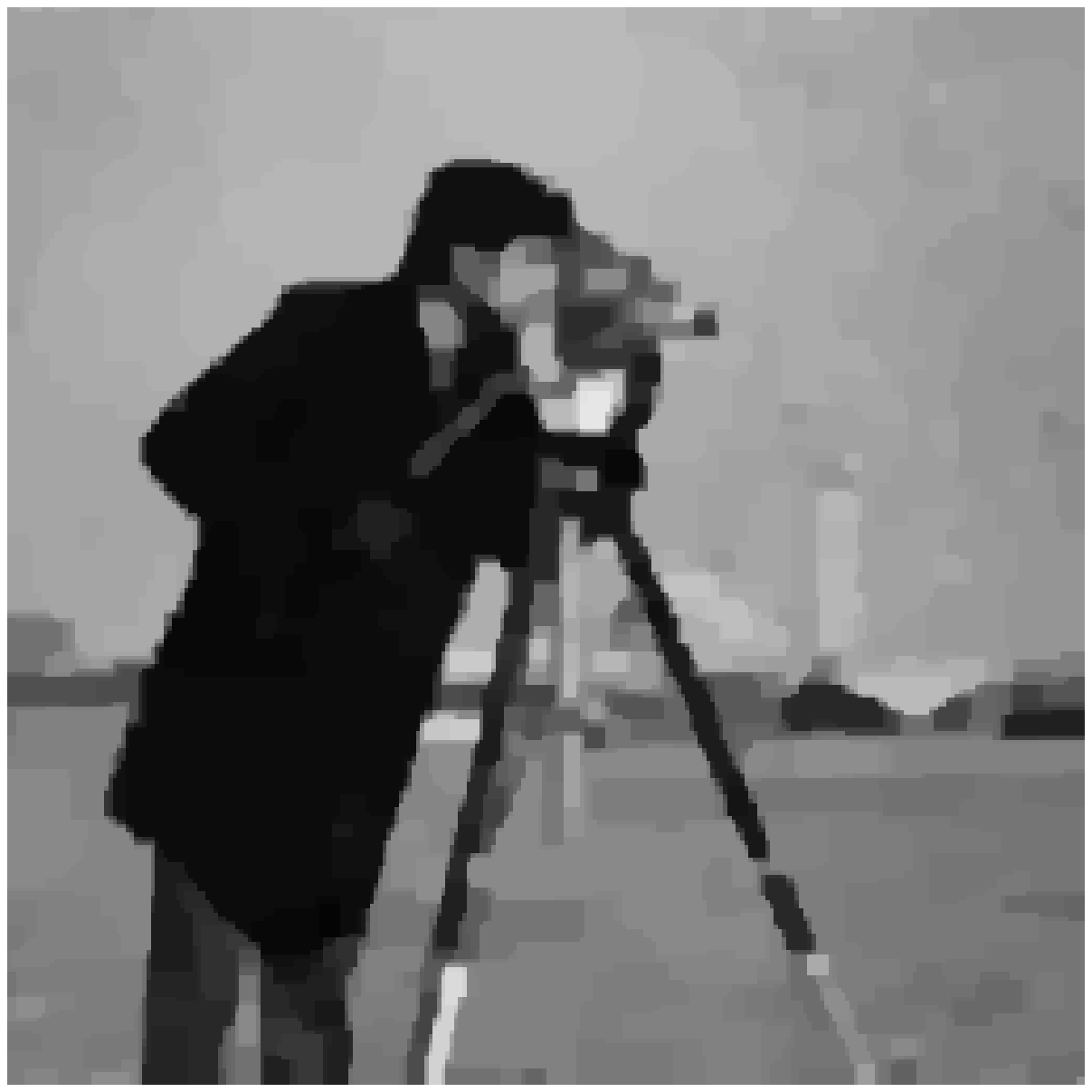}}
\end{center}    
\caption{\small \it Reconstruction of the cameraman-image corrupted by Gaussian white noise with $\sigma=0.1$ and Gaussian blur.} 
\label{fig_cam_Gb01}
\end{figure}

\subsection{Impulse noise removal} 

Since it turns out that the LATV- and pLATV-algorithm produce nearly the same output, here, we compare only our pAPS- and pLATV-algorithm for $L^1$-TV minimization, with the SA-TV method introduced in \cite{HinRin}, where a semi-smooth Newton method is used to generate an estimate of the minimizer of \eqref{LtauTVmultimodel2}. For the sake of a fair comparison an approximate solution of the minimization problem in the pLATV-algorithm is solved by the semi-smooth Newton method described in Appendix \ref{Appendix:Semismooth}. For the SA-TV method we use the parameters suggested in \cite{HinRin} and hence $\mathcal{I}_{i,j}=\Omega_{i,j}^\omega$. Moreover, we set $\lambda_0=0.2$ in our experiments which seems sufficiently small. In Table \ref{table_comparison_sap1} and Table \ref{table_comparison_rv1} we report on the results obtained by the pAPS-, SA-TV-, and pLATV-algorithm for restoring images corrupted by salt-and-pepper noise or random-valued impulse noise, respectively. While the pAPS- and pLATV-algorithm produce quite similar restorations for both type of noises, the SA-TV algorithm seems to be outperformed in most examples. For example, in Fig. \ref{fig_bar_SAP03} we observe that the pAPS- and pLATV-algorithm remove the noise considerable while the solution of the SA-TV method still contains noise. On the contrary for the removal of random-valued impulse noise in Fig. \ref{fig_bar_RV03} we see that all three methods produce similar restorations.
\begin{table*}
\scriptsize
\caption{\small \it PSNR- and MSSIM-values of the reconstruction of different images corrupted by salt-and-pepper noise with $r_1=r_2$ via pAPS-, pLATV-algorithm with $\alpha_0=10^{-2}$ and SA-TV-algorithm with $\lambda_0=0.2$ and window-size $21 \times 21$. In the pLATV-algorithm we use $\mathcal{I}_{i,j}=\tilde \Omega_{i,j}$ with window-size $11 \times 11$ pixels in the interior and we set $p_0=\frac{1}{2}$. }
\label{table_comparison_sap1}
\begin{center}
\begin{tabular}{c|c|c|c|c|c|c|c|c|c|c}
\toprule
\multicolumn{2}{c|}{}  &  \multicolumn{3}{c|}{pAPS (scalar $\alpha$)} & \multicolumn{3}{c|}{SA-TV}& \multicolumn{3}{c}{pLATV} \\
Image          & $r_1=r_2$ & PSNR    & MSSIM    & MAE  & PSNR    & MSSIM   & MAE &  PSNR    & MSSIM & MAE \\ \hline
phantom      & $0.3$ & 14.48 & 0.7040 &{\bf 0.0519} & {\bf 15.28} & 0.6540 & 0.0605  &  {\bf 14.50} & {\bf 0.7053} & {\bf 0.0519} \\
                    &   $0.1$     & 18.39 & 0.8412 & 0.0214 & {\bf 19.57} & {\bf 0.8703} & {\bf 0.0196}  &  18.63 &  { 0.8610} &  { 0.0202 }\\
					&   $0.05$   &  {21.61} &0.9257 &  {\bf 0.0103}  & {\bf 22.81} & {\bf 0.9362} & {\bf 0.0103} & 21.55 &  {\bf 0.9327} &  0.0104 \\ \hline
cameraman & $0.3$ &  {\bf 21.60} & 0.7269 & {\bf 0.0343} & 21.34 & 0.6871 & 0.0390 & 21.59 &  {\bf 0.7271} & 0.0344  \\
                    & $0.1$ & 25.49 & 0.8822 & 0.0155 & {\bf 25.80} & 0.8774 & 0.0157 & {25.57} & {\bf 0.8844} & {\bf 0.0154}  \\
					& $0.05$ & {\bf 28.80} & {\bf 0.9389} & {\bf 0.0087} & 28.50 & 0.9251 & 0.0095 & 27.99 & 0.9263 & 0.0095 \\ \hline
barbara		 & $0.3$     & {\bf 21.56} &  0.6242  &  {\bf 0.0486}& 20.54 & 0.5889 & 0.0537 &  21.51 & {\bf 0.6253}& 0.0488  \\
                    &   $0.1$     &25.49 & 0.8729 & 0.0211& 25.27 & 0.8650 & {\bf 0.0202}&  {\bf 25.81} & {\bf 0.8759} &  0.0208 \\
					&   $0.05$ & {\bf 28.46} & {\bf 0.9338} & {0.0118} & 27.90 & 0.9313 & {\bf 0.0110}    & 28.34 & 0.9325 & 0.0121 \\ \hline
lena     	   & $0.3$     &23.29 & 0.6807 & 0.0360 & 22.61 & 0.6397 & 0.0404 & {\bf 23.32} &  {\bf 0.6811} &  {\bf 0.0359} \\
                   &   $0.1$      & 27.60 & 0.8508 & 0.0151 & 27.78 & 0.8459 & 0.0152 & {\bf 27.99} &  {\bf 0.8530} &  {\bf 0.0148}  \\
					&   $0.05$   & 29.45 &  {\bf 0.8946} &  {\bf 0.0096} &{\bf 29.74} & 0.8863 & 0.0100 &  {29.53} & 0.8931  & 0.0097 \\ \hline
\end{tabular}

\end{center}
\end{table*}

\begin{table*}
\scriptsize
\caption{\small \it PSNR- and MSSIM-values of the reconstruction of different images corrupted by random-valued impulse noise via pAPS-, pLATV-algorithm with $\alpha_0=10^{-2}$ and SA-TV-algorithm with $\lambda_0=0.2$ and window-size $21 \times 21$. In the pLATV-algorithm we use $\mathcal{I}_{i,j}=\tilde \Omega_{i,j}$ with window-size $11 \times 11$ pixels in the interior and we set $p_0=\frac{1}{2}$ in the pLATV-algorithm. }
\label{table_comparison_rv1}
\begin{center}
\begin{tabular}{c|c|c|c|c|c|c|c|c|c|c}
\toprule
\multicolumn{2}{c|}{}  &  \multicolumn{3}{c|}{pAPS (scalar $\alpha$)} & \multicolumn{3}{c|}{SA-TV}& \multicolumn{3}{c}{pLATV} \\
Image          & $r$ & PSNR    & MSSIM    & MAE  & PSNR    & MSSIM    & MAE   &  PSNR    & MSSIM & MAE \\ \hline
phantom      & $0.3$ & 17.83 & 0.8120 &{0.0317} &  {\bf 18.68} & 0.8012 & 0.0319 &  { 18.19} & {\bf 0.8303} & {\bf 0.0305} \\
                    &   $0.1$     & 22.46 & 0.9273 & 0.0113 & {\bf 23.83} & 0.9278 & {\bf 0.0100} &  {22.58} &  {\bf 0.9328} &  {0.0112 }\\
					&   $0.05$   &  {25.55} & 0.9636 & {0.0058} & {\bf 26.56} & 0.9642 & {\bf  0.0054}  & 25.45 &  {\bf 0.9665} &  {0.0057} \\ \hline
cameraman & $0.3$ &  {\bf 24.87} & {\bf 0.8337} & {\bf 0.0213} & 23.48 & 0.7583 & 0.0237 & 24.19 &  {0.7887} & 0.0234  \\
                    & $0.1$ & {\bf 29.33} &{\bf 0.9359} &{\bf 0.0087} & 27.72 & 0.9087 & 0.0089 & { 28.60} & { 0.9204} & { 0.0093}  \\
					& $0.05$ & {\bf 31.46} & {\bf 0.9603} & 0.0053 & 30.53 & 0.9478 &  {\bf 0.0052}& 30.84 & 0.9442 & 0.0058 \\ \hline			
barbara		& $0.3$     & {\bf 24.24} &  {\bf 0.8040} &0.0301 & 23.96 & 0.7977  &  {\bf 0.0280} &  {\bf 24.24}  &0.7992 &  { 0.0302} \\
                    &   $0.1$     & {\bf 29.20} &  {\bf 0.9355} &  {0.0118} & 28.60 & 0.9327 &  {\bf 0.0101}  & 28.91 & 0.9305 &  0.0120 \\
					&   $0.05$   & {\bf 31.95} & {\bf 0.9650} & {0.0065} & 30.65 & 0.9578 &  {\bf 0.0059}  & { 31.85} & { 0.9640} & {0.0066} \\ \hline
lena     		& $0.3$     &{\bf 26.80} & {\bf 0.8124} & {\bf 0.0205} & 24.74 & 0.7560 & 0.0236 &  {26.63} &  {0.8082} &  { 0.0208} \\
                    &   $0.1$      &  {\bf 30.34} & {\bf 0.8965} & {\bf 0.0092} & 28.96 & 0.8833& 0.0095 &  30.07 &  0.8918 &  {0.0095}  \\
					&   $0.05$   & {\bf 31.36} &  {\bf 0.9189} &  0.0062& 30.42 & 0.9180 & {\bf 0.0059} &  {31.08} & 0.9159  & 0.0063 \\ \hline
\end{tabular}

\end{center}
\end{table*}

\graphicspath{{./graphics/}}
\begin{figure}[htbp!]
\begin{center}
\hspace{0cm}
\subfigure[noisy image]{\includegraphics[height=4.1cm]{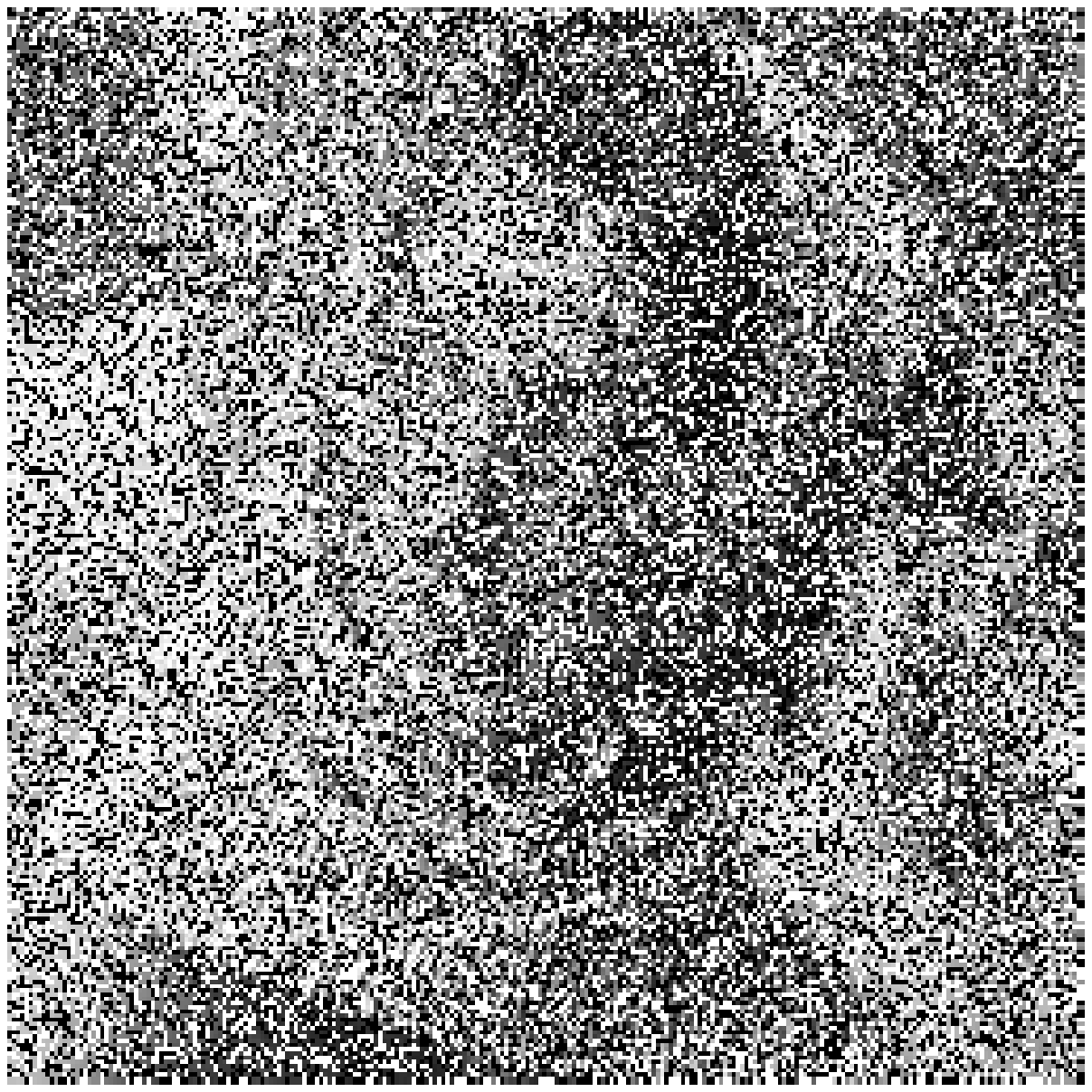}}
    \subfigure[pAPS (PSNR: $21.56$; MSSIM: $0.6242$)]{\includegraphics[height=4.1cm]{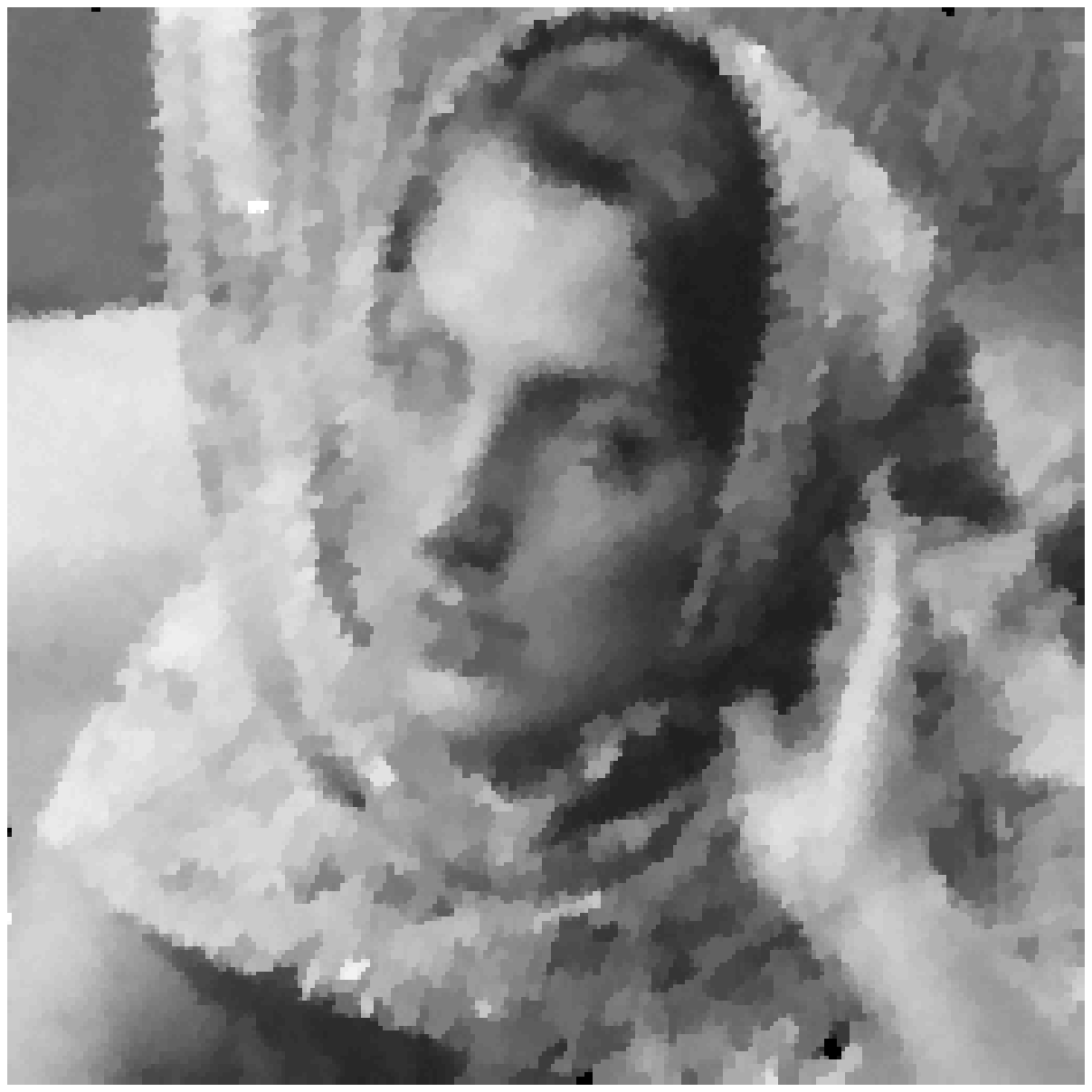}}
    \subfigure[SA-TV (PSNR: $20.54$; MSSIM: $0.5889$)]{\includegraphics[height=4.1cm]{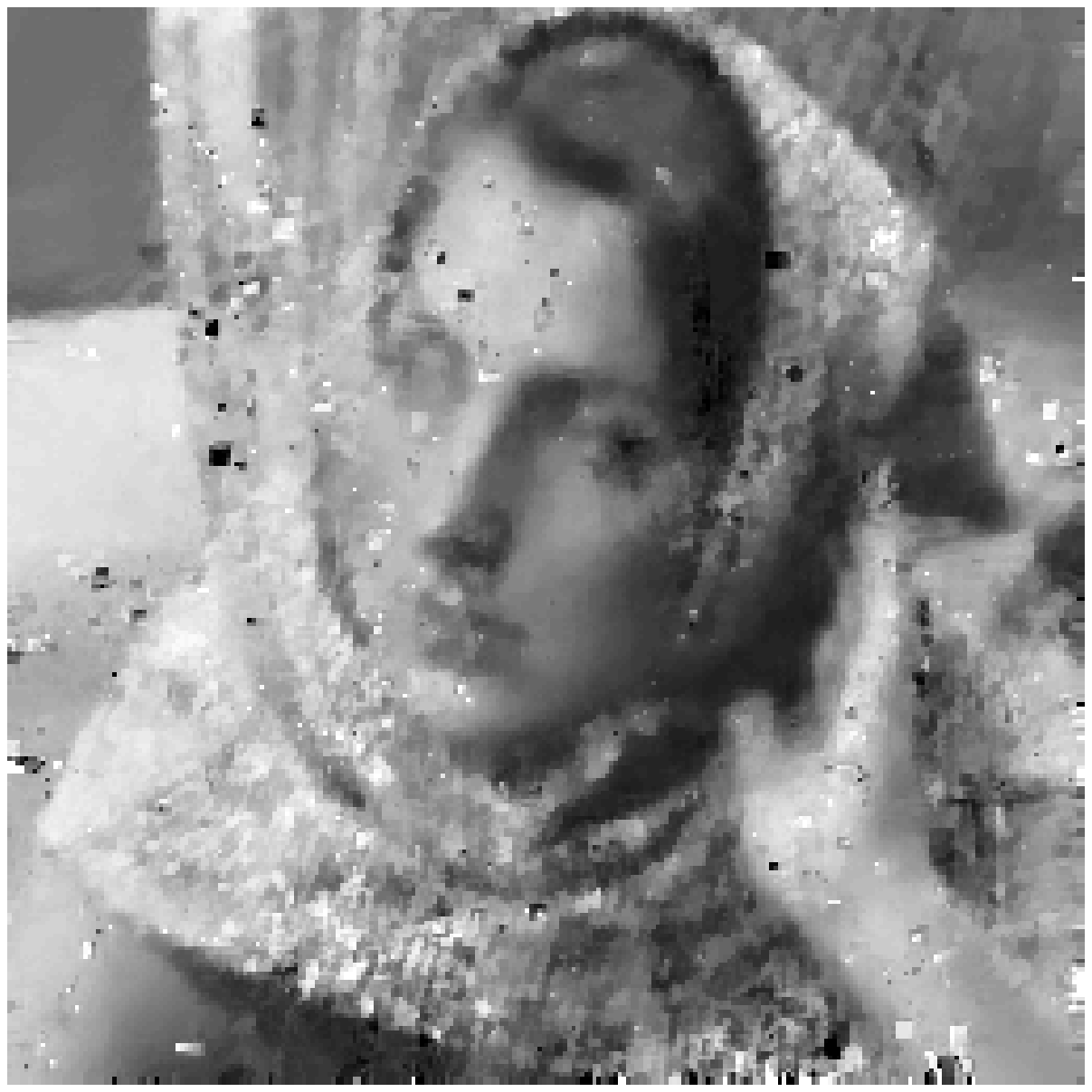}}
    \subfigure[pLATV (PSNR: $21.51$; MSSIM: $0.6253$)]{\includegraphics[height=4.1cm]{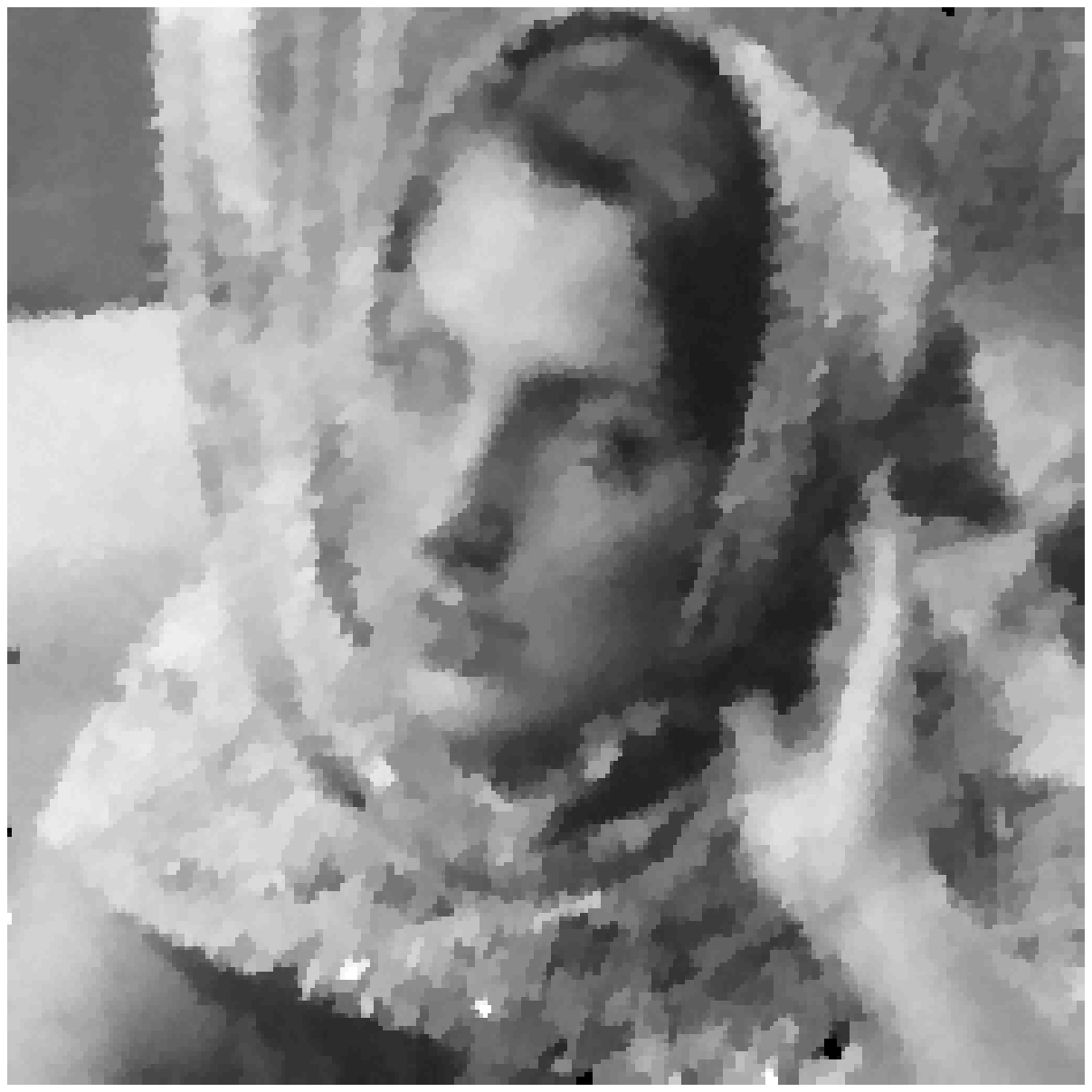}}
\end{center}    
\caption{\small \it Reconstruction of the barbara-image corrupted by salt-and-pepper noise with $r_1=r_2=0.3$.} 
\label{fig_bar_SAP03}
\end{figure}

\graphicspath{{./graphics/}}
\begin{figure}[htbp!]
\begin{center}
\hspace{0cm}
\subfigure[noisy image]{\includegraphics[height=4.1cm]{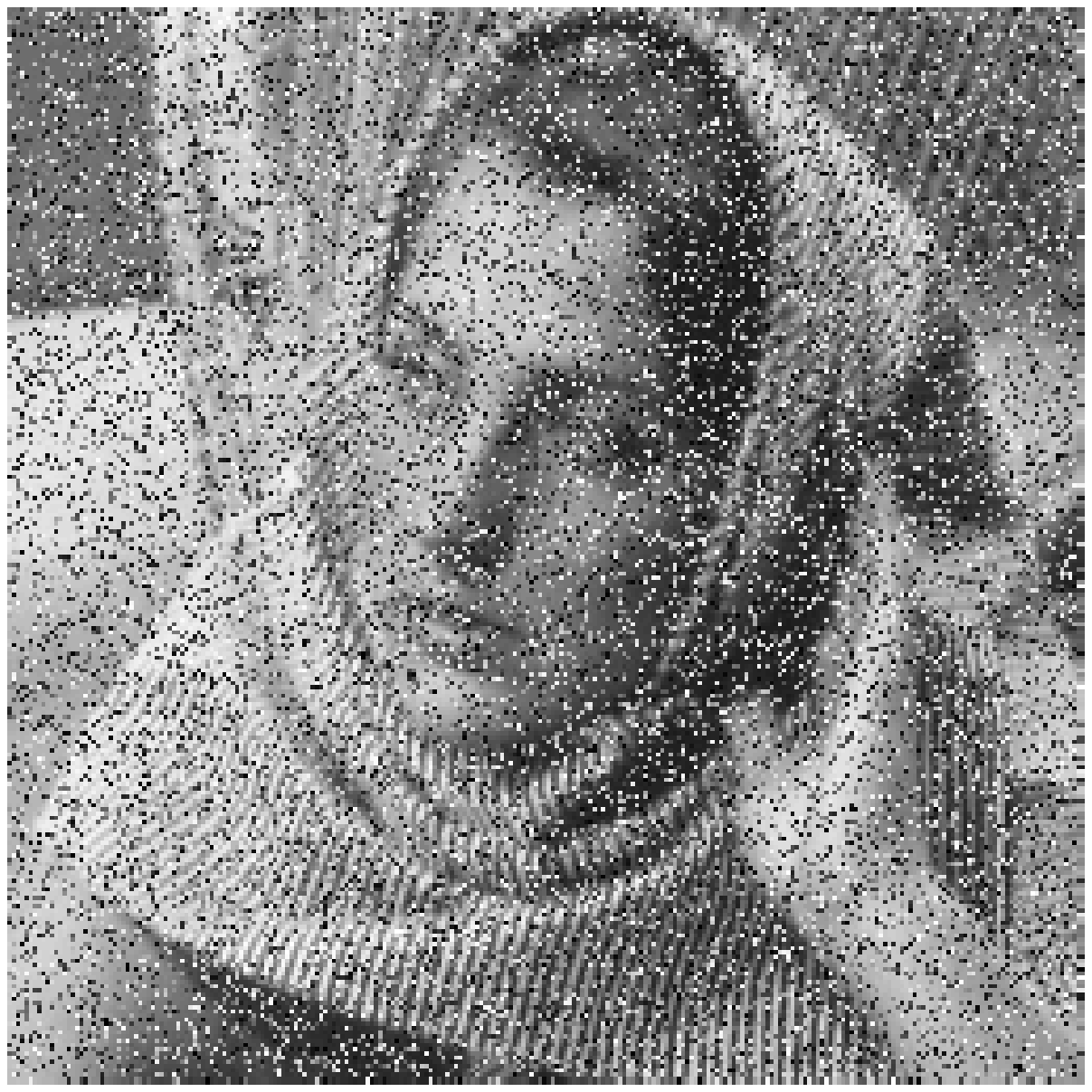}}
    \subfigure[pAPS (PSNR: $24.24$; MSSIM: $0.8040$)]{\includegraphics[height=4.1cm]{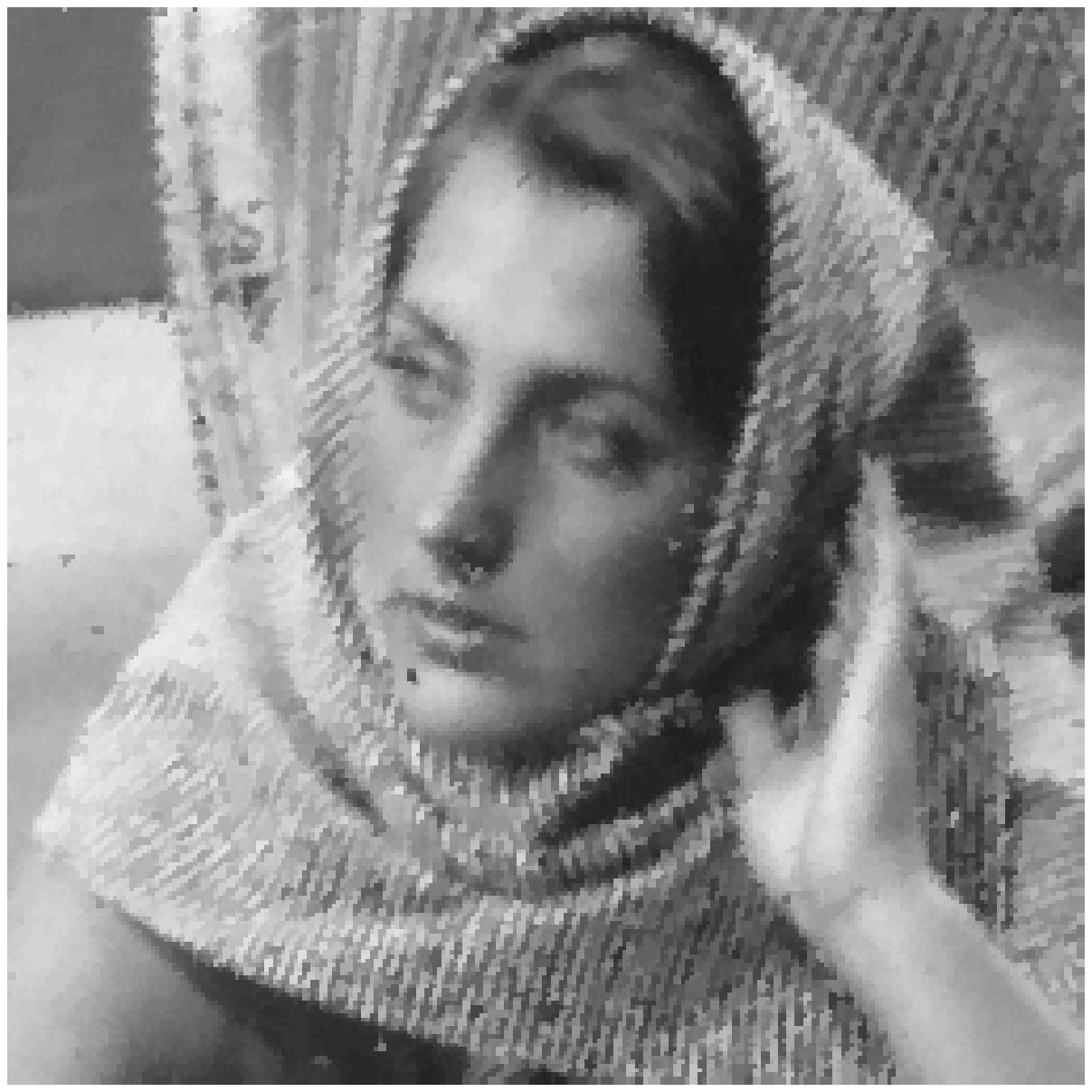}}
    \subfigure[SA-TV (PSNR: $23.96$; MSSIM: $0.7977$)]{\includegraphics[height=4.1cm]{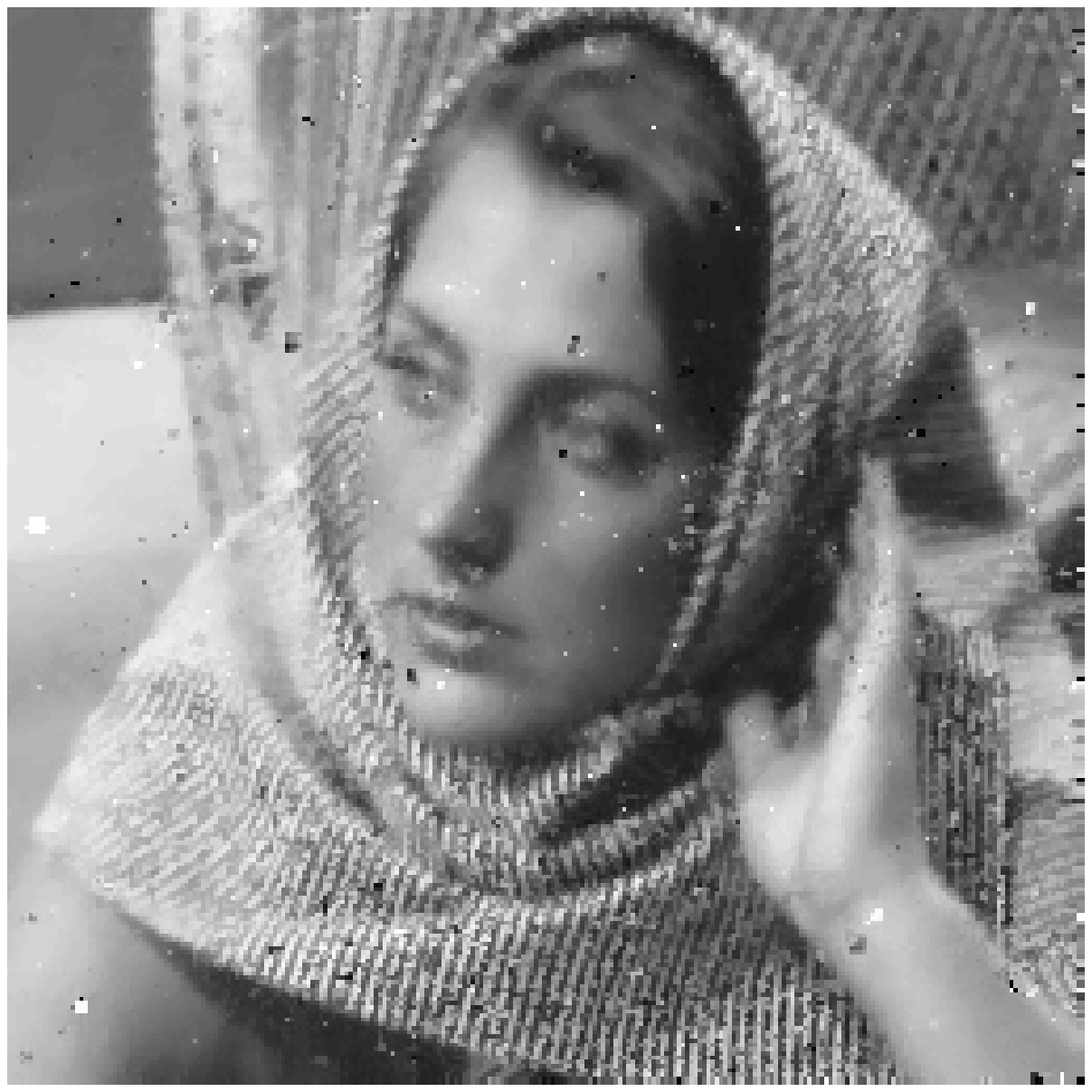}}
    \subfigure[pLATV (PSNR: $24.24$; MSSIM: $0.7992$)]{\includegraphics[height=4.1cm]{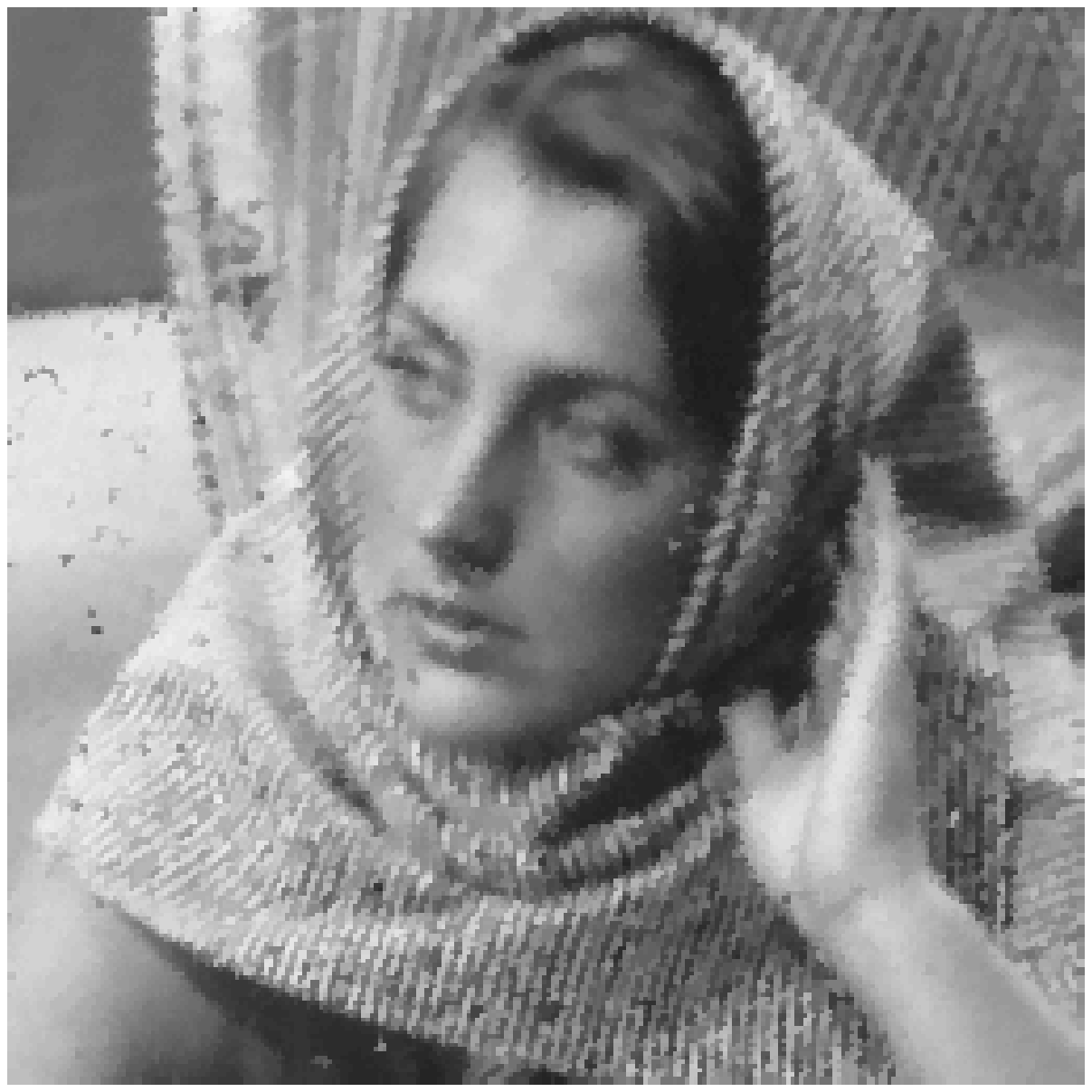}}
\end{center}    
\caption{\small \it Reconstruction of the barbara-image corrupted by random-valued impulse noise with $r=0.3$.} 
\label{fig_bar_RV03}
\end{figure} 

\section{Conclusion and extensions}\label{Sec:Conclusion}

For $L^1$-TV and $L^2$-TV minimization including convolution type of problems automatic parameter selection algorithms for scalar and locally dependent weights $\alpha$ are presented. In particular, we introduce the APS- and pAPS-algorithm for automatically determining a suitable scalar regularization parameter. While for the APS-algorithm its convergence only under some assumptions is shown, the pAPS-algorithm is guaranteed to converge always. Besides the general applicability of these two algorithms they also possess a fast numerical convergence in practice.

In order to treat homogeneous regions differently than fine features in images, which promises a better reconstruction, cf. Proposition \ref{Prop:TVlvsTVs} and Remark \ref{Remark:TVlvsTVs}, algorithms for automatically computing locally adapted weights $\alpha$ are proposed. These methods are much more stable with respect to the initial $\alpha_0$ than the state-of-the-art SA-TV method. Moreover, while in the SA-TV-algorithm the initial $\lambda_0>0$ has to be chosen sufficiently small, in our proposed methods any arbitrary $\alpha_0>0$ is allowed. Hence the LATV- and pLATV-algorithm are much more flexible with respect to the initialization. By numerical experiments it is shown that the reconstructions obtained by the newly introduced algorithms are similar with respect to image quality measure to the restorations obtained by the SA-TV algorithm. In the case of Gaussian noise removal (including deblurring) for sufficiently small noise-levels reconstructions obtained by locally varying weights seem to be qualitatively better than results with scalar parameters. On the contrary, for removing impulse noise a spatially varying $\alpha$ or $\lambda$ is in general not always improving the restoration quality.

For computing a minimizer of the respective multi-scale total variation model we present first and second order methods and show their convergence to a respective minimizer.

Although the proposed parameter selection algorithms are constructed to estimate the parameter $\alpha$ in \eqref{LtauTVmodel} and \eqref{LtauTVmultimodel}, they can be easily adjusted to find a good candidate $\lambda$ in \eqref{LtauTVmodel2} and \eqref{LtauTVmultimodel2}, respectively, as well. 



Note, that the proposed parameter selection algorithms are not restricted to total variation minimization, but may be extended to other type of regularizers as well by imposing respective assumptions that guarantee a minimizer of the considered optimization problem. In order to obtain similar (convergence) results as presented in Section \ref{Sec:APS} and Section \ref{Sec:locallyTV} the considered regularizer should be convex, lower semi-continuous and one-homogeneous. In particular for proving convergence results as in Section \ref{Sec:APS} (cf. Theorem \ref{Theorem:conv2} and Theorem \ref{Theorem:conv}) an equivalence of the penalized and corresponding constrained minimization problem, as in Theorem \ref{Thm:equi}, is needed. 
 An example of such a regularizer for which the presented algorithms are extendable is the total generalized variation \cite{BreKunPoc}.

\begin{acknowledgements}
The author would like to thank M. Monserrat Rincon-Camacho for providing the spatially adaptive parameter selection code for the $L^1$-TV model of \cite{HinRin}.
\end{acknowledgements}

\begin{appendix}

\section{Semi-smooth Newton method for solving \eqref{Eq:Op:cond:pd}}\label{Appendix:Semismooth}

A semi-smooth Newton algorithm for solving \eqref{Eq:Op:cond:pd} can be derived similar as in \cite{HinRin} by means of vector-valued variables. Therefore let $u^h\in \R^N$, $q^h\in \R^{2N}$, $\alpha^h\in \R^N$, $g^h\in \R^N$ where $N=N_1 N_2$, denote the discrete image intensity, the dual variable, the spatially dependent regularization parameter, and the observed data vector, respectively. Correspondingly we define $\nabla^h \in \R^{2N \times N}$ as the discrete gradient operator, $\Delta^h \in \R^{N\times N}$ as the discrete Laplace operator, $T^h\in \R^{N\times N}$ as the discrete operator, and $(T^h)^t$ as the transpose of $T^h$. Here $|\cdot|$, $\max\{\cdot, \cdot\}$, and $\sign(\cdot)$ are understood for vectors in a componentwise sense. Moreover, we use the function $[|\cdot|] : \R^{2N} \to \R^{2N}$ with $[|v^h|]_i =[|v^h|]_{i+N} = \sqrt{(v^h_i)^2 + (v^h_{i+N})^2}$ for $1\leq i \leq N$.

For solving \eqref{Eq:Op:cond:pd} in every step of our Newton method we need to solve 
\begin{equation}\label{Newton:system}
\begin{split}
\left(\begin{matrix}
A_k^h & -D^h(m_{\beta_k}) & 0 \\
-\frac{1}{\beta	+ \mu} (T^h)^t T^h + \kappa \Delta^h  & -\frac{\mu}{\mu + \beta} (T^h)^t & -(\nabla^h)^t \\
B_k^h & 0 & D^h(m_{\gamma_k})
\end{matrix}
\right)
\left(\begin{matrix}
\delta_u \\
\delta_v \\
\delta_q
\end{matrix}
\right) \\
=
\left(\begin{matrix}
-\mathfrak F_1^k \\
-\mathfrak F_2^k \\
-\mathfrak F_3^k 
\end{matrix}
\right)
\end{split}
\end{equation}
where 
\begin{equation*}
\begin{split}
&A_k^h = [D^h(e_N) - D^h(v_k^h) \chi_{\mathcal{A}_{\beta_k}} D^h(\sign(T^h u_k^h - g^h))] T^h, \\
&B_k^h = [ D^h(q_k^h) \chi_{\mathcal{A}_{\gamma_k}} D^h(m_{\gamma_k})^{-1} M^h(\nabla^h u_k^h) \\
&\phantom{\mathfrak{F}_2^k = -(\nabla^h)^t q_k^h +}-D^h((\alpha^h,\alpha^h)^t)] \nabla^h , \\
&\mathfrak{F}_1^k = Tu_k^h - g^h - D^h(m_{\beta_k})v_k^h, \\
&\mathfrak{F}_2^k = -(\nabla^h)^t q_k^h + \kappa \Delta^h u_k^h -\frac{1}{\beta	+ \mu} (T^h)^t (T^h u_k^h -g^h) \\
& \phantom{\mathfrak{F}_2^k = -(\nabla^h)^t q_k^h +}-\frac{\mu}{\mu + \beta} (T^h)^t v_k^h,\\
&\mathfrak{F}_3^k = -D^h((\alpha^h,\alpha^h)^t) \nabla^h u_k^h + D^h(m_{\gamma_k})q_k^h,
\end{split}
\end{equation*}
$e_N \in \R^{N}$ is the identity vector, $D^h(v)$ is a diagonal matrix with the vector $v$ in its diagonal, $m_{\beta_k} = \max\{\beta, |T^h u_k^h - g^h| \}$, $m_{\gamma_k} = \max\{\gamma \alpha^h,[ |\nabla^h u_k^h|]\}$,
\begin{equation*}
\begin{split}
& \chi_{\mathcal{A}_{\beta_k}}  = D^h(t_{\beta_k}) \quad \text{with} \ \ (t_{\beta_k})_i = \begin{cases}
 0 & \text{if } (m_{\beta_k})_i = \beta, \\
 1 & \text{else};
 \end{cases}
 \\
&\chi_{\mathcal{A}_{\gamma_k}}  = D^h(t_{\gamma_k}) \quad \text{with} \ \ (t_{\gamma_k})_i = \begin{cases}
 0 & \text{if } (m_{\gamma_k})_i = \gamma (\alpha^h)_i, \\
 1 & \text{else};
 \end{cases} 
 \\
 &M^h(v)=
 \left(
 \begin{matrix}
 D^h(v_x) & D^h(v_y) \\
 D^h(v_x) & D^h(v_y) 
 \end{matrix}
 \right) \quad \text{with} \ \ v=(v_x,v_y)^t \in \R^{2N}.
\end{split}
\end{equation*}
Since the diagonal matrices $D^h(m_{\beta_k}) $ and $D^h(m_{\gamma_k}) $ are invertible, we eliminate $\delta_v$ and $\delta_q$ from \eqref{Newton:system}, which leads to the following resulting system
\begin{equation*}
H_k \delta_u = f_k
\end{equation*}
where
\begin{equation*}
\begin{split}
&H_k:= \frac{1}{\beta	+ \mu} (T^h)^t T^h - \kappa \Delta^h + \frac{\mu}{\mu + \beta} (T^h)^t D^h(m_{\beta_k})^{-1} A_k^h \\
&\phantom{H_k:= \frac{1}{\beta	+ \mu} (T^h)^t T^h -}+ (\nabla^h)^t D^h(m_{\gamma_k})^{-1} (-B_k^h) ,\\
&f_k:= \mathfrak{F}_2^k - \frac{\mu}{\mu + \beta} (T^h)^t D^h(m_{\beta_k})^{-1} \mathfrak{F}_1^k  \\
&\phantom{H_k:= \frac{1}{\beta	+ \mu} (T^h)^t T^h -}+ (\nabla^h)^*D^h(m_{\gamma_k})^{-1} \mathfrak{F}_2^k .
\end{split}
\end{equation*}
In general $B_k^h$ and hence $H_k$ is not symmetric. In \cite{HinSta} it is shown that the matrix $H_k^h$ at the solution $(u_k^h,v_k^h,q_k^h) = (\bar{u},\bar{v},\bar{q})$ is positive definite whenever
\begin{equation}\label{Eq:conditions}
[|q_k^h|]_i \leq (\alpha^h)_i \quad \text{and} \quad (|v_k^h|)_i \leq 1
\end{equation}
for $i=1,\ldots, N$.

In case these two inequalities are not satisfied we project $q_k^h$ and $v_k^h$ onto their feasible set, i.e., $((q_k^h)_i,(q_k^h)_{i+N})$ is set to $(\alpha^h)_i \max\{(\alpha^h)_i, [|q_k^h|]_i\}^{-1} ( (q_k^h)_i, (q_k^h)_{i+N})$ and $(v_k^h)_i$ is replaced by $\max\{1, (|v_k^h|)_i\} (v_k^h)_i$. Then the modified system matrix, denoted by $H_k^+$ is positive definite; see \cite{DonHinNer}. As pointed out in \cite{HinRin} we may use $H_k^+ + \varepsilon_k D^h(e_N)$ with $\kappa=0$, $\varepsilon_k>0$ and $\varepsilon_k \downarrow 0$ as $k\to \infty$ instead of $\kappa>0$ to obtain a positive definite matrix.  Then our semi-smooth Newton solver may be written as in \cite{HinRin}:

\noindent
\fbox{
\begin{minipage}{7.9cm}
\textbf{Semi-smooth Newton method:} Initialize $(u_0^h,q_0^h) \in \R^{N} \times \R^{2N}$ and set $k:=0$.
\begin{enumerate}
\item Determine the active sets $\chi_{\mathcal{A}_{\beta_k}} \in \R^{N\times N}$ and $\chi_{\mathcal{A}_{\gamma_k}}\in \R^{2N\times 2N}$
\item If \eqref{Eq:conditions} is not satisfied, then compute $H_k^+$; otherwise set  $H_k^+:=H_k$.
\item Solve $H_k^+ \delta_u = f_k$ for $\delta_u$.
\item Compute $\delta_q$ by using $\delta_u$.
\item Update $u_{k+1}^h:=u_{k}^h + \delta_u$ and $q_{k+1}^h:=q_{k}^h + \delta_q$.
\item Stop or set $k:=k+1$ and continue with step 1).
\end{enumerate}
\end{minipage}
}\\

This algorithm converges at a superlinear rate, which follows from standard theory; see \cite{HinKun,HinSta}.

In our experiments we always choose $\kappa=0$, $\beta=10^{-3}$, $\gamma=10^{-2}$, and $\mu=10^{6}$.
\end{appendix}


\bibliographystyle{spmpsci}      
\bibliography{Ref}{}


\end{document}